\documentclass{IEEEtran}
 
\usepackage{amsmath,amsthm,amssymb,bm}
\usepackage{algpseudocode}
\usepackage{listings}
\usepackage{color} 
\usepackage{graphicx}
\usepackage{bm}
\usepackage{mathtools}
\usepackage{subfigure}
\usepackage{hyperref}
\usepackage{algorithm}
\usepackage{dsfont}
\usepackage{enumitem}
\usepackage{amsmath}
\usepackage{algpseudocode}
\usepackage{bbm}
\usepackage{dsfont}
\usepackage{epstopdf}
\usepackage{graphicx,subfigure}
\usepackage[labelfont=bf]{caption}
\usepackage[export]{adjustbox}
\hypersetup{
    colorlinks=true,
    linkcolor=[RGB]{165,42,42},
    citecolor=[RGB]{25,25,112},
    filecolor=magenta,      
    urlcolor=cyan,
}
\definecolor{mygreen}{RGB}{28,172,0} 
\definecolor{mylilas}{RGB}{170,55,241}

\newtheorem{theorem}{\bf Theorem}

\newtheorem{lemma}[theorem]{\bf Lemma}

\theoremstyle{definition}

\newtheorem{remark}{\normalfont \textit{Remark}}

\usepackage{etoolbox}
\makeatletter
\patchcmd{\math@cr@@@align}
  {\place@tag}
  {\bgroup\color[RGB]{165,42,42}\place@tag\egroup}
  {}{}
\makeatother

\theoremstyle{assumption}
\newtheorem{assumption}{\bf Assumption}

\usepackage{filecontents}
\usepackage{mathrsfs}

\newcommand{\expect}{{\rm I\!E}}

\newcommand{\real}{{\rm I\!R}}
\newcommand{\Prob}{{\rm I\!P}}
\newcommand{\ball}{{\rm I\!B}}
\def\D{\mathrm{d}}

\begin{document}

 
\title{Distributed Regularized Primal-Dual Method: Convergence Analysis and Trade-offs  
\author{Masoud Badiei Khuzani, Na Li \\
\small {John A. Paulson School of Engineering and Applied Sciences} \\ 
\small {Harvard University}\\
\{mbadieik,nali\}@seas.harvard.edu.
}}
\date{}
\maketitle
\begin{abstract}
We study deterministic and stochastic primal-dual sub-gradient algorithms for distributed optimization of a separable objective function with global inequality constraints. In both algorithms, the norm of the Lagrangian multipliers are controlled by augmenting the corresponding Lagrangian function with a quadratic regularization term. Specifically, we show that as long as the stepsize of each algorithm satisfies a certain restriction, the norm of the Lagrangian multipliers is upper bounded by an expression that is inversely proportional to the parameter of the regularization. We use this result to compute upper bounds on the sub-gradients of the Lagrangian function. For the deterministic algorithm, we prove a convergence rate for attaining the optimal objective value. In the stochastic optimization case, we similarly prove convergence rates both in the expectation and with a high probability, using the method of bounded martingale difference. For both algorithms, we demonstrate a trade-off between the convergence rate and the decay rate of the constraint violation, in the sense that improving the convergence rate slows the decay rate of the constraint violation and vice versa. We demonstrate the convergence of our proposed algorithms numerically for distributed regression with the hinge and logistic loss functions over different graph structures.
\end{abstract}


 \setcounter{equation}{0}

\section{Introduction}
Recent advances in networked systems such as sensor networks as well as the increasing need
for solving high dimensional problems more efficiently have stimulated a significant interest in distributed optimization methods. In the distributed optimization approach, each node of a network solves a sub-problem locally based on information it sends and receives from its neighborhood. Distributed optimization has many applications, such as trajectory optimization for formation control of vehicles \cite{jadbabaie2003coordination, stipanovic2004decentralized, fax2004information}, decentralized control of power systems \cite{bakirtzis2003decentralized}, packet routing \cite{stern1977class}, and estimation problems in sensor networks \cite{ogren2004cooperative}. 

In this paper, we propose and analyze a distributed optimization method for the convex optimization problems of the following form
\begin{align}
\label{Eq:emperical_11}
\min_{x\in \mathcal{X}}   f(x)\coloneqq  \dfrac{1}{n}\sum_{i=1}^{n} f_{i}(x),
 \end{align}\normalsize
where $f_{i}(\cdot),i=1,2,\cdots,n$ are convex functions. Further, $\mathcal{X}\subset \real^{d}$ is a non-empty, convex, compact set that is characterized by a set of inequality constraints
\begin{align}
\label{Eq:emperical_22}
\mathcal{X}\coloneqq \{x\in \real^{d}:g_{k}(x)\leq 0,k=1,2,\cdots,m\},
 \end{align}\normalsize
where $g_{k}:\real^{d}\rightarrow \real$ are convex functions for all $k=1,2,\cdots,m$. 

More specifically, we propose distributed deterministic and stochastic primal-dual algorithms for the optimization problem in eqs. \eqref{Eq:emperical_11}-\eqref{Eq:emperical_22}. At each step of the distributed algorithms, the primal variables are projected onto the Euclidean ball centered at the origin that contains the feasible set $\mathcal{X}$, that is $\mathcal{X}\subseteq \ball_{d}(R)\coloneqq\{x\in \real^{d}:\|x\|_{2}\leq R\}$. Since the projection onto the Euclidean ball has a closed form expression, each step of the distributed algorithm is computed efficiently.

\subsection{Contributions}

We prove a convergence rate for the distributed deterministic and stochastic primal-dual algorithm under the Lipschitz continuity assumption on the objective function and the inequality constraints. In particular, we prove convergence rates for achieving the optimal value of the objective function. We also prove two constraint violation bounds for the primal-dual algorithm. In particular, we show that when one of the inequality constraints is binding at the optimal point(s), there is a trade-off in the convergence rate and the constraint violation rate. In particular, improving the convergence rate deteriorates the constraint violation rate and vice versa. Interestingly, we show that such a trade-off does not exist if the constraints are strictly feasible at the optimal point(s). 

The convergence analysis we present relies on the regularization of the Lagrangian multipliers in the Lagrangian function. 
In particular, by augmenting the Lagrangian function with a quadratic regularization term, we establish an upper bound on the norm of the Lagrangian multipliers that is inversely proportional to the parameter of the regularization. By controlling the norm of the Lagrangian multipliers, we in turn control the norm of the sub-gradients of the Lagrangian function. 

We also propose a distributed stochastic primal-dual algorithm to efficiently solve the constrained optimization problems with a large number of constraints (\textit{i.e.} large $m$). In each step of the stochastic algorithm, each agent only needs to compute one sub-gradient of the inequality constraints. In contrast, in the deterministic primal-dual algorithm, the sub-gradients of all the constraints are needed. 

\subsection{Related Works}
Distributed optimization methods dates back to the seminal
work of Bertsekas and Tsitsiklis on parallel computation \cite{bertsekas1989parallel}. More recent developments in distributed optimization are concerned with developing efficient distributed algorithms for constrained optimization problems, \textit{e.g.}, see \cite{koshal2011multiuser,duchi2012dual,lee2013distributed,ram2010distributed}. In \cite{duchi2012dual}, a distributed dual averaging algorithm is proposed, where each agent projects its local variable onto the feasible set $\mathcal{X}$. When the feasible set has more structure, \textit{i.e.}, it can be written as the intersection of finitely many simple convex constraints, a distributed random projection algorithm is studied \cite{lee2013distributed}. Therein, the projection is computed locally by each agent based on the random observations of the local constraint components.
 
 In the case of optimization with coupled linear equality constraints, \textit{i.e.}, when the decision variables of agents must jointly satisfy a set of linear equality constraints, distributed penalty and barrier function methods are studied \cite{li2014decoupling}. Moreover, based on a game theoretic argument, the asymptotic convergence to the optimal solution has been proved. For distributed optimization with a set of global non-linear inequality constraints like this paper, distributed primal-dual methods are studied in \cite{koshal2011multiuser,yuan2011distributed,zhu2012distributed}. A variation of this method is also studied \cite{chang2014distributed}, where each agent has local inequality constraints. However, the proposed methods in \cite{koshal2011multiuser,yuan2011distributed,zhu2012distributed} require a projection of the Lagrangian multipliers onto a simplex at each algorithm iteration, where the simplex itself is compute using a Slater vector. Since computing a Slater vector can be computationally expensive in practice, such distributed primal-dual methods are not suitable when agents have a low computational budget. 
 
\subsection{Organization}
The rest of this paper is organized as follows. In Section \ref{Sec:Problem_Statement}, we present the list of assumptions and define the Lagrangian function. In Section \ref{Section:Distributed Deterministic Primal-Dual Algorithm}, we describe a distributed regularized primal-dual algorithm and prove a convergence rate. We also prove two asymptotic bounds on the constraint violation of the primal-dual solutions. In Section \ref{Section:Distributed Stochastic Primal-Dual Method}, we describe a distributed stochastic primal-dual algorithm and prove convergence rates in expectation and with a high probability. In Section \ref{Sec:Numerical_Simulations}, we present numerical simulations for both deterministic and stochastic algorithms on random and structured graphs. Lastly, in Section \ref{Discussion_and_Conclusion}, we discuss our results and conclude the paper. 

\textbf{Notation}. Throughout the paper, we work with the standard $\ell_{2}$-norm which we denote by $\|\cdot\|$. We define the sub-differential set of a function $f:\real^{d}\rightarrow \real$ as follows
\small \begin{align*}
&\partial f(x)\\
&\coloneqq \left\{\nabla f \in \real^{d}\big |f(y)+\langle \nabla f,y-x \rangle \leq f(x),\forall\ x,y\in \text{dom}(f) \right\}.
\end{align*}\normalsize
Furthermore, we denote the projection of a point $x$ onto the set $\mathcal{X}$ by $\Pi_{\mathcal{X}}(x)\coloneqq \arg\min_{y\in \mathcal{X}}\|x-y\|$. We also use the standard notation $[x]_{+}\coloneqq \max\{0,x\}$ and use the shorthand notation for sets, \textit{e.g.}, $[n]=\{1,2,\cdots,n\}$. For two vectors $x=(x_{1},\cdots,x_{n})$ and $y=(y_{1},\cdots,y_{n})$, $x\preceq y$ means the element-wise inequality $x_{i}\leq y_{i},\forall i\in [n]$. 

We use the standard asymptotic notation for sequences. If $a_{n}$ and $b_{n}$ are positive
sequences, then $a_{n}=\mathcal{O}(b_{n})$ means that $\lim \sup_{n\rightarrow \infty} a_{n}/b_{n}< \infty$, whereas  $a_{n} = \Omega(b_{n})$ means that
$\lim \inf_{n\rightarrow \infty} a_{n}/b_{n} > 0$. Furthermore, $a_{n}=\widetilde{\mathcal{O}}(b_{n})$ implies $a_{n}=\mathcal{O}(b_{n}\text{poly}\log(b_{n}))$. Moreover $a_{n}=o(b_{n})$ means that $\lim_{n\rightarrow \infty}a_{n}/b_{n}=0$ and $a_{n}=\omega(b_{n})$ means that $\lim_{n\rightarrow \infty} a_{n}/b_{n}=\infty$. Lastly, we have $a_{n}=\Theta(b_{n})$ if $a_{n}=\mathcal{O}(b_{n})$ and $a_{n}=\Omega(b_{n})$.

\section{Preliminaries}
\label{Sec:Problem_Statement}
In this section, we formally state the optimization problem as well as the assumptions that we consider in the rest of the paper. 

\subsection{The Lagrangian function}

We consider distributed primal-dual algorithms for solving the optimization problem characterized in eqs. \eqref{Eq:emperical_1}-\eqref{Eq:emperical_2}, which we repeat here
\begin{subequations}
\begin{align}
\label{Eq:emperical_1}
&\min_{x\in \mathcal{X}}   f(x)\coloneqq  \dfrac{1}{n}\sum_{i=1}^{n} f_{i}(x),\\
\label{Eq:emperical_2}
&\mathcal{X}\coloneqq \{x\in \real^{d}:g_{k}(x)\leq 0,k=1,2,\cdots,m\}.
\end{align}
\end{subequations}
We denote the optimal solution of the problem in eqs. \eqref{Eq:emperical_1}-\eqref{Eq:emperical_2} by $x_{\ast}$. Often, when convenient, we will write the inequality constraints $g_{k}(x)\leq 0$,
$k=1,\cdots,m,$ compactly as $g(x)\preceq 0$ with $g(x)\coloneqq (g_{1}(x),\cdots, g_{m}(x))^{T}$. Similarly,
we use $\nabla g(x)$ to denote the matrix $\nabla g(x)\coloneqq
(\nabla g_{1}(x),\cdots,\nabla g_{m}(x))^{T}\in \real^{m\times d}$.

To describe a distributed optimization algorithm for the constraint optimization problem in eqs. \eqref{Eq:emperical_1}-\eqref{Eq:emperical_2}, we define a Lagrangian function for each agent. Specifically, each function $f_{i}(\cdot)$ in eq. \eqref{Eq:emperical_1} is assigned with one agent in a network of $n$ nodes. The regularized Lagrangian function associated with the $i$-th agent is then defined by,
\begin{align}
\label{Eq:Lagrangian_function}
L_{i}(x,\lambda)\coloneqq f_{i}(x)+\langle \lambda,g(x)\rangle  -\dfrac{\eta}{2}\|\lambda\|_{2}^{2},
 \end{align}\normalsize
for all $i=1,2,\cdots,n$, where $\lambda \coloneqq (\lambda_{1},\cdots,\lambda_{m})$ is the vector of the Lagrangian multipliers. 

We also define the sub-gradients of the Lagrangian function as follows
\begin{align}
\label{Eq:sub_gradient_L1}
\nabla_{x} L_{i}(x,\lambda)&\coloneqq \nabla f_{i}(x) +\sum_{k=1}^{m}\lambda_{k}\cdot \nabla g_{k}(x) , \\
\label{Eq:sub_gradient_L2}
\nabla_{\lambda} L_{i}(x,\lambda)&\coloneqq g(x)-\eta \lambda.
 \end{align}\normalsize

Based on the definition of the Lagrangian function $L_{i}(\cdot,\cdot)$ in \eqref{Eq:Lagrangian_function}, we design a distributed algorithm for the following minimax optimization problem
\begin{align}
\label{Eq:min-max-regularized}
\min_{x\in \real^{d}}\max_{\lambda\in \real^{m}_{+}}\dfrac{1}{n} \sum_{i=1}^{n}L_{i}(x,\lambda).
 \end{align}\normalsize

\subsection{Assumptions}
\label{Assumptions}
We make the following assumptions about the feasible set and the underlying functions:
\begin{assumption} \textsc{(Compact Feasible Set)}
\label{Assumption:Compact_Feasible_Set}
The feasible set $\mathcal{X}$ is non-empty, convex, and compact. Furthermore, the feasible set $\mathcal{X}$ is known by each agent. 
\end{assumption}

Let $R\in \real_{+}$ denotes the smallest radius of the $\ell_{2}$-ball centered at the origin that contains the feasible set, \textit{\textit{i.e.}}, $\mathcal{X}\subseteq \ball_{d}(R)\coloneqq \{x\in \real^{d}: \|x\|\leq R\}.$ 

\begin{assumption}\textsc{(Slater Condition)}
There exists a Slater vector $x\in \text{relint}(\mathcal{X})$ such that $g_{k}(x)<0$ for all $k=1,2,\cdots,m$.
\end{assumption}
Under the Slater condition, the primal problem in eq. \eqref{Eq:emperical_1}-\eqref{Eq:emperical_2} and its dual problem have the same optimal objective value, and a dual optimal solution $\lambda_{\ast}$ exists and is finite $\lambda_{\ast}<\infty$, see \cite{nedic2009sub-gradient}. 

The primal-dual pair $(x_{\ast},\lambda_{\ast})\in\mathcal{X}\times \real_{+}$ is a saddle point of the minimax optimization problem in eq. \eqref{Eq:min-max-regularized}, if it satisfies the inequalities  
\begin{align}
\sum_{i=1}^{n} L_{i}(x_{\ast},\lambda)\leq \sum_{i=1}^{n} L_{i}(x_{\ast},\lambda_{\ast})\leq \sum_{i=1}^{n} L_{i}(x,\lambda_{\ast}),
\end{align} 
for all $x\in \mathcal{X}$, and $\lambda\in \real_{+}$. Note that the saddle point $(x_{\ast},\lambda_{\ast})$ is not unique, unless at least one function $f_{i}(\cdot)$ is strictly convex. Therefore, in the following, the primal-dual pair $(x_{\ast},\lambda_{\ast})$ denotes a generic saddle point of the minimax problem \eqref{Eq:min-max-regularized}.

The following assumption is standard in the optimization literature:
\begin{assumption}
\label{Assumption:Lipschitz Functions}
\textsc{(Lipschitz Functions)} We assume that the functions $f_{i}(\cdot)$ and $g_{k}(\cdot)$ are convex on the Euclidean ball $\ball_{d}(R)$, for all $i\in [n]$ and $k\in [m]$. Further, the sub-gradients $\nabla f_{i}(x)\in \partial f_{i}(x)$, and $\nabla g_{k}(x)\in \partial g_{k}(x),\forall k\in [m]$ are bounded
\begin{align*}
\|\nabla f_{i}(x)\|&\leq L,\quad \|\nabla g_{k}(x)\|\leq L, 
\end{align*}
for all $x\in \ball_{d}(R)$, where $L<\infty$ is a constant. 
\end{assumption}

In Assumption \ref{Assumption:Lipschitz Functions}, the Lipschitz continuity conditions on the underlying functions are defined on the Euclidean ball $\ball_{d}(R)$, which is a larger set compared to the feasible set $\mathcal{X}$. This extension is essential since we confine the primal variables to the Euclidean ball $\ball_{d}(R)$ instead of $\mathcal{X}$ to simplify the projection in the primal-dual algorithm (cf. Algorithm \ref{CHalgorithm}).

The communication network between the $n$ agents is represented with a connected graph $G=(V,E)$, where $V=\{1,2,\cdots,n\}$ is the set of nodes of the graph, and $E\subseteq V\times V$ is the set of edges between those nodes. Thus, $(i,j)\in E$ if the node (agent) $i$ communicates with the node (agent) $j$, and vice versa. We assume that the connectivity graph is fixed in the sense that it does not change during the algorithm runtime. 

 Associated with the graph $G=(V,E)$, we consider a weight matrix $W\coloneqq [W]_{ij},(i,j)\in V\times V$ for averaging the information that each node receives from its neighbors. We consider the following assumption regarding $W$: 

\begin{assumption}
\label{Assumption:Weight_matrix}
\textsc{(Doubly Stochastic Weight Matrix)}
The graph $G$ and the weight matrix $W$ satisfy the following conditions:
\begin{itemize}
\item The graph $G$ is connected.

\item The weight matrix $W$ is doubly stochastic,
\begin{align*}
W\times \mathbbm{1}_{n}&=\mathbbm{1}_{n},\\
\mathbbm{1}^{T}_{n}\times W&=\mathbbm{1}^{T}_{n},
\end{align*}
where $\mathbbm{1}_{n}\in \real^{n}$ is the column
vector with all elements equal to one. 
\item The weight matrix $W$ respects the structure of the graph $G=(V,E)$, \textit{i.e.},
\begin{align*}
&W_{ij}>0\quad \text{if}\quad (i,j)\in E\\
&W_{ij}=0\quad \text{if}\quad (i,j)\notin E.
\end{align*}
\end{itemize}
\end{assumption}

For $n\times n$ doubly stochastic matrices, the singular values can be sorted in a non-increasing fashion $\sigma_{1}(W)\geq \sigma_{2}(W)\geq \cdots \geq \sigma_{n}(W)\geq 0$, where $\sigma_{1}(W)=1$ due to Assumption \ref{Assumption:Weight_matrix}. Throughout the paper, we refer to $1-\sigma_{2}(W)$ as the spectral gap of the matrix $W$.

In the following, we review two popular weight matrices $W$ that are proposed in the optimization literature:

\paragraph{Lazy Metropolis Matrix} Motivated by the hitting time of the lazy Markov chains, Olshevsky \cite{olshevsky2014linear} has proposed the \textit{lazy Metropolis} matrix for the weight matrix, \textit{i.e.},
\begin{align}
\label{Eq:LazyasIam}
[W]_{ij}=\begin{cases}
      \dfrac{1}{2\max(d(i)+1,d(j)+1)} & \text{if}\ (i,j)\in E \\
      0, & \text{if}\ (i,j)\not\in E.
    \end{cases}
\end{align}
Here, $d(i)$ and $d(j)$ are degrees of the nodes $i$ and $j$, respectively. To
choose the weights according to eq. \eqref{Eq:LazyasIam}, agents will need to spend an additional round at the beginning
of the algorithm broadcasting their degrees to their neighbors.

It is easy to verify that the lazy Metropolis matrix $W$ is stochastic, symmetric, and diagonally dominant. Further, due to the symmetry, the singular values are simply the absolute value of the eigenvalues. More importantly, the inverse of the spectral gap has an upper bounded proportional to $n^{2}$ \cite{olshevsky2014linear}. Specifically, as shown in \cite{olshevsky2014linear}, regardless of the graph structure $G$, the spectral gap corresponding to the lazy Metropolis weight matrix is given by  
\begin{align}
\label{Eq:Moreimportnantly}
\dfrac{1}{1-\sigma_{2}(W)}\leq 71n^{2}.
\end{align}   

\paragraph{Normalized Graph Laplacian} Another popular choice of the weight matrix is the graph Laplacian \cite{duchi2012dual}. Consider the graph adjacency matrix $A$, where $A_{ij}=0$ if $(i,j)\not = 1$, and $A_{ij}=1$ otherwise. Further, consider the diagonal matrix $D\coloneqq \text{Diag}(d_{1},\cdots,d_{n})$, where $d_{i}\coloneqq \sum_{j=1}^{n}A_{ij}$. The normalized graph Laplacian is defined as
\begin{align*}
\mathcal{L}(G)\coloneqq I-D^{-1/2}A D^{-1/2}.
\end{align*}
Now, let $\delta\coloneqq \max_{i\in V}\sum_{j=1}^{n}A_{ij}$. When the matrix is degree regular, \textit{i.e.}, $d_{i}=d$ for all $i\in [n]$, the following weight matrix $W$ is proposed in \cite{duchi2012dual},
\begin{align*}
W\coloneqq I-\dfrac{d}{d+1}\mathcal{L}.
\end{align*}
Further, for the case of non-degree regular graphs, the following weight matrix is proposed
\begin{align*}
W\coloneqq I-\dfrac{1}{d_{\max}+1}D^{1/2}\mathcal{L}D^{1/2},
\end{align*}
where $d_{\max}\coloneqq \max_{i\in V}d_{i}$.

\section{Distributed Deterministic Primal-Dual Algorithm}
\label{Section:Distributed Deterministic Primal-Dual Algorithm}

In our proposed distributed primal-dual algorithm, the $i$-th agent maintains a local copy of the primal variables $x_{i}(t) \in \real^{d}$ and the Lagrangian multipliers $\lambda_i(t)\in \real^{m}$. Here, $x_{i}(t)$ and $\lambda_{i}(t)$ stands for the estimate of the $i$-th agent of the decision variable $x$ and $\lambda$ after $t$ steps. Therefore, $x_i(t)$ and $\lambda_i(t)$ have the same dimension as the primal variable $x$ and dual variable $\lambda$. The initialization and update rule of $x_i(t)$ and $\lambda_i(t)$ is described in Algorithm \ref{CHalgorithm}.

\begin{algorithm}[t!]
\caption{\footnotesize{\textsc{Distributed Regularized Primal-Dual Method}}}
\label{CHalgorithm}
\begin{algorithmic}[1]
\State \textbf{Initialize}: $x_{i}(0)=0\in \ball_{d}(R)$, $\lambda_{i}(0)=0\in \real_{+}^{m}, \forall i\in V$ and a non-negative, non-increasing step size sequence $\{\alpha(t)\}_{t=0}^{\infty}$.
\For{$t=0,1,2,\cdots$ at the $i$-th node}
\State Update the auxiliary primal and dual variables
\begin{subequations}
\begin{align}
\label{Eq:aux_1}
y_{i}(t)&=x_{i}(t)-\alpha(t)\nabla_{x}L_{i}(x_{i}(t),\lambda_{i}(t)),\\
\label{Eq:aux_2}
\gamma_{i}(t)&=\lambda_{i}(t)+\alpha(t)\nabla_{\lambda}L_{i}(x_{i}(t),\lambda_{i}(t)).
 \end{align}\normalsize
\end{subequations}

\State Run the consensus steps 
\begin{subequations}
\begin{align}
\label{Eq:That_is_the_way}
x_{i}(t+1)&= \Pi_{\ball_{d}(R)}\left(\sum_{j=1}^{n}[W]_{ij}y_{j}(t)\right)\\  \label{Eq:EasyProjection1}
&=\dfrac{R\cdot \left(\sum_{j=1}^{n}[W]_{ij}y_{j}(t)\right)}{\max\{R,\|\sum_{j=1}^{n}[W]_{ij}y_{j}(t)\|\}}, \\
\label{Eq:EasyProjection2}
\lambda_{i}(t+1)&=\Pi_{\real^{m}_{+}}\left(\sum_{j=1}^{n}[W]_{ij}\gamma_{j}(t)\right).
 \end{align}\normalsize
\end{subequations}

\State Compute the weighted average: $\widehat{x}_{i}(t)={\sum_{s=0}^{t+1}\alpha(s)x_{i}(s)\over \sum_{s=0}^{t+1}\alpha(s)}$ for all $i\in V$.

\EndFor
\State \textbf{Output}: $\widehat{x}_{i}(t)$ for all $i\in V$.
\end{algorithmic}
\end{algorithm}

We remark that Algorithm \ref{CHalgorithm} is an example of ``anytime algorithm", meaning that it is stoppable at any time and it returns $\widehat{x}_{i}(t)$ as the solution of the $i$-th agent to the optimization problem in eq. \eqref{Eq:emperical_1}-\eqref{Eq:emperical_2}. Moreover, the solution improves as $t$ increase in the sense that the  cumulative objective function $f(\widehat{x}_{i}(t))$ of the $i$-th agent tends to the optimal objective value $f(x_{\ast})$ for all $i\in [n]$ as $t\rightarrow \infty$.

In Algorithm \ref{CHalgorithm}, the projection onto the Euclidean ball $\ball_{d}(R)$ is essential since without it, the primal variables $x_{i}(t+1)$ in eq. \eqref{Eq:EasyProjection1} can take any value from $\real^{d}$. In such circumstance, the Lipschitz continuity of functions $f_{i}(\cdot)$ and $g_{k}(\cdot)$ in Assumption \ref{Assumption:Lipschitz Functions} must be extended to the entire Euclidean space $\real^{d}$, which is too stringent for many functions. 

However, the projections onto the Euclidean ball $\ball_{d}(R)$ and the non-negative orthant $\real_{+}^{m}$ in eqs. \eqref{Eq:EasyProjection1}-\eqref{Eq:EasyProjection2}, respectively, have closed form solutions. Therefore, each iteration of Algorithm \ref{CHalgorithm} can be computed efficiently. Notice that since the Euclidean ball $\ball_{d}(R)$ contains the feasible set $\mathcal{X}$, the inequality constraints in eq. \eqref{Eq:emperical_2} can be violated. To provide a guarantee on the asymptotic feasibility of solutions of Algorithm \ref{CHalgorithm}, we establish an upper bound on the constraint violation and prove that it goes to zero as the number of steps goes to infinity $t\rightarrow \infty$ (cf. Theorem \ref{Thm:3}).

\begin{remark}
\label{Remark}
To compute a concise convergence rate, in Algorithm \ref{CHalgorithm} we use the special initialization $x_{i}(0)=0\in \ball_{d}(R)$, $\lambda_{i}(0)=0\in \real_{+}^{m}$. Without this restriction, the convergence analysis of Algorithm \ref{CHalgorithm} is valid, but the convergence rates differ from what we present in this paper. In practice, Algorithm \ref{CHalgorithm} can be initialized from any feasible point in the Euclidean ball $\ball^{d}_{2}(r)\times \real_{+}$ as we demonstrate in the numerical simulations (cf. Section \ref{Sec:Numerical_Simulations}). 
\end{remark}

\subsection{Comparison with Related Primal-Dual Methods}

Augmented Lagrangian methods for constrained optimization have been studied extensively  \cite{koshal2011multiuser,mahdavi2012trading,mahdavi2012stochastic,yuan2016regularized}. In \cite{mahdavi2012trading}, a regularized online primal-dual method is studied, where it has been shown that it achieves a sub-linear `regret' and satisfies the inequality constraints asymptotically. However, the analysis of \cite{mahdavi2012trading} is not applicable to the multi-agent settings since it does not provide a guarantee for the boundedness of the norm of the Lagrangian multipliers $\|\lambda_{i}(t)\|$. It turns out that bounding this norm is essential for analyzing the `consensus terms' (cf. Lemma \ref{Lemma:Consensus}). 

To ensure the boundedness of the norm of the Lagrangian multipliers in the multiagent settings, a distributed regularized primal-dual algorithm similar to Algorithm \ref{CHalgorithm} is proposed in \cite{yuan2016regularized}. However, the optimization problem only includes one constraint $g(x)\leq 0$ (\textit{i.e.}, $m=1$) under the additional assumption that $\min_{x:g(x)=0} \|\nabla g(x)\|_{2}\geq \rho, \nabla g(x)\in \partial g(x)$, for some $\rho>0$. Moreover, the analysis of the convergence rate in \cite{yuan2016regularized} depends on $\rho$. Specifically, the difference in function value at the final estimate and the optimal value is upper bounded by an expression which is proportional to $1/\rho$. Therefore, when $\rho$ is small, the upper bound is potentially very loose. More importantly, the convergence rate of \cite{yuan2016regularized} has a network scaling of $\mathcal{O}(n^{3})$ compared to $\mathcal{O}(\log^{3\over 2}(n))$ that we prove in this paper (cf. Theorem \ref{Thm:3}). 

To bound the norm of the Lagrangian multipliers in distributed primal-dual methods, a different strategy is pursed in \cite{yuan2011distributed,zhu2012distributed, chang2014distributed,koshal2011multiuser}. Specifically, consider a Slater vector $\tilde{x}\in \text{relin}(\mathcal{X}$), \textit{i.e.}, the vector that satisfies
\begin{align}
\label{Eq:Assumption}
g(\tilde{x})\prec 0.
\end{align}
Let $\mu \coloneqq \min_{k=1,2,\cdots,m}\{-g_{k}(\tilde{x})\}$ and define
\begin{align*}
\mathfrak{F}(\lambda)\coloneqq \inf_{x\in \mathcal{X}} f(x)+\langle\lambda,g(x)\rangle.
\end{align*}
In the proposed primal-dual algorithms in \cite{yuan2011distributed,zhu2012distributed, chang2014distributed,koshal2011multiuser}, each agent projects its local Lagrangian multipliers $\lambda_{i}(t)$ onto the following simplex 
\begin{align}
\Lambda	\coloneqq \{\lambda\in \real_{+}^{m}:{\|\lambda\|_{1}\leq \mu^{-1}\cdot(f(\tilde{x})-\mathfrak{F}(\hat{\lambda}))}\},
\end{align}
where $\hat{\lambda}\in \real_{+}^{m}$ is an arbitrary vector. However, there are two drawbacks with the projection onto the simplex:

First, to compute the simplex $\Lambda$, a Slater vector $\tilde{x}$ must be computed which is inefficient.\footnote{To guarantee a zero duality gap, we also require the Slater condition (or any other constraint qualifications) to  hold. However, computing a Slater vector is not needed in Algorithm \ref{CHalgorithm}.} For instance, to compute a Slater vector $\tilde{x}$ for a feasible set defined by linear inequality constraints $\mathcal{X}=\{x\in \real^{d}:\langle a_{k},x\rangle\leq b_{k},k=1,2,\cdots,m\}$ where $a_{k}\in \real^{d},b_{k}\in \real$, we must solve the following optimization problem 
\begin{align}
\label{Eq:Slater_vector}
\tilde{x}=\arg \min_{x\in \mathcal{X}} b-Ax,
\end{align}
where $b\coloneqq (b_{1},\cdots,b_{m})^{T}\in \real^{m\times 1}$ and $A\coloneqq (a_{1};\cdots;a_{m})\in \real^{m\times d}$. Provided that there exists a vector $\tilde{x}$ that satisfies $A\tilde{x}<b$, the solution of the minimization problem \eqref{Eq:Slater_vector} yields a Slater vector.
 
Second, the projection onto the simplex $\Lambda$ requires solving a separate minimization problem. In comparison, the projection in Algorithm \ref{CHalgorithm} is onto the non-negative orthant $\real^{m}_{+}$ which can be computed efficiently by replacing each negative component of the vector $\sum_{j=1}^{n}[W]_{ij}\lambda_{j}(t)$ in Step 4 of Algorithm \ref{CHalgorithm} with zero. 

\subsection{Convergence Rate and Constraint Violation Bounds}
\label{Section:Distributed Regularized Primal-Dual Algorithm}

Here, we prove a convergence rate for Algorithm \ref{CHalgorithm}, and also establish two constraint violation bounds. We defer the proof of the theorems to Appendix \ref{App:Proofs_of_Main_Results}. 

We first establish a general upper bound for the cost function in eq. \eqref{Eq:emperical_1} for an arbitrary choice of the stepsize $\alpha(t)$ and the regularization parameter $\eta$:
\begin{lemma}
\label{Lem:1}
After $T\in \mathbb{N}$ iterations of Algorithm \ref{CHalgorithm}, the estimation $\widehat{x}_{i}(T)$ of the primal variable of each agent $i=1,2,\cdots,n$ satisfies
\footnotesize \begin{align}
\label{Eq:Prop_inequality}
&f(\widehat{x}_{i}(T))-f(x_{\ast})\leq \dfrac{1}{\sum_{t=0}^{T-1}\alpha(t)}\Bigg[\dfrac{1}{2} \|x_{\ast}\|^{2}\\ \nonumber
&\hspace{-2mm}+\dfrac{L}{n} \sum_{t=0}^{T-1}\sum_{j=1}^{n}\alpha(t) \|x_{i}(t)-x_{j}(t)\|-\dfrac{\eta}{2n}\sum_{t=0}^{T-1}\sum_{j=1}^{n}\alpha(t)\|\lambda_{j}(t)\|^{2}\\ \nonumber
&\hspace{-2mm} +\dfrac{1}{2n}\sum_{t=0}^{T-1}\sum_{j=1}^{n}\alpha^{2}(t)\left(\|\nabla_{x}L_{j}(x_{j}(t),\lambda_{j}(t))\|^{2}+\|\nabla_{\lambda}L_{j}(x_{j}(t),\lambda_{j}(t))\|^{2}\right)\Bigg].
 \end{align}\normalsize
\end{lemma}
\begin{proof}
See Appendix \ref{Proof of Proposition 1}.
\end{proof}

To parse the upper bound in Lemma \ref{Lem:1}, we examine each term separately. 

The first term is intuitive as it measures the distance between the initial point (which is chosen to be the origin $x_{i}(0)=0,\forall i\in [n]$ in Algorithm \ref{CHalgorithm}) and an optimal point $x_{\ast}$. 

The second term measures the distance between the primal variables of different agents in the network since it includes the pairwise difference $\|x_{i}(t)-x_{j}(t)\|$. This term, often referred to as the ``consensus" term in the distributed optimization literature, is related to the spectral gap of the weight matrix $W$.  

The third term is due to the regularization term that is included in the Lagrangian function \eqref{Eq:Lagrangian_function}. Although choosing an arbitrary large regularization parameter $\eta$ results in a smaller upper bound, certain trade-offs between the convergence rate of the algorithm and the constraint violation of the inequality constraints \eqref{Eq:empirical_risk_form_1} prohibits a large value for $\eta$ (see Theorem \ref{Thm:3} below). 

The last term in the upper bound \eqref{Eq:Prop_inequality} includes the norms of the sub-gradients defined in eqs. \eqref{Eq:sub_gradient_L1}-\eqref{Eq:sub_gradient_L2}. In the earlier studies of the primal-dual methods, these norms were bounded under different assumptions on the feasible set, see \cite{yuan2011distributed,zhu2012distributed, chang2014distributed}. The challenge is due to the fact that the Lagrangian multipliers $\lambda_{i}(t)$ in the sub-gradients may take a large value. Therefore, to ensure that the vector of Lagrangian multipliers has a bounded norm, various assumptions on the feasible set and the inequality constraints were considered. In our analysis, the norm of the Lagrangian multipliers is controlled by adding the regularization term to the Lagrangian function, see \eqref{Eq:Lagrangian_function}.   
 
 Based on Lemma \ref{Lem:1}, we derive the following explicit convergence rate for Algorithm \ref{CHalgorithm} using a decreasing stepsize:
\begin{theorem} \textsc{(Convergence rate)}
\label{Thm:2}
Consider $T$ iterations of Algorithm \ref{CHalgorithm} with the stepsize $\alpha(t)={R\over \sqrt{t+1}}$ and the regularization parameter $\eta\alpha(t)\leq {1\over 2}$. The estimation of the $i$-th agent $\widehat{x}_{i}(T)$ satisfies
\small \begin{align}
\label{Eq:ConvergenceRate_Rationa0}
f(\widehat{x}_{i}(T))-f(x_{\ast})&\leq \dfrac{RC\log(T)}{\sqrt{T}-1}, \quad T\geq 2,
\end{align}\normalsize
for all $i=1,2,\cdots,n$, where $C$ is defined as follows, 
\small \begin{align}
\label{Eq:CONSTANT_C}
C \coloneqq 1+{5\over 2} mL^{2}R^{2}+20 L^{2} \left(1+\dfrac{nm^{3/2}LR}{\eta}\right)^{2}\left(\dfrac{\log(T\sqrt{nT})}{1-\sigma_{2}(W)}\right)^{3\over 2}.
\end{align}\normalsize
\end{theorem} 
\begin{proof}
See Appendix \ref{App:TheRestofTheProof}.
\end{proof}

Let us emphasize a few points about Theorem \ref{Thm:2}. 

The constraint $\eta\alpha(t)\leq 1$ on the regularization parameter can be easily satisfied since $\{\alpha(t)\}_{t=0}^{\infty}$ is a decreasing sequence and usually takes a small value. In addition, with a regularization parameter $\small \eta=\Theta(\sqrt{n})$, the scaling of the algorithm is $\mathcal{O}(\log^{3\over 2}(\sqrt{n}))$ which is slightly worse than the dual averaing algorithm by a factor of $\log(\sqrt{n})$ \cite{duchi2012dual}. Moreover, when $\eta=\Theta(\sqrt{m})$, the upper bound in eq. \eqref{Eq:ConvergenceRate_Rationa0} grows linearly in the number of constraints $m$. It is interesting to see whether the linear growth rate can be improved.

In the upper bound \eqref{Eq:ConvergenceRate_Rationa0}, the convergence rate of the algorithm is given by $\widetilde{\mathcal{O}}(T^{-{1\over 2}})$ when $\eta$ is independent of $T$. It is well-known that a lower bound for the regret of the centralized first order methods with non-smooth objective functions has an order of $\Omega(T^{-1/2})$, see \cite{agarwal2009information}. Therefore, when $\eta$ is independent of the number of steps $T$, Algorithm \ref{CHalgorithm} is order optimal up to a polynomial factor of $\log(T)$.
 
As mentioned in Section \ref{Sec:Problem_Statement}, the local variable of each agent $\widehat{x}_{i}(T)$ in Algorithm \ref{CHalgorithm} is computed via the projection of the primal variables onto the Euclidean ball $\ball_{d}(R)$ that contains the feasible set $\mathcal{X}$. Therefore, in principle the inequality constraints can be violated. In the next theorem, we show that the upper bound on the constraint violation is related to the regularization parameter $\eta$.
\begin{theorem}\textsc{(Constraint Violation Bound)}
\label{Thm:3}
Consider $T$ iterations of Algorithm \ref{CHalgorithm} with the stepsize $\alpha(t)={R \over \sqrt{t+1}}$ and the regularization parameter $\eta\alpha(t)\leq {1\over 2},\forall t\in [T]$. Further, 
The constraint violation has the following asymptotic bound for all $i\in V$,
\small \begin{align}
\label{Eq:FirstConstraintViolationBound}
\left\|\left[\dfrac{1}{n}\sum_{i=1}^{n}g(\widehat{x}_{i}(T)) \right]_{+}\right\|^{2}_{2}=\mathcal{O}(\eta),
 \end{align}\normalsize
Furthermore, if the optimal solution $x_{\ast}$ is strictly feasible $g(x_{\ast})\prec 0$ at an optimal point, we have
\small \begin{align}
\label{Eq:SecondConstraintViolationBound}
\left\|\left[\dfrac{1}{n}\sum_{i=1}^{n}g(\widehat{x}_{i}(T)) \right]_{+}\right\|^{2}_{2}=\mathcal{O}\left(\dfrac{\eta \log(T)}{\sqrt{T}}\right).
 \end{align}\normalsize
\end{theorem}
\begin{proof}
The proof is deferred to Appendix \ref{Proof of Theorem 3}. 
\end{proof}

From Theorems \ref{Thm:2} and \ref{Thm:3}, we observe that when one of the constraints is binding at the optimal solution, \textit{i.e.}, $g_{k}(x_{*})=0$ for at least one coordinate $k\in [m]$, there is a tension between the convergence rate in eqs. \eqref{Eq:ConvergenceRate_Rationa0}-\eqref{Eq:CONSTANT_C} and the decay rate of the constraint violation bound in eq. \eqref{Eq:FirstConstraintViolationBound}. Clearly, by adopting a $\eta$, we obtain a small constraint violation bound. However, a small $\eta$ yields a large upper bound in eqs. \eqref{Eq:ConvergenceRate_Rationa0}-\eqref{Eq:CONSTANT_C}. To examine this trade-off more precisely, suppose $\eta=\Theta(T^{-r})$ for $r\in (0,1/2)$. In this case, the convergence rate as characterized in eqs. \eqref{Eq:ConvergenceRate_Rationa0}-\eqref{Eq:CONSTANT_C} is $\widetilde{\mathcal{O}}(1/T^{{1\over 2}-r})$, while the constraint violation in eq. \eqref{Eq:FirstConstraintViolationBound} becomes $\mathcal{O}(1/T^{r})$. Interestingly, when the inequality constraints are satisfied strictly at an optimal point, \textit{i.e.}, $g_{k}(x_{\ast})< 0, \forall  k\in [m]$, then the constraint violation bound in eq. \eqref{Eq:SecondConstraintViolationBound} decays to zero as $T\rightarrow \infty$ even for $\eta=\mathcal{O}(1)$. Consequently, there is no trade-off between the convergence rate and the constraint violation when $g_{k}(x_{\ast})<0, \forall  k\in [m]$.

\section{Distributed Stochastic Primal-Dual Method}
\label{Section:Distributed Stochastic Primal-Dual Method}

\begin{algorithm}[t!]
\caption{\footnotesize{\textsc{Distributed Stochastic Primal-Dual Method}}}
\label{CHalgorithm-1}
\begin{algorithmic}[1]
\State \textbf{Initialize}: $x_{i}(0)=0\in \ball_{d}(R)$, $\lambda_{i}(0)=0\in \real_{+}^{m}, \forall i\in V$ and a non-negative, non-increasing stepsize sequence $\{\alpha(t)\}_{t=0}^{\infty}$.
Select $p_{i}(0)={\tt{Uniform}}\{1,2,\cdots,m\}$.
\For{$t=0,1,2,\cdots$ at the $i$-th node $i\in V$}
\State Draw a random index $K_{i}(t)\in\{1,2,\cdots,m\}$ according to the distribution $K_{i}(t)\sim p_{i}(t)$.
\State Update the primal and dual variables
\begin{align*}
y_{i}(t)&=x_{i}(t)-\alpha(t)\nabla_{x}\widehat{L}_{i}(x_{i}(t),\lambda_{i}(t);K_{i}(t))\\
\gamma_{i}(t)&=\lambda_{i}(t)+\alpha(t)\nabla_{\lambda}\widehat{L}_{i}(x_{i}(t),\lambda_{i}(t)).
\end{align*}
\State Run the consensus step
\begin{align*}
x_{i}(t+1)&=  \Pi_{\ball_{d}(R)}\left(\sum_{j=1}^{n}[W]_{ij}y_{j}(t)\right)\\
&=\dfrac{R\cdot \left(\sum_{j=1}^{n}[W]_{ij}y_{j}(t)\right)}{\max\{R,\|\sum_{j=1}^{n}[W]_{ij}y_{j}(t)\|_{2}\}}, \\
\lambda_{i}(t+1)&=\Pi_{\real_{+}^{m}}\left(\sum_{j=1}^{n}[W]_{ij}\gamma_{i}(t)\right).
\end{align*}
\State Update $p_{i}(t)={1\over \|\lambda_{i}(t)\|_{1}}(\lambda_{1}(t),\cdots,\lambda_{m}(t)),\forall i\in V$. Set $p_{i}(t)={\tt{Uniform}}\{1,\cdots,m\}$ if $\lambda_{i}(t)=0$.  

\State Compute the weighted average: $\widehat{x}_{i}(t)={\sum_{s=0}^{t+1}\alpha(s)x_{i}(s)\over \sum_{s=0}^{t+1}\alpha(s)}$ for all $i\in V$.
\EndFor
\State \textbf{Output}: $\widehat{x}_{i}(t)$ for all $i\in V$.
\end{algorithmic}
\end{algorithm}

As mentioned earlier in the previous section, the projection onto the Euclidean ball $\ball_{d}(R)$ in eq. \eqref{Eq:EasyProjection1} of Algorithm \ref{CHalgorithm} has a closed form expression, and thus it can be computed efficiently. However, the algorithm may still be computationally inefficient, especially when there is a large number of constraints. This is due to the fact that the sub-gradients of all the constraints must be calculated in eq \eqref{Eq:sub_gradient_L1}. 

To resolve this issue, in this section we propose a distributed stochastic primal-dual algorithm. In contrast to Algorithm \ref{CHalgorithm} which requires the sub-gradients of all inequality constraints at each step, the stochastic algorithm only requires one sub-gradient, namely the sub-gradient associated with the constraint that has the largest Lagrangian multiplier.   

More precisely, at each step $t=0,1,2,\cdots$ of the stochastic algorithm, we prescribe a distribution  $p_{i}(t)\coloneqq (p_{i,1}(t),p_{i,2}(t),\cdots,p_{i,m}(t)),\sum_{k=1}^{m}p_{i,k}(t)=1$ for each agent on the set of labels $\{1,2,\cdots,m\}$ associated with the inequality constraints \eqref{Eq:emperical_2}. The distribution $p_{i}(t)$ of each agent is determined based on the observed Lagrangian multipliers at time $t$, \textit{i.e.},
\begin{align*}
p_{i,k}(t)\coloneqq \lambda_{i,k}(t)/\|\lambda_{i}(t)\|_{1}, \quad \|\lambda_{i}(t)\|_{1}\not=0.
\end{align*}
When the Lagrangian multipliers are all zero $\lambda_{i}(t)=0\in \real_{+}^{m}$, we consider a uniform distribution, \textit{i.e.}, $p_{i}(t)={\tt{Uniform}}\{1,\cdots,m\}$. Let $K_{i}(t)$ denotes a random variable with the distribution $p_{i,k}(t)$, that is $p_{i,k}(t)=\Prob[K_{i}(t)=k]$. For each given index $k\in \{1,2,\cdots,m\}$ and for the pair of variables $(x,\lambda)\in \ball_{d}(R)\times \real_{+}^{m}$, we also let
\begin{subequations}
\begin{align}
\label{Eq:Label1}
\nabla_{x}\widehat{L}_{i}(x,\lambda;k)&\coloneqq  \nabla f_{i}(x)+\|\lambda\|_{1}\nabla g_{k}(x)\\
\label{Eq:Label2}
\nabla_{\lambda}\widehat{L}_{i}(x,\lambda)&	\coloneqq g(x)-\eta \lambda.
\end{align}
\end{subequations}

Equipped with these definitions, in Algorithm \ref{CHalgorithm-1} we present the distributed stochastic primal-dual algorithm. Let $\mathfrak{F}_{t}$ denotes the $\sigma$-algebra of all random variables $\{(x_{i}(s),\lambda_{i}(s),K_{i}(s))\}_{s=0}^{t-1}$. Conditioned on $\mathfrak{F}_{t}$, the stochastic sub-gradient $\nabla_{x}\widehat{L}_{i}(x_{i}(t),\lambda_{i}(t);K_{i}(t))$ defined in eq. \eqref{Eq:Label1} is an \textit{unbiased} estimate of the deterministic sub-gradient  $\nabla_{x}L_{i}(x_{i}(t),\lambda_{i}(t))$ defined in eq. \eqref{Eq:sub_gradient_L1}. In particular, by computing the expectation of the estimator $\nabla_{x}\widehat{L}_{i}(x_{i}(t),\lambda_{i}(t);K_{i}(t))$ with respect to the distribution $p_{i}(t)$, we obtain the deterministic sub-gradient
\begin{align}
\nonumber
&\expect_{p_{i}(t)}[\nabla_{x}\widehat{L}_{i}(x_{i}(t),\lambda_{i}(t);K_{i}(t))|\mathfrak{F}_{t}]\\ \nonumber
&=\sum_{k=1}^{m}\left(\nabla f_{i}(x_{i}(t))+\|\lambda_{i}(t)\|_{1}\nabla g_{k}(x_{i}(t))\right)p_{i,k}(t)\\ \nonumber
&=\nabla f_{i}(x_{i}(t))+ \sum_{k=1}^{m}\|\lambda_{i}(t)\|_{1}\nabla g_{k}(x_{i}(t))\dfrac{\lambda_{i,k}(t)}{\|\lambda_{i}(t)\|_{1}}\\ \nonumber
&=\nabla f_{i}(x_{i}(t))+ \sum_{k=1}^{m}\lambda_{i,k}(t)\nabla g_{k}(x_{i}(t)))\\ \label{Eq:Unbiased_Estimator}
&=\nabla_{x}L_{i}(x_{i}(t),\lambda_{i}(t)).
\end{align}
In addition, the stochastic sub-gradient $\nabla_{\lambda}\widehat{L}_{i}(x_{i}(t),\lambda_{i}(t))$ corresponds to the deterministic definition in eq. \eqref{Eq:sub_gradient_L2}, \textit{i.e.},  
\begin{align}
\label{Eq:Unbiased_Estimator01}
\nabla_{\lambda}\widehat{L}_{i}(x_{i}(t),\lambda_{i}(t))=\nabla_{\lambda}L_{i}(x_{i}(t),\lambda_{i}(t)).
\end{align}
\subsection{Convergence Rate and Constraint Violation Bounds}

As we demonstrated in eq. \eqref{Eq:Unbiased_Estimator}, the stochastic sub-gradient defined in eq. \eqref{Eq:Label1} is an unbiased estimator for the deterministic sub-gradient. We thus leverage the method of bounded martingale difference to derive a high probability convergence bound for Algorithm \ref{CHalgorithm-1}. 

\begin{theorem}
\label{Thm:High_Probability_Bound}
Consider Algorithm \ref{CHalgorithm-1} with the stepsize $\alpha(t)={R\over \sqrt{t+1}}$ and the regularization parameter $\eta\alpha(t)\leq {1\over 2}$. Let $\widehat{x}_{i}(T)$ denotes the estimate of the $i$-th agent at the end of $T$ iterations. Then,
\begin{itemize}
\item [(i)] With the probability of at least $1-{1\over T}$, 
\begin{align}
\label{Eq:Stochastic}
f(\widehat{x}_{i}(T))&-f(x_{\ast})\\  \nonumber &\leq  \dfrac{\log(T)}{\sqrt{T}-1}\left(RC+\dfrac{4\sqrt{10}nm^{2}L^{2}R^{3}}{\eta}\right),
\end{align}
for all $i\in V$ and all $T\geq 2$, where $C$ is the constant defined in eq. \eqref{Eq:CONSTANT_C}.  

\item[(ii)] The expected convergence rate is given by
\begin{align}
&\expect[f(\widehat{x}_{i}(T))-f(x_{\ast})]\leq \dfrac{RC\log(T)}{\sqrt{T}-1}.
\end{align}
\end{itemize}

\end{theorem}

From eq. \eqref{Eq:Stochastic}, we observe that $f(\widehat{x}_{j}(T))\rightarrow f(x_{\ast})$ almost surely as $T\rightarrow \infty$. Moreover, by comparing the high probability bound in eq. \eqref{Eq:Stochastic} with the convergence rate of the deterministic algorithm in \eqref{Eq:ConvergenceRate_Rationa0}, we see that both Algorithms \ref{CHalgorithm} and \ref{CHalgorithm-1} yield the same convergence rate of $\mathcal{O}(\log(T)/\sqrt{T})$. This is due to the fact that in both algorithms, the averaging step (Steps 4 of Algorithm \ref{CHalgorithm} and Step 5 of Algorithm \ref{CHalgorithm-1}) is the bottleneck of the convergence rate. 

In the next theorem, we address the constraint violation performance of Algorithm \ref{CHalgorithm-1}. The proof is omitted since it is similar to the proofs of Theorems \ref{Thm:3} and \ref{Thm:High_Probability_Bound}. 

\begin{theorem}
Consider $T$ iterations of Algorithm \ref{CHalgorithm} with the stepsize $\alpha(t)={R \over \sqrt{t+1}}$ and the regularizer's parameter $\eta\alpha(t)\leq {1\over 2},\forall t\in [T]$. With the probability of at least $1-{1\over T}$, the constraint violation has the following asymptotic bound for all $i\in V$,
\small \begin{align}
\left\|\left[\dfrac{1}{n}\sum_{i=1}^{n}g(\widehat{x}_{i}(T)) \right]_{+}\right\|^{2}_{2}=\mathcal{O}(\eta),
 \end{align}\normalsize
Furthermore, if the optimal solution $x_{\ast}$ is strictly feasible at an optimal point $g(x_{\ast})\prec 0$, we have
\small \begin{align}
\left\|\left[\dfrac{1}{n}\sum_{i=1}^{n}g(\widehat{x}_{i}(T)) \right]_{+}\right\|^{2}_{2}=\mathcal{O}\left(\dfrac{\eta \log(T)}{\sqrt{T}}\right).
 \end{align}\normalsize
\end{theorem}

\section{Numerical Experiments}
\label{Sec:Numerical_Simulations}

In this section, we report the numerical simulations studying the convergence of the regularized primal-dual method for distributed regression on synthetic data. To demonstrate the performance of Algorithm \ref{CHalgorithm}, we consider two examples of smooth and non-smooth classifiers.
\begin{itemize}
\item \textbf{Smooth case}: we consider a logistic loss function with a norm constraint as well as a set of box constraints
\begin{subequations}
\begin{align}
\label{Eq:SVM}
&\min_{x\in \real^{d}} f(x)\coloneqq \dfrac{1}{n}\sum_{i=1}^{n}\log(1+\exp(b_{i}\langle a_{i},x\rangle) )\\ \nonumber
&\text{subject to} \quad g_{k}(x)=-l-x_{k}\leq 0,  \\ \nonumber
&\hspace{17mm} g_{k+d}(x)=x_{k}-u\leq 0, \quad k=1,\cdots,d, \\  \label{Eq:SVM_cons1}
&\hspace{17mm} \|x\|_{2}\leq 1,
 \end{align}\normalsize
\end{subequations}

where $(a_{i},b_{i})\in \real^{d}\times \{-1,+1\}$. 

\item \textbf{Non-smooth case}: we consider a hinge loss function with a norm constraint as well as a set of box constraints
\begin{subequations}
\begin{align}
\label{Eq:SVM1}
&\min_{x\in\real^{d}} f(x)\coloneqq \dfrac{1}{n}\sum_{i=1}^{n}\left[1-b_{i}\langle a_{i},x\rangle\right]_{+}\\ \nonumber
&\text{subject to} \quad g_{k}(x)=-l-x_{k}\leq 0,  \\  \nonumber
&\hspace{17mm} g_{k+d}(x)=x_{k}-u\leq 0, \quad k=1,\cdots,d\\   \label{Eq:SVM_cons2}
&\hspace{17mm} \|x\|_{2}\leq 1,
 \end{align}\normalsize
\end{subequations}
where $(a_{i},b_{i})\in \real^{d}\times \{-1,+1\}$. 

\end{itemize}

The optimization problems of the type \eqref{Eq:SVM}-\eqref{Eq:SVM_cons1} and \eqref{Eq:SVM1}-\eqref{Eq:SVM_cons2} are common in the context of classification in supervised learning, where $\{(a_{1},b_{1}),\cdots,(a_{n},b_{n})\}$ is the set of $n$ training data such that $a_{i}$ is the feature vector (a.k.a. the explanatory variables in the regression), and $b_{i}$ is its associated label. In the case of the logistic classifier, to make a prediction given a new vector $a$, the classifier outputs $b=\pm 1$ with the probability of $\Prob(b=\pm 1 | a,x)=\dfrac{1}{1+\exp(\pm \langle x,a\rangle )}$. In the case of the hinge loss function, the goal is to obtain a linear classifier of the form $a\mapsto \text{sign}(\langle a,x \rangle)$ for some vector $x\in \real^{d}$.

 In our simulations with the logistic classifier, we generate $a_{i}$ from a uniform distribution on the unit sphere. We then choose a random vector from Gaussian distribution $w\sim \mathsf{N}(0,I_{d\times d})$ and generate the labels $b_{i}\sim {\tt{Bernoulli}} (p)$, where $p=\dfrac{1}{1+\exp(\langle w,a_{i}\rangle) }$. It is straightforward to verify that $L=\max_{i=1,2,\cdots,n}\|a_{i}\|=1$ and $R=1$. Note that the solution of the optimization problem in eq. \eqref{Eq:SVM} approximates $w$ under the restrictions specified in Eqs. \eqref{Eq:SVM_cons1}. 
We consider vectors of the dimension $d=5$ (thus $m=10$) and study three different network sizes, $n\in \{50,100,200\}$ and two different upper/lower limits $l=u=0.1$. To show that Algorithm \ref{CHalgorithm} works for any initialization, instead of using the origin as the initialization point of Algorithm \ref{CHalgorithm}, we generate a random vector $v\in \mathsf{N}(0,I_{d\times d})$ and then choose $x_{i}(0)=v/\|v\|_{2}$. We also use the stepsize $\alpha(t)={R}/{\sqrt{t+1}}$ in all simulations, where here $R=1$.

For a graph $G$ of $n$ nodes, let $\varepsilon_{G}(t;n)$ denotes the maximum relative error of the network, \textit{i.e.}, $\varepsilon_{G}(t;n) \coloneqq \max_{i=1,2,\cdots,n}\left|\dfrac{f(\widehat{x}_{i}(t))-f(x_{\ast})}{f(\widehat{x}_{i}(0))-f(x_{\ast})}\right|$ for every node in the graph $i\in V$. Further, we define $\delta_{G}(t;n)\coloneqq \max_{i=1,2,\cdots,n}\|g(\widehat{x}_{i}(t))\|/\|g(\widehat{x}_{i}(0))\|$ as the maximum constraint violation among all the nodes in the network. In the case of the centeralized the primal-dual method, we similarly use $\varepsilon (t,n)$ and $\delta(t;n)$ to denote the relative error gap and the constraint violation, respectively. In our simulations, we use MATLAB convex programming toolbox ${\tt{CVX}}$ \cite{grant2008cvx} to compute $f(x_{\ast})$.

To investigate the performance of Algorithm \ref{CHalgorithm} on different networks, we consider random and structured graphs in our simulations, namely (a): Watts-Strogatz small-world graph model \cite{watts1998collective}, (b) Erd\"{o}s-R\'{e}yni random graph \cite{bollobas1998random}, (c) unwrapped 8-connected neighbors lattice, (d) two-clique graph (barbell graph).  See Fig. \ref{Fig:1}.

The Watts-Strogatz model is a mathematical model to generate random graphs with small-world properties, \textit{i.e.}, graphs that are highly clustered locally (like regular lattices) and with a small separation globally. Social networks is an example where each person is only five or six people away from anyone else. Watts-Strogatz model has two structural features, namely the clustering and the average path length. These features are captured by two parameters, namely the mean degree $K$, and a parameter $\vartheta$ that interpolates between a lattice $(\vartheta=0)$ and a random graph $(\vartheta=1)$.

\hspace*{-10mm}\begin{figure*}[t!]
     \begin{center}
        \subfigure{
        \includegraphics[trim={2cm 1.5cm 1.5cm .2cm}, width=.2\linewidth]{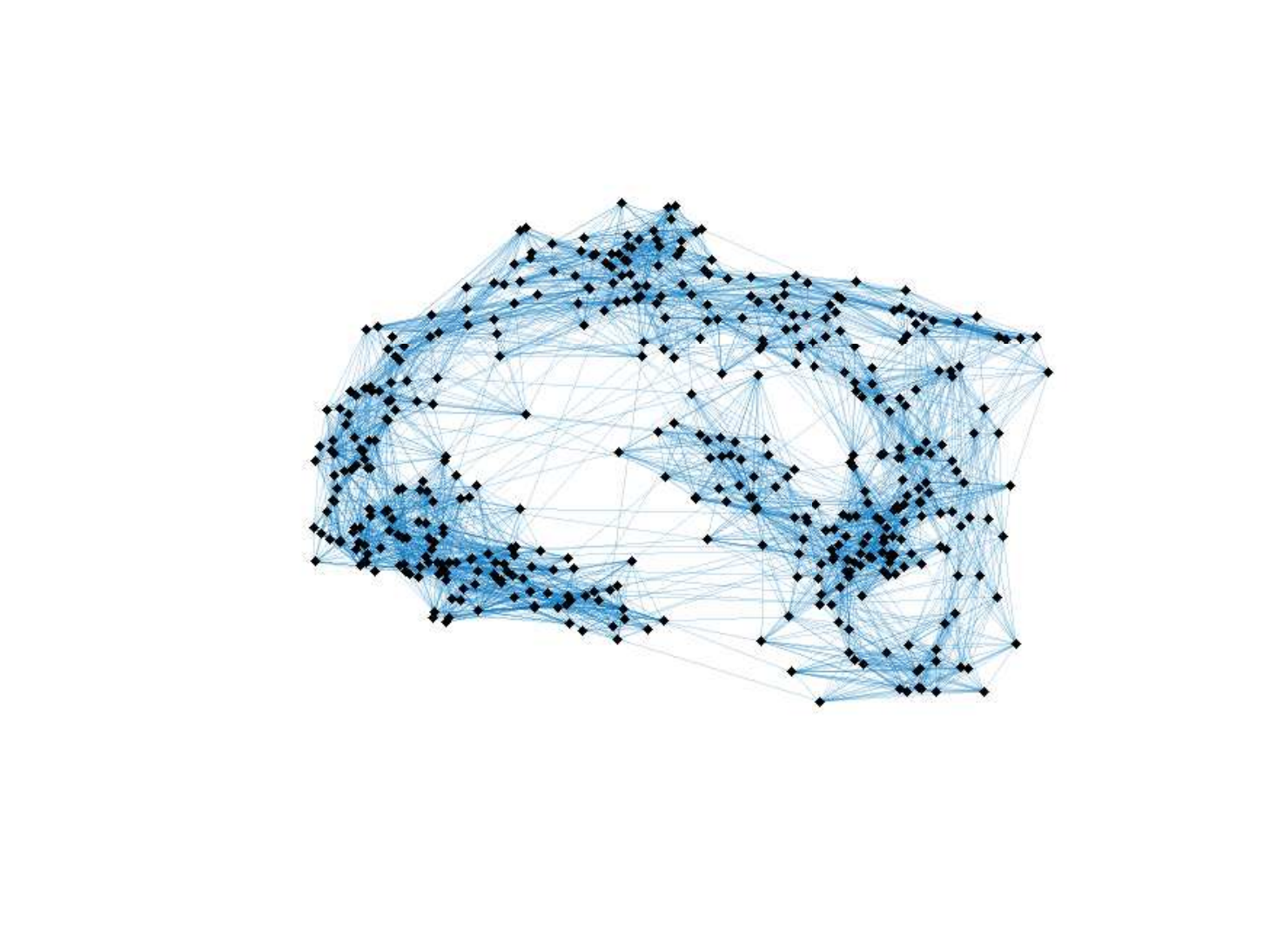}
        \includegraphics[trim={2cm .75cm 1.5cm .2cm}, width=.2\linewidth]{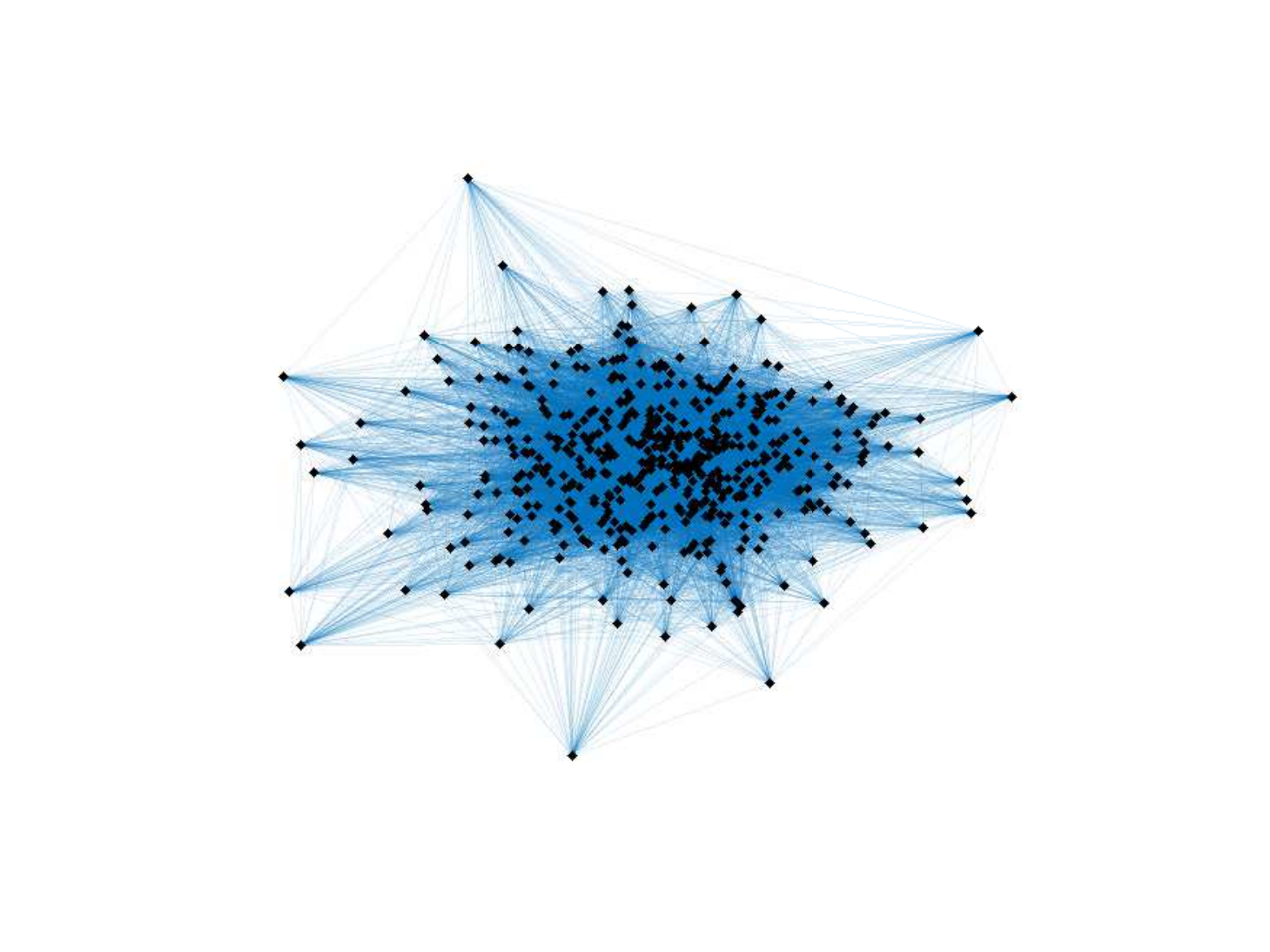}
        \includegraphics[trim={2cm .75cm 1.5cm .2cm}, width=.2\linewidth]{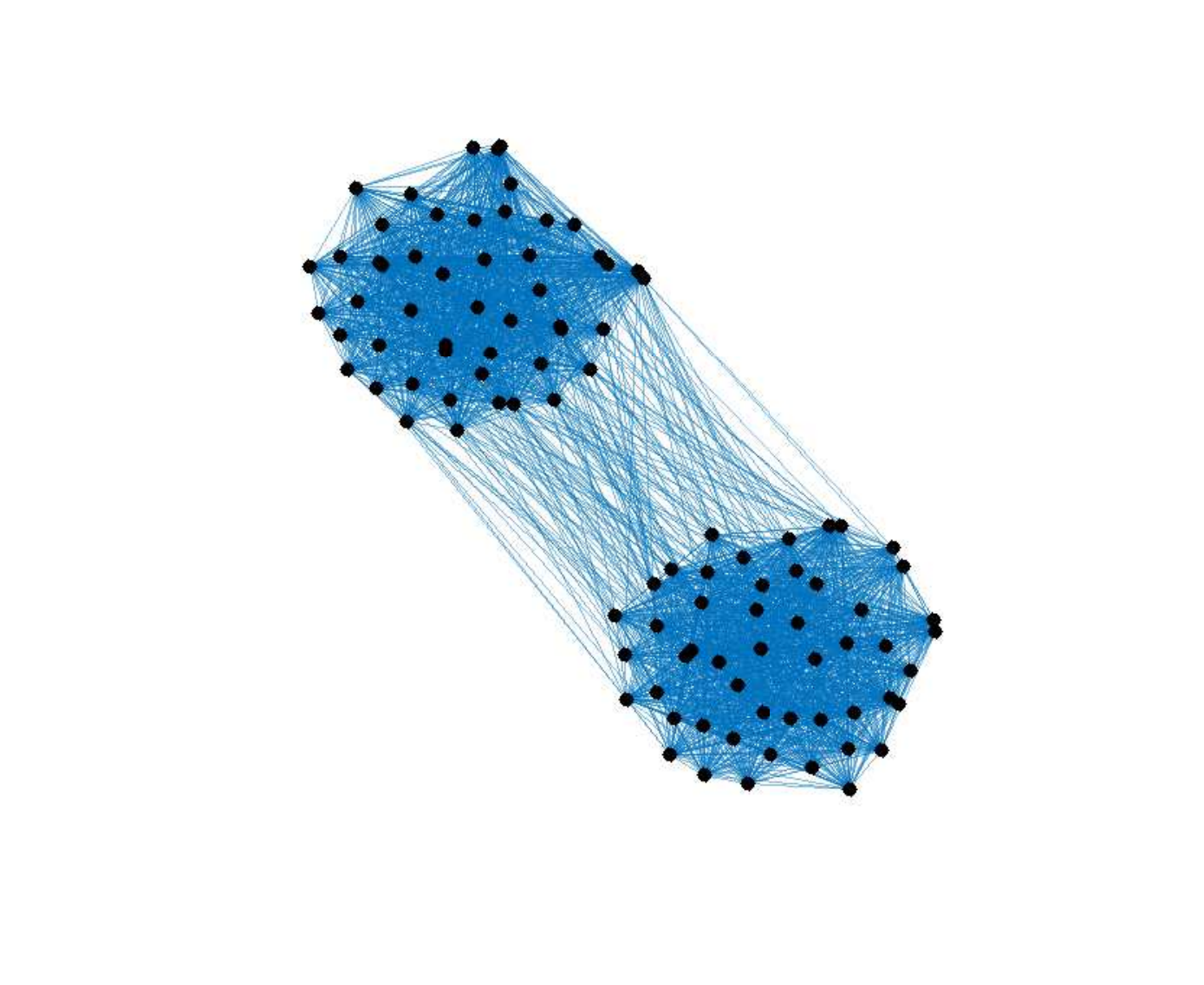} 
        \includegraphics[trim={2cm 1.5cm 2cm .2cm}, width=.2\linewidth]{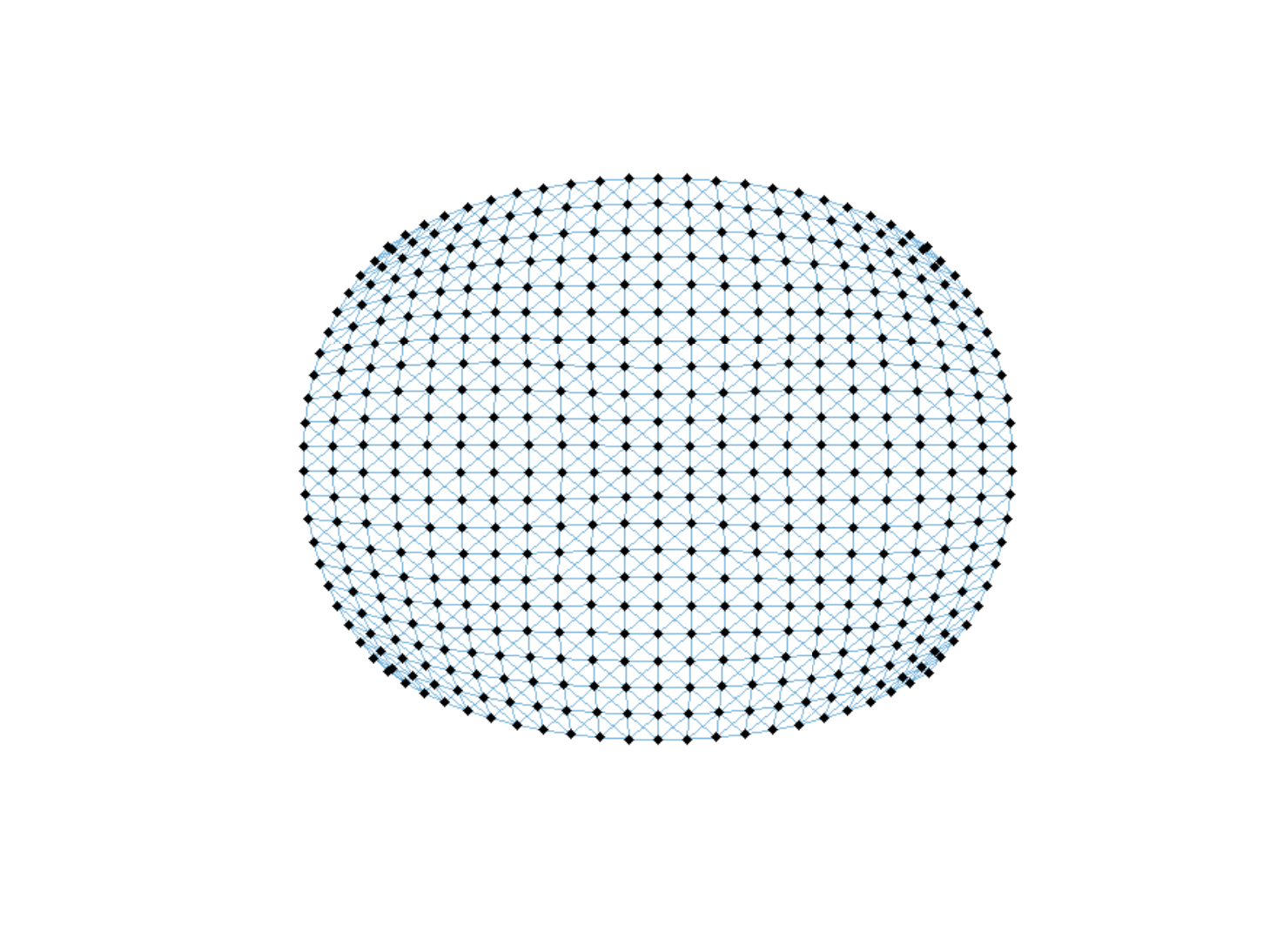}                    
        }
        \subfigure{
        \includegraphics[trim={.6cm .2cm .4cm .2cm}, width=.2\linewidth]{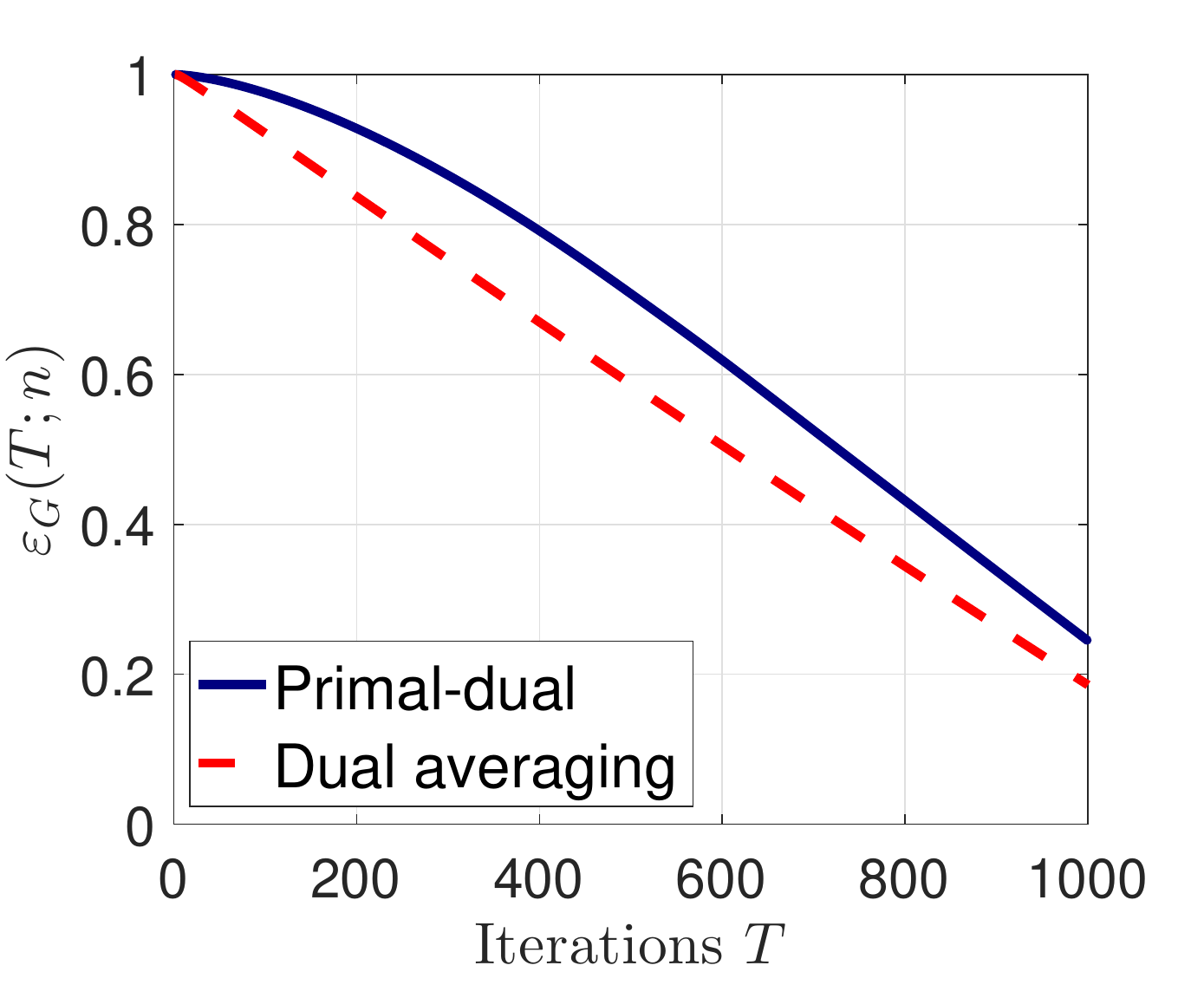} 
        \includegraphics[trim={.6cm .2cm .4cm .2cm}, width=.2\linewidth]{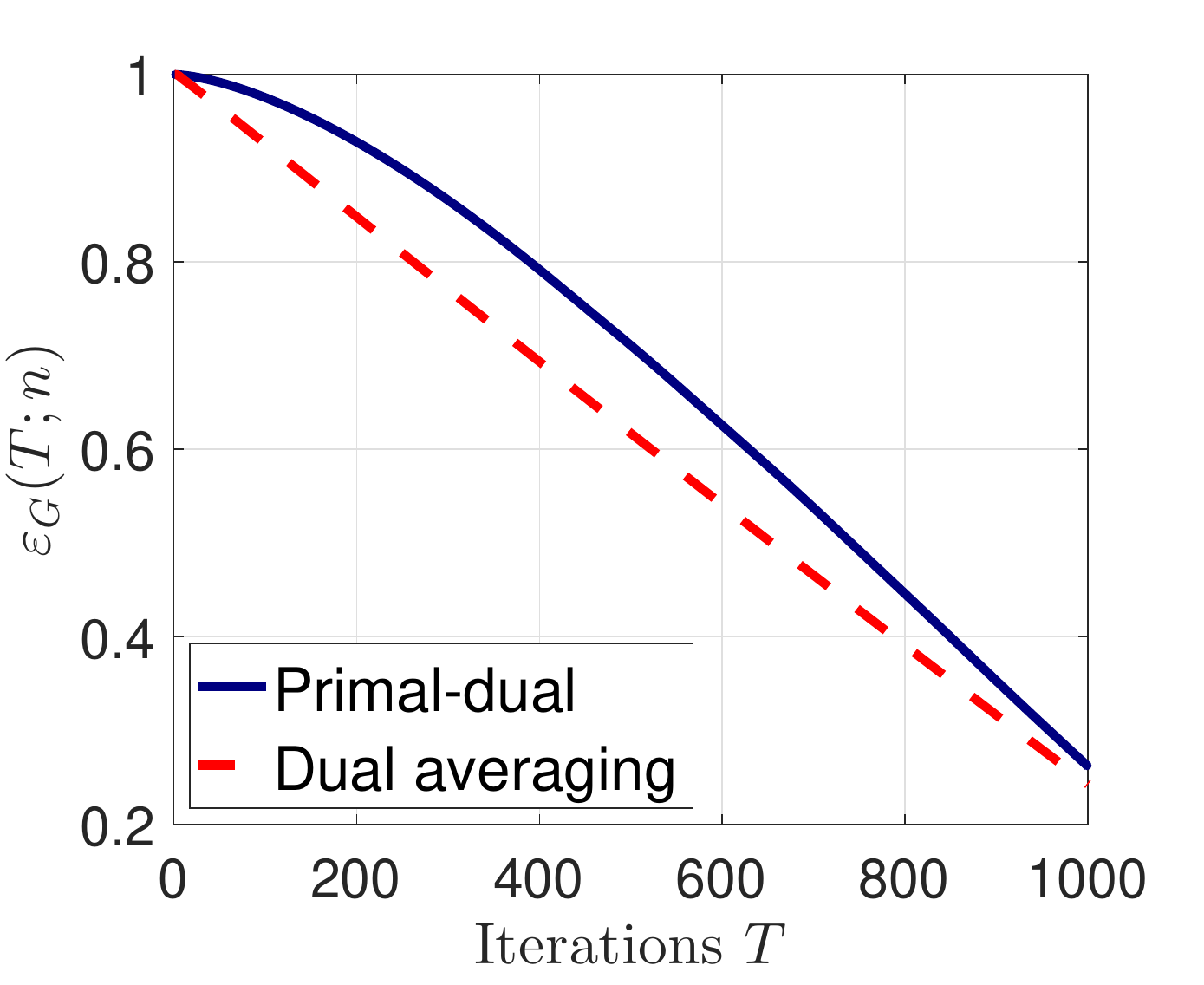}
        \includegraphics[trim={.6cm .2cm .4cm .2cm}, width=.2\linewidth]{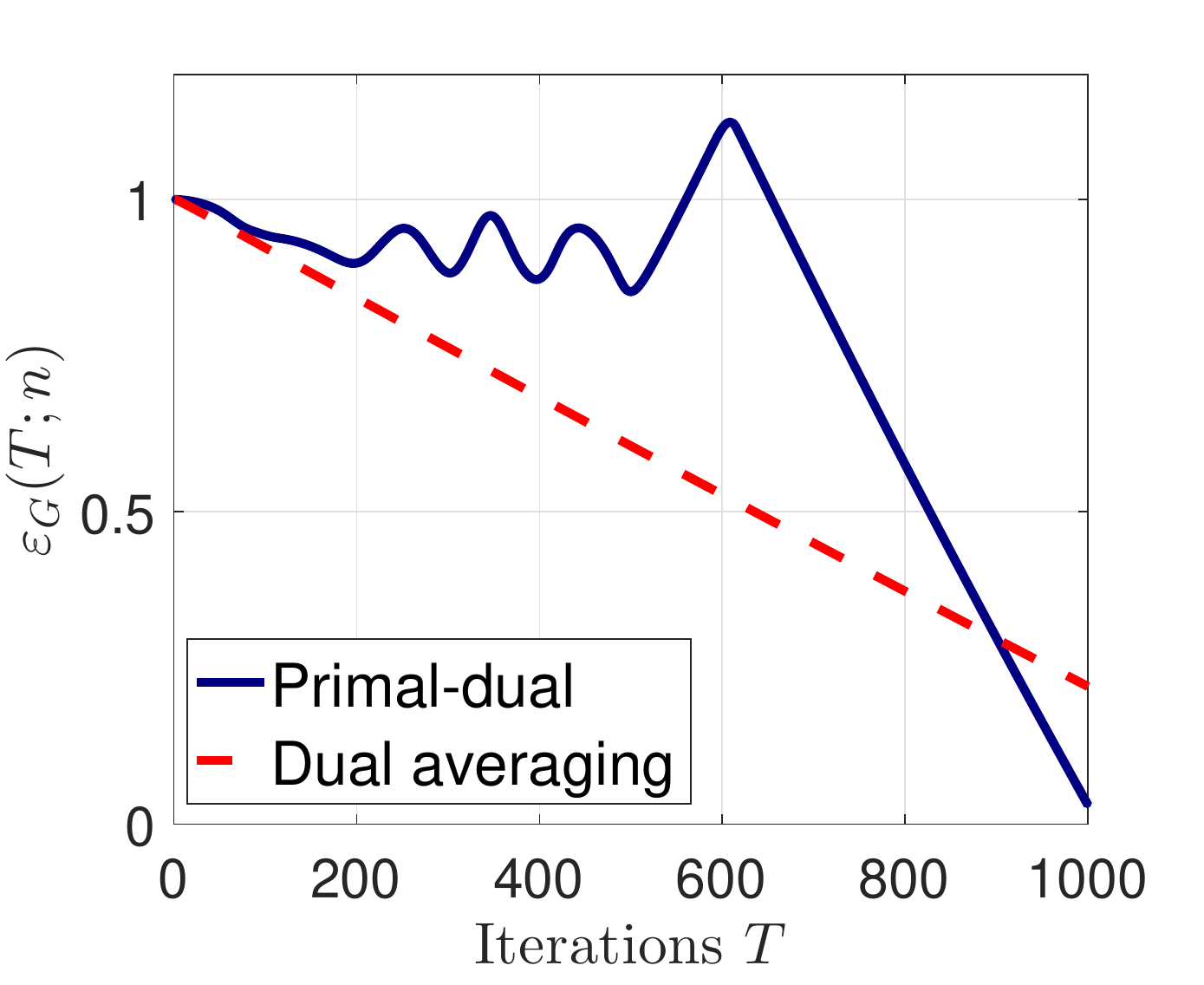}  
        \includegraphics[trim={.6cm .2cm .4cm .2cm}, width=.2\linewidth]{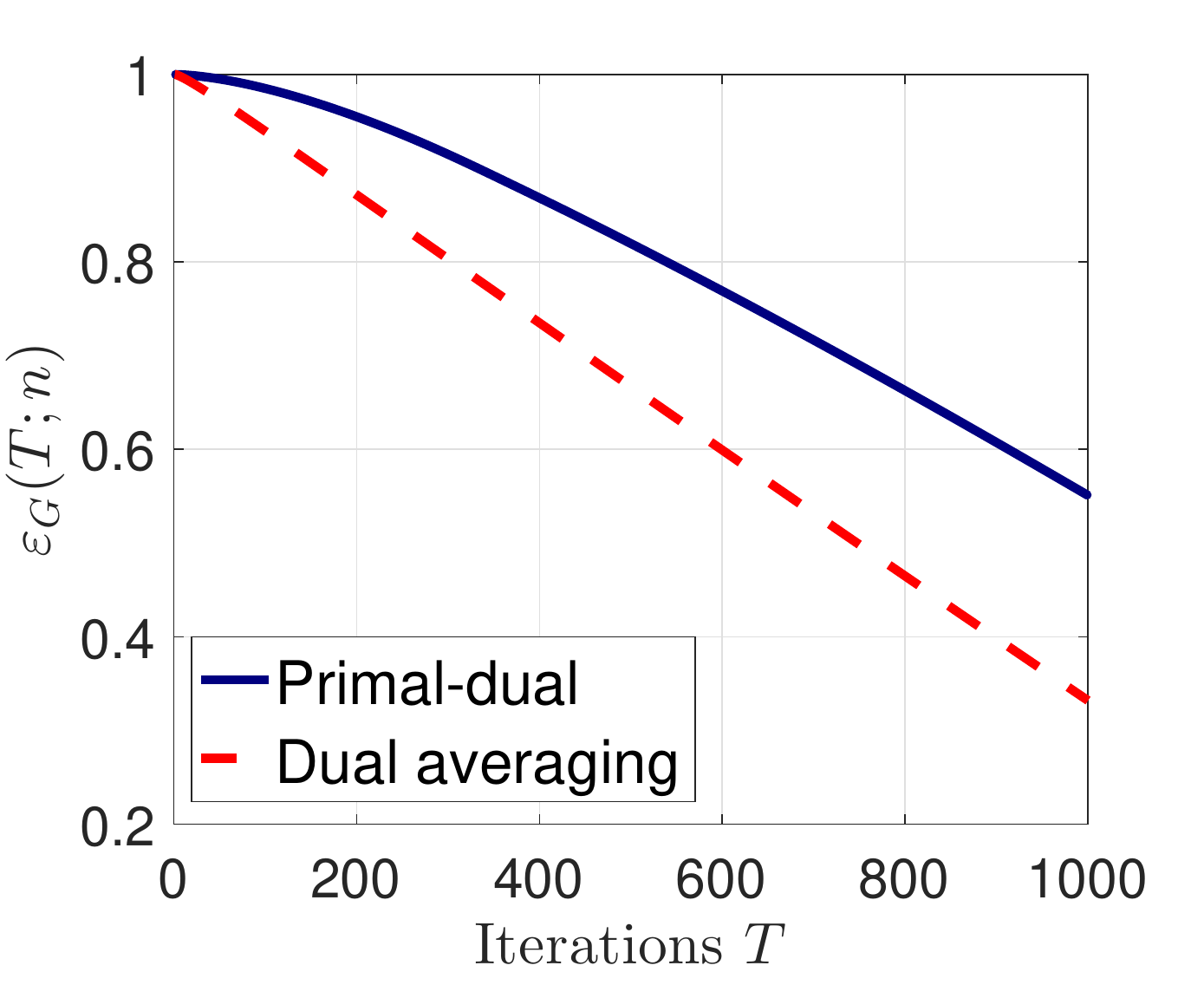}   
        }
        \subfigure{             
        \includegraphics[trim={.6cm .2cm .4cm .2cm}, width=.2\linewidth]{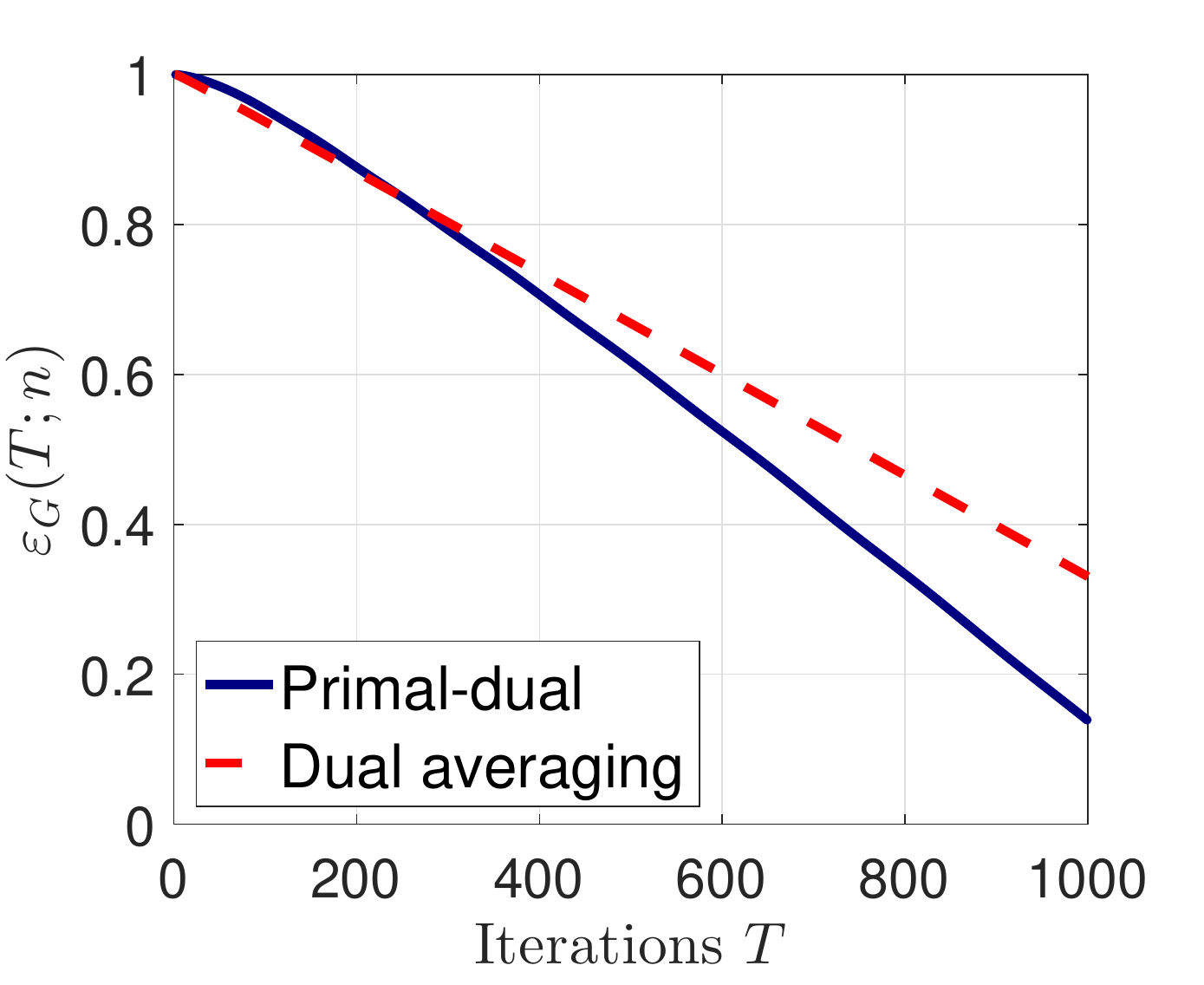}
         \includegraphics[trim={.6cm .2cm .4cm .2cm}, width=.2\linewidth]{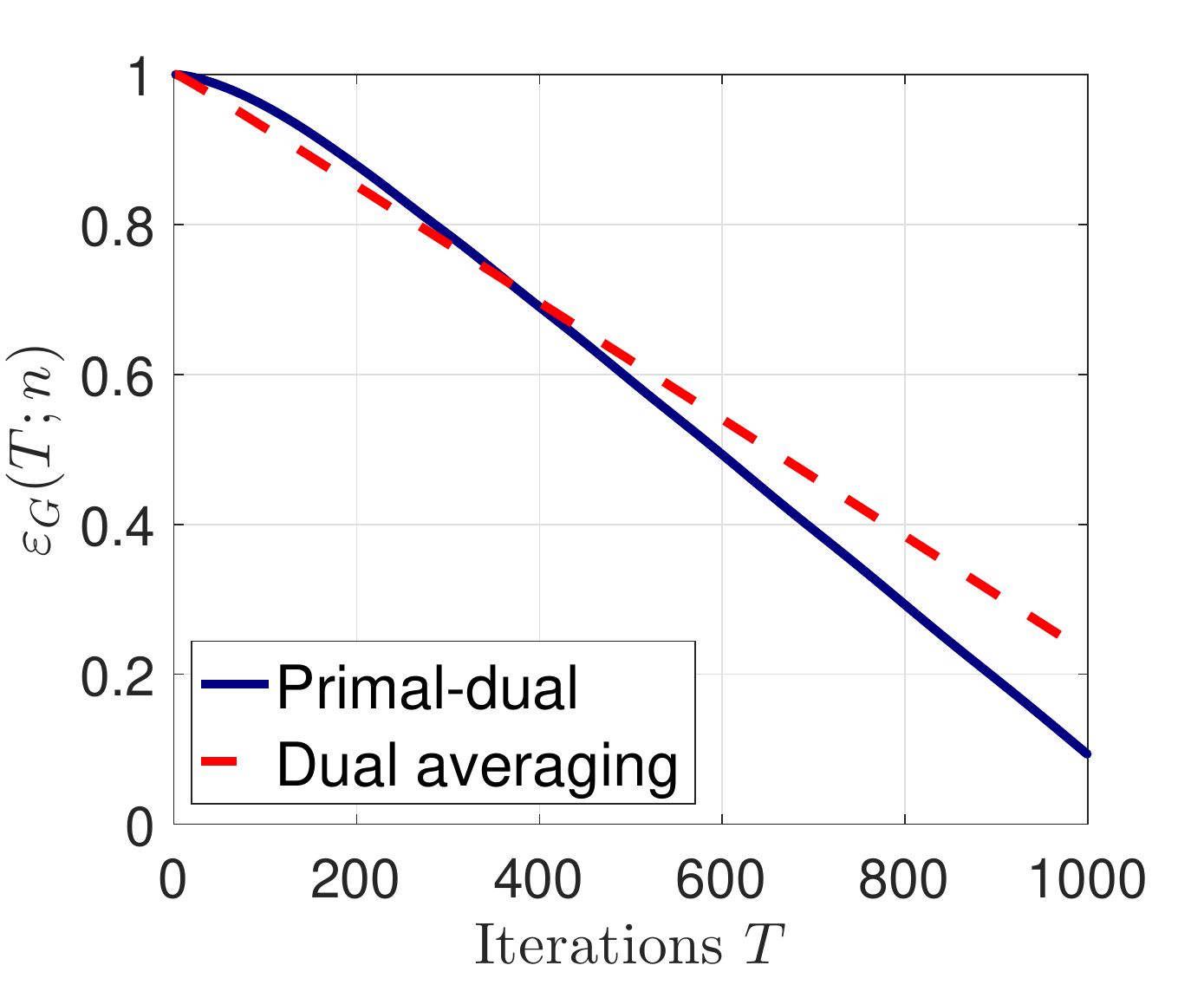}
        \includegraphics[trim={.6cm .2cm .4cm .2cm}, width=.2\linewidth]{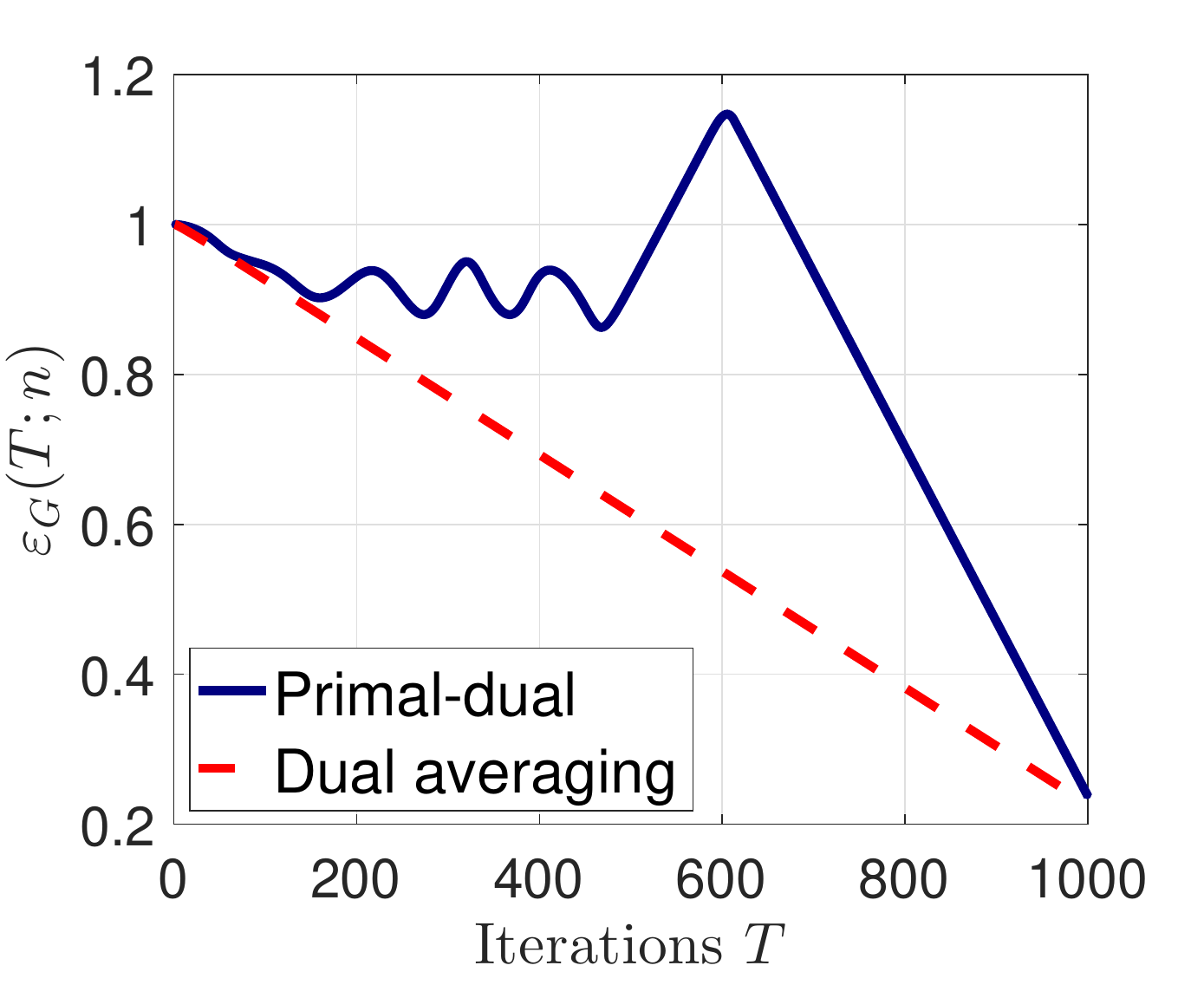}  
        \includegraphics[trim={.6cm .2cm .4cm .2cm}, width=.2\linewidth]{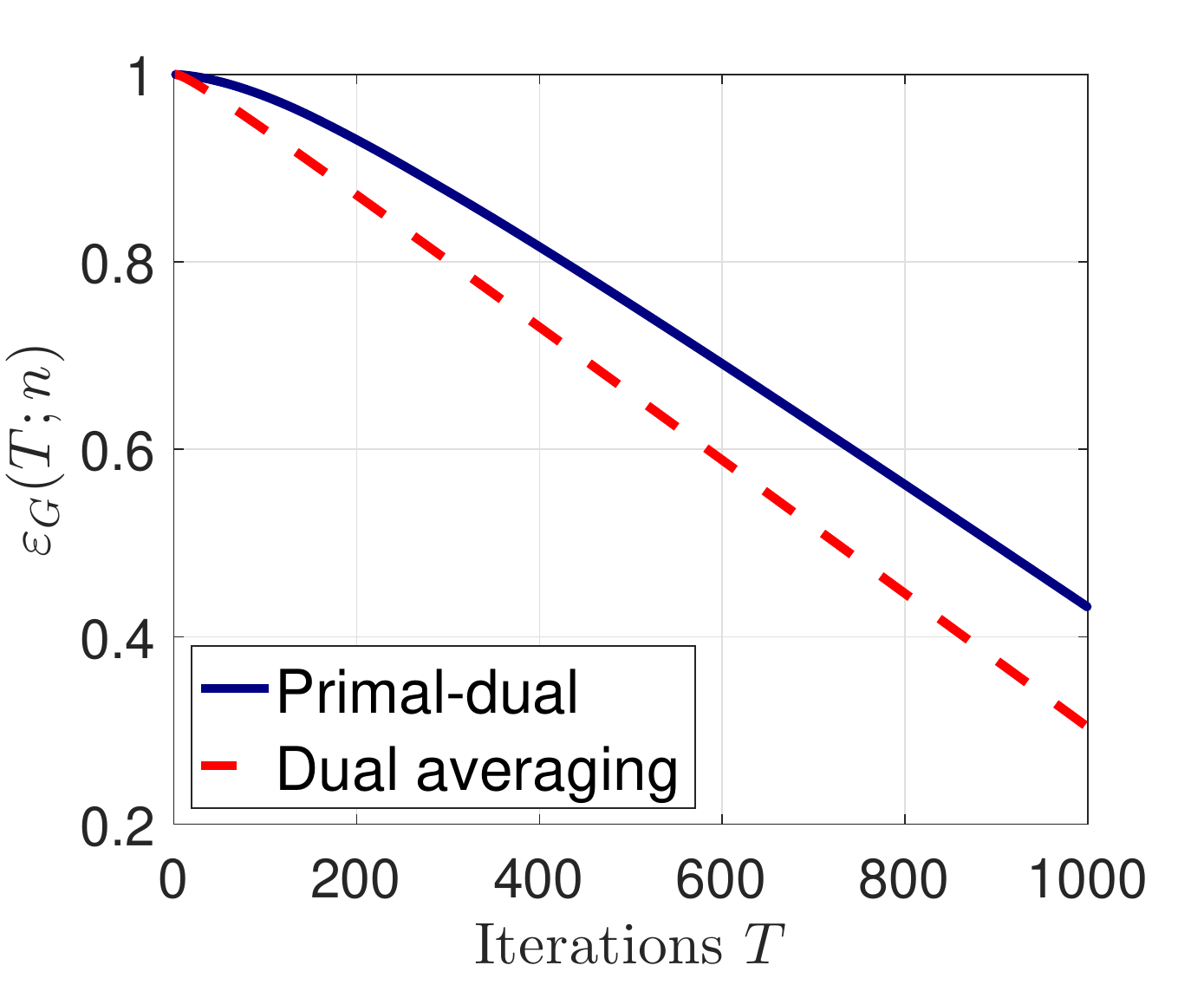}    
        }   
    \end{center}
\caption{\footnotesize{Illustration of three graph models used in simulations and the corresponding average maximum relative error gap $\varepsilon_{G}(T;n)$ with $n=100$ ($l=u=1$, $d=5$) for logistic loss function (middle row) and hinge loss function (bottom row). Top to bottom: Watts-Strogatz graph with $K=20$ and $\vartheta=0.02$, Erd\"{o}s-R\'{e}yni random graph with $p=0.06$, unwrapped 8-connected neighbors lattice, two-clique (barbell) graph.} }
\label{Fig:1}
\end{figure*}

In the Erd\"{o}s-R\'{e}yni random graph, the edge between each pair of nodes is included in the graph with the probability $p$ independent from every other edges. Note that the Watts-Strogatz small-world graph model reduces to the Erd\"{o}s-R\'{e}yni random graph model when $\vartheta=1$. To aggregate the information of neighbors, we use the weight matrix $W$ according to the Lazy metropolis matrix in \eqref{Eq:LazyasIam}.
In the unwrapped graph, each node is adjacent to $8$ neighbors. Lastly, in the barbell graph, we have two cliques of size $n/2$ which are connected by a few links. 

Figure \ref{Fig:1} shows the maximum error $\varepsilon_{G}(T;n)$ for the distributed dual averaging algorithm \cite{duchi2012dual} (dotted lines), and the distributed deterministic primal-dual algorithm (solid lines). For both algorithms, we used the normalized graph Laplacian as the weight matrix (cf. Section \ref{Assumptions}). 
In the distributed dual averaging algorithm, the stepsize is given by $\alpha(t)={R\sqrt{1-\sigma_{2}(W)}\over 4L\sqrt{t+1}}$. For the primal-dual algorithm, the stepsize $\alpha(t)={R\over \sqrt{t+1}}$ is independent of the spectral gap. It is clear from Figure \ref{Fig:1} that on the barbell graph as well as on the lattice, the convergence of both algorithms is slow. This is due to the fact that the spectral gap $1-\sigma_{2}(W)$ of both networks is quite small and thus reaching consensus on these networks is more difficult. 

We also observe that the primal-dual algorithm shows an oscillatory behavior on the barbell graph, whereas the dual averaging algorithm does not. This difference is attributed to the choice of stepsizes. In the dual averaging algorithm, the stepsize is modulated by the spectral gap. Therefore, when the spectral gap is very small (as is the case for the barbell graph), the stepsize is small which suppresses the oscillations. In contrast, the stepsize of the primal-dual algorithm is independent of the spectral gap. Notice that due to incorporating the spectral gap in the dual averaging algorithm, the structure of the network must be known a priori by each agent. This requires extra communication at the beginning of the dual averaging algorithm.

Figure \ref{Fig:2} shows the constraint violation as well as the convergence rate in the centralized primal-dual algorithm without regularization and the decentralized regularized primal-dual algorithm with the values $u=l=0.1$. In this simulations, we choose the initial points $(x_{i}(0),\lambda_{i}(0))$ of the distributed primal-dual algorithm randomly from the feasible region (cf. Remark \ref{Remark}). In this particular example, we observe that in the decentralized primal-dual algorithm, the algorithm output $\widehat{x}_{i}(t)$ is almost feasible for all $t$ and $i\in V$. In contrast, in the centralized primal-dual algorithm, the outputs are infeasible. Here, we thus clearly observe that the regularization can mitigate the constraint violation. 

\begin{figure*}[t!]
\centering
        \centering
         \subfigure[]{
         \label{fig:first}
        \includegraphics[trim={.6cm .2cm .5cm .5cm}, width=.23\linewidth]{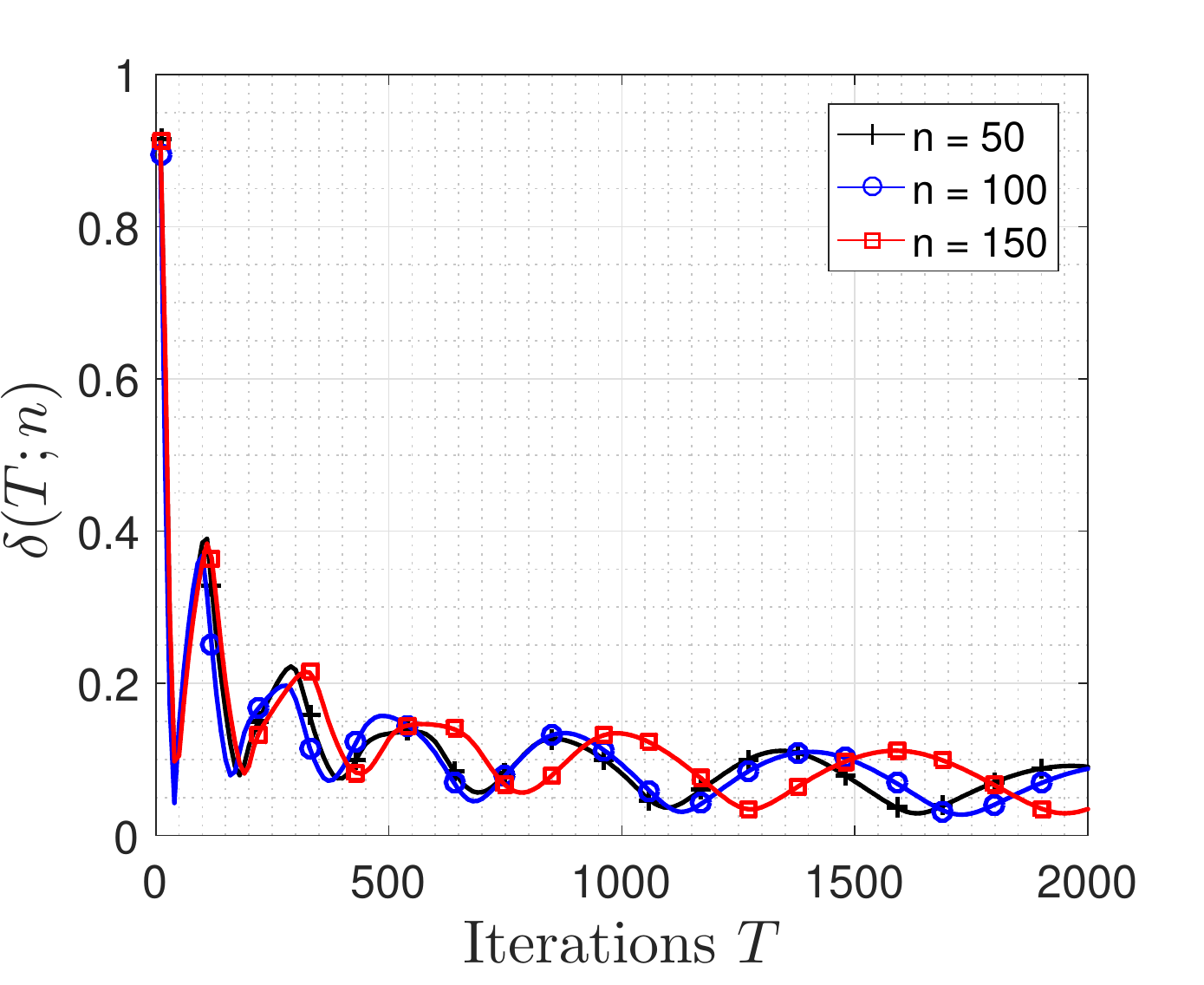} 
        }\hspace{-2mm}
         \subfigure[]{
         \label{fig:second}
        \includegraphics[trim={.6cm .2cm .5cm .5cm}, width=.23\linewidth]{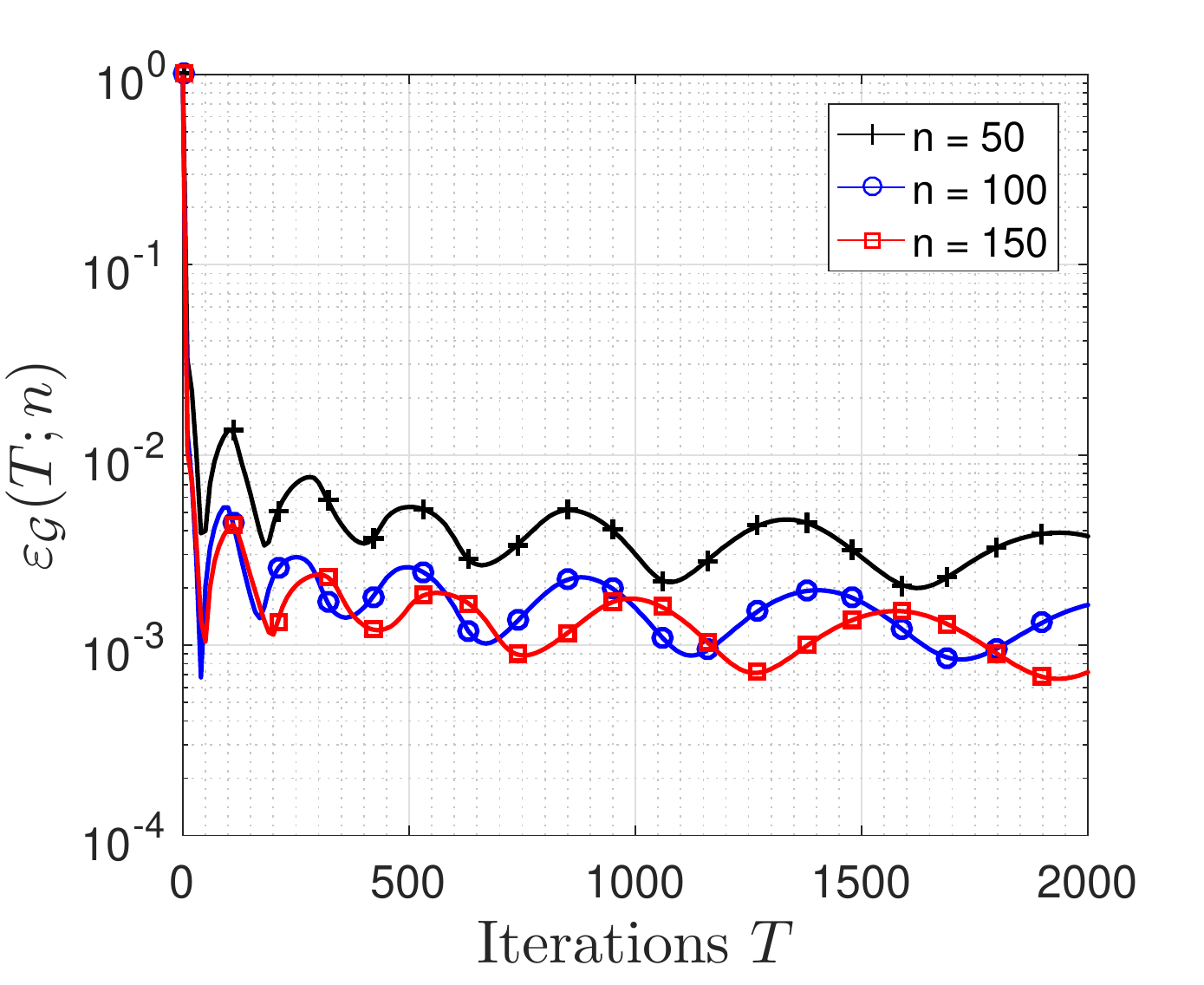}
        }\hspace{-2mm}
         \subfigure[]{
         \label{fig:third}
        \includegraphics[trim={.6cm .2cm .5cm .5cm}, width=.23\linewidth]{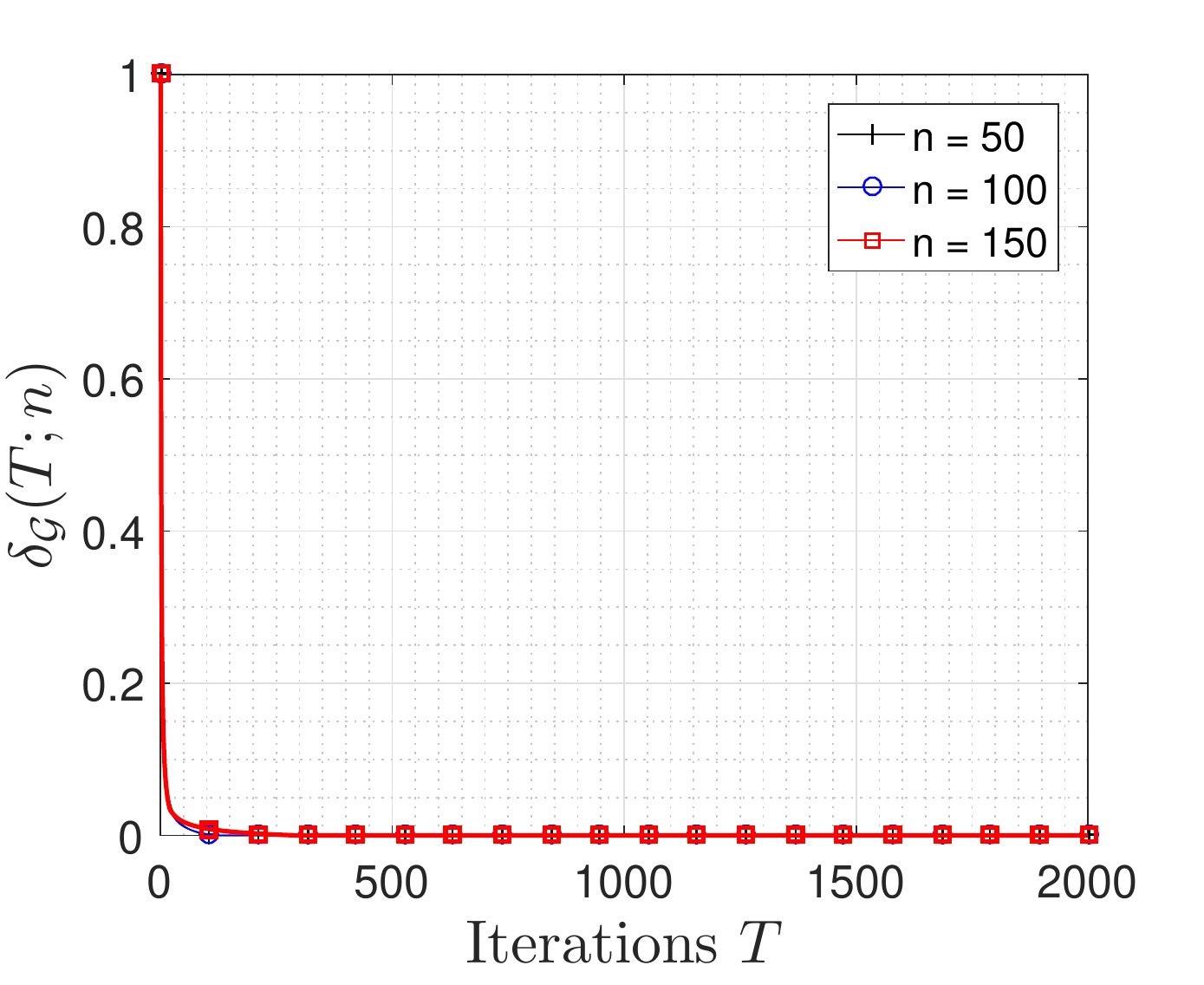}   
        }\hspace{-2mm}
        \subfigure[]{    
        \label{fig:fourth}
        \includegraphics[trim={.6cm .2cm .5cm .5cm}, width=.23\linewidth]{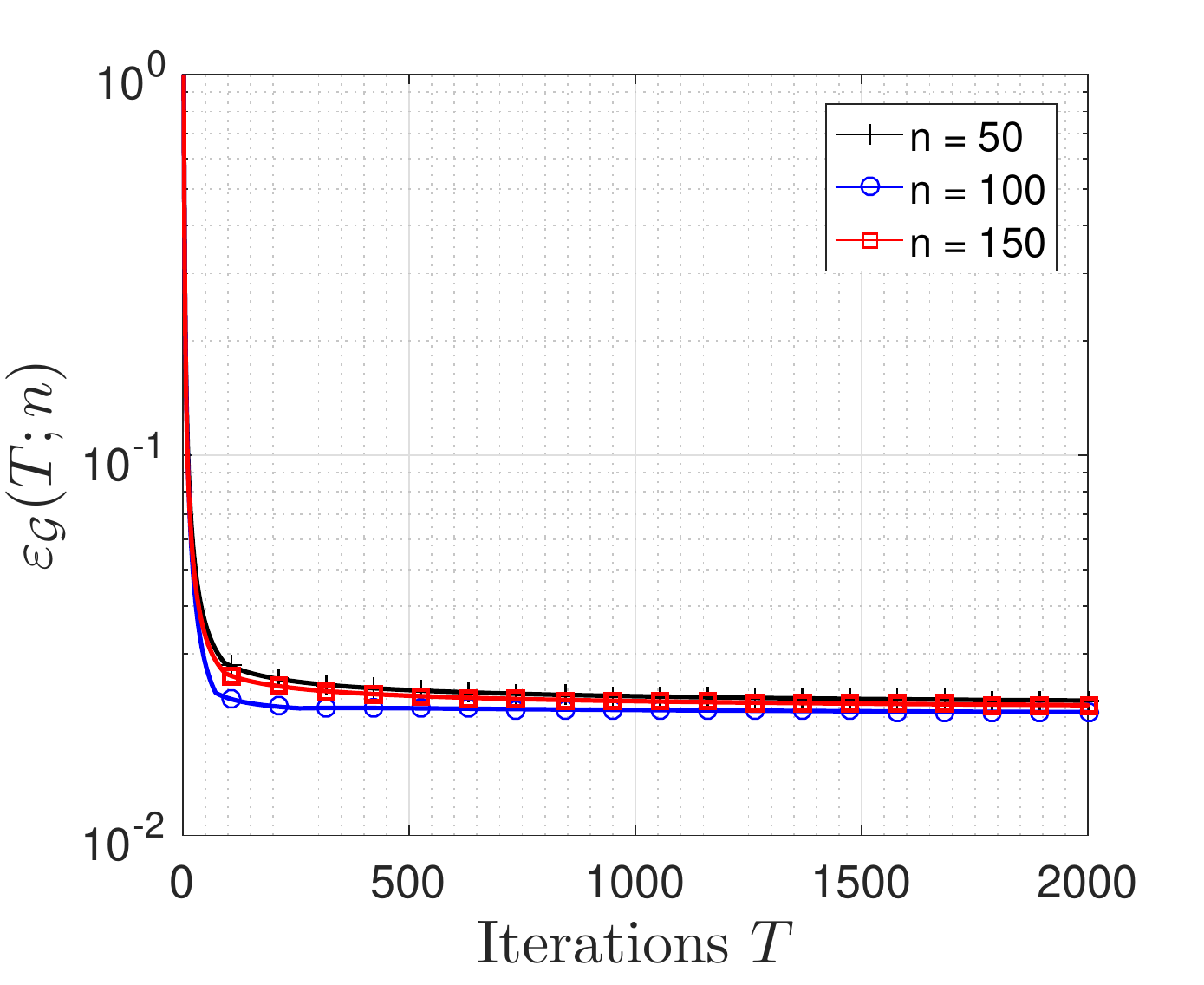}
        }
\caption{\footnotesize{Distributed logistic regression on synthetic data using Watts-Strogratz graph with $K=20$, $\vartheta=0.02$, $\eta=1$, $\alpha(t)= {1\over \sqrt{t+1}}$ and $l=u=0.001$, Panel (a): Constraint violation $\delta(T;n)$ of the centralized PD algorithm without regularization, Panel (b): Convergence rate $\varepsilon(T;n)$ of the centralized PD algorithm without regularization, Panel (c) Constraint violation $\delta_{\mathcal{G}}(T;n)$ of the decentralized PD algorithm, Panel (d): Convergence rate $\varepsilon_{\mathcal{G}}(T;n)$ of the decentralized PD algorithm.}}
\label{Fig:2}
\end{figure*}

\section{Conclusion and Discussion}
\label{Discussion_and_Conclusion}

In this paper, we have studied a distributed regularized primal-dual methods for convex optimization of separable objective functions with inequality constraints. In the proposed distributed methods, the Lagrangian function is regularized with the squared norm of the Lagrangian multipliers. As a result, the norm of Lagrangian multipliers are bounded from above, and consequently, the norm of sub-gradients of the Lagrangian function are also bounded. Using this regularization, we proved a convergence rate for attaining the optimal objective value, and we also presented an asymptotic analysis of the constraint violation of the primal-dual solutions. We showed an interesting trade-off between the convergence rate of the algorithm and the constraint violation bound. In particular, by choosing a large regularization parameter, we can achieve a fast convergence rate. However, a large regularization parameter increases the constraint violation of the primal-dual estimations. Interestingly, when the constraints are satisfied strictly at an optimal solution, such a trade-off does not exists.

We also proposed and analyzed a distributed stochastic primal-dual algorithm. At each step of the stochastic algorithm, one inequality constraint is selected randomly, and its sub-gradient is computed. Therefore, for optimization problems with many inequality constraints, a distributed stochastic primal-dual algorithm is more efficient compared to the deterministic algorithm.
 
As a future research, it is interesting to have a comprehensive analysis of the distributed penalty/barrier function methods, and compare their convergence rates with the distributed primal-dual algorithms we developed in this paper.


\appendix
\section{Proofs of Main Results}
\label{App:Proofs_of_Main_Results}

\subsection{Proof of Lemma \ref{Lem:1}}
\label{Proof of Proposition 1}
The general plan to prove Lemma \ref{Lem:1} is to establish a recursion for the primal variables using the update rule in Steps 3-4 of  Algorithm \ref{CHalgorithm}. For the primal point of the saddle point $(x_{\ast},\lambda_{\ast})\in \mathcal{X}\times \real_{+}\subseteq \ball_{d}(R)\times \real_{+}$ in the minimax problem \eqref{Eq:min-max-regularized}, the following equalities hold,
\begin{align}
\nonumber
&\|x_{i}(t+1)-x_{\ast}\|^{2}
\stackrel{\rm{(a)}}{=} \left\|\Pi_{\ball_{d}(R)}\left(\sum_{j=1}^{n}[W]_{ij}y_{i}(t)\right)-x_{\ast}\right\|^{2}\\ \label{Eq:Yesterday}
&\stackrel{\rm{(b)}}{=} \left\|\Pi_{\ball_{d}(R)}\left(\sum_{j=1}^{n}[W]_{ij}y_{i}(t)\right)-\Pi_{\ball_{d}(R)}(x_{\ast})\right\|^{2},
\end{align}
where ${\rm{(a)}}$ follows by substituting $x_{i}(t+1)$ from \eqref{Eq:That_is_the_way} in Algorithm \ref{CHalgorithm}, and ${\rm{(b)}}$ follows since the optimal point $x_{\ast}$ is in the interior of the Euclidean ball $x_{\ast}\in \mathcal{X}\subset \ball_{d}(R)$ and therefore the projection of the optimal point is the optimal point itself, \textit{i.e.}, $\Pi_{\ball_{d}(R)}(x_{\ast})=x_{\ast}$. 
Due to the non-expansive property of projection (cf. \cite[Chapter III.3]{hiriart1996convex}), the following inequality holds
\begin{align*}
&\left\|\Pi_{\ball_{d}(R)}\left(\sum_{j=1}^{n}[W]_{ij}y_{i}(t)\right)-\Pi_{\ball_{d}(R)}(x_{\ast})\right\|\\ &\leq \left\|\sum_{j=1}^{n}[W]_{ij}y_{i}(t)-x_{\ast}\right\|.
\end{align*}
Replacing the preceding inequality in eq. \eqref{Eq:Yesterday} yields
\begin{align}
\label{Eq:And_hence}
&\|x_{i}(t+1)-x_{\ast}\|^{2}\\ \nonumber
&\leq \left\|\sum_{j=1}^{n}[W]_{ij}y_{i}(t)-x_{\ast}\right\|^{2}\\ \nonumber
&\stackrel{(\rm{c})}{\leq}\sum_{j=1}^{n}[W]_{ij}\|x_{j}(t)-x_{\ast}-\alpha(t)\nabla_{x}L_{j}(x_{j}(t),\lambda_{j}(t))\|^{2}\\  \nonumber
&=\sum_{j=1}^{n}[W]_{ij}\Big(\|x_{j}(t)-x_{\ast}\|^{2}\\ \nonumber &\hspace{20mm} -2\alpha(t)\langle\nabla_{x}L_{j}(x_{j}(t),\lambda_{j}(t)),x_{j}(t)-x_{\ast}\rangle \\ \nonumber &\hspace{20mm} +\alpha^{2}(t)\|\nabla_{x}L_{j}(x_{j}(t),\lambda_{j}(t))\|^{2}\Big),
\end{align}
where $\rm{(c)}$ follows from the convexity of the squared norm and substituting for $y_{i}(t)$ from eq. \eqref{Eq:aux_1} in Step 3 of Algorithm \ref{CHalgorithm}. We compute the sum of both sides of inequality \eqref{Eq:And_hence} over all $i=1,2,\cdots,n$,
\begin{align}
\label{Eq:rec_W}
\sum_{i=1}^{n}\|x_{i}(t+1)-x_{\ast}\|^{2}&\leq 
\sum_{i=1}^{n}\Big(\|x_{i}(t)-x_{\ast}\|^{2}\\ \nonumber &-2\alpha(t)\langle\nabla_{x}L_{i}(x_{i}(t),\lambda_{i}(t)),x_{i}(t)-x_{\ast}\rangle\\ \nonumber
&+\alpha^{2}(t)\|\nabla_{x}L_{i}(x_{i}(t),\lambda_{i}(t))\|^{2} \Big),
 \end{align}\normalsize 
where we used the fact that $W$ is a doubly stochastic matrix by Assumption \ref{Assumption:Weight_matrix}, and therefore we have $\sum_{i=1}^{n}[W]_{ij}=1$.

From the recursion \eqref{Eq:rec_W} with the initial point $x_{i}(0)=0,i\in [n]$, we derive
\begin{align}
\label{Eq:preceding}
\sum_{i=1}^{n}&\|x_{i}(T)-x_{\ast}\|^{2}\leq n\|x_{\ast}\|^{2}\\  \nonumber & -\sum_{t=0}^{T-1}\sum_{i=1}^{n}2\alpha(t)\langle\nabla_{x}L_{i}(x_{i}(t),\lambda_{i}(t)),x_{i}(t)-x_{\ast}\rangle\\ \nonumber
&+\sum_{t=0}^{T-1}\sum_{i=1}^{n}\alpha^{2}(t)\|\nabla_{x}L_{i}(x_{i}(t),\lambda_{i}(t))\|^{2}.
 \end{align}\normalsize
Since the left hand side of eq. \eqref{Eq:preceding} is non-negative, we further obtain the following inequality
\small \begin{align}
\label{Eq:Comb01}
&\sum_{t=0}^{T-1}\sum_{i=1}^{n}2\alpha(t)\langle\nabla_{x}L_{i}(x_{i}(t),\lambda_{i}(t)),x_{i}(t)-x_{\ast}\rangle\\ \nonumber &\leq  n\|x_{\ast}\|^{2}+\sum_{t=0}^{T-1}\sum_{i=1}^{n}\alpha^{2}(t)\|\nabla_{x}L_{i}(x_{i}(t),\lambda_{i}(t))\|^{2}.
 \end{align}\normalsize

Now, recall from Assumption \ref{Assumption:Lipschitz Functions} that $f_{i}(\cdot)$ and $g_{k}(\cdot)$ are convex functions  for all $i\in [n]$ and $k\in [m]$. Therefore, the Lagrangian function $L_{i}(\cdot,\lambda_{i})$ is convex, \textit{i.e.},
\begin{align}
\nonumber
L_{i}(x_{i}(t),\lambda_{i}(t))&-L_{i}(x_{\ast},\lambda_{i}(t))\\ \label{Eq:011}
&\leq  \langle \nabla_{x}L_{i}(x_{i}(t),\lambda_{i}(t)), x_{i}(t)-x_{\ast}\rangle.
 \end{align}\normalsize
Substituting the inequality \eqref{Eq:011} in eq. \eqref{Eq:Comb01} gives us
\begin{align}
\nonumber
&\sum_{t=0}^{T-1}\sum_{i=1}^{n}2\alpha(t)\Big(L_{i}(x_{i}(t),\lambda_{i}(t))-L_{i}(x_{\ast},\lambda_{i}(t))\Big)\\
\label{Eq:after_convex}
&\leq n\|x_{\ast}\|^{2}+\sum_{t=0}^{T-1}\sum_{i=1}^{n}\alpha^{2}(t)\|\nabla_{x}L_{i}(x_{i}(t),\lambda_{i}(t))\|^{2}.
 \end{align}\normalsize
 
After expanding the Lagrangian function on the left hand of eq. \eqref{Eq:after_convex}, we derive
\begin{align*}
&\sum_{t=0}^{T-1}\sum_{i=1}^{n} 2\alpha(t)\Big(f_{i}(x_{i}(t))-f_{i}(x_{\ast})\Big)\\ &+\sum_{t=1}^{T-1}\sum_{i=1}^{n}2\alpha(t)\Big(\langle g(x_{i}(t)),\lambda_{i}(t) \rangle- \langle g(x_{\ast}),\lambda_{i}(t)\rangle\Big)\\ &\leq \text{r.h.s. of eq. \eqref{Eq:after_convex}}.
\end{align*}
Since $x_{\ast}\in \mathcal{X}$ is a saddle point, it must satisfies the inequality constraints, \textit{i.e.}, $g(x_{\ast})\preceq 0$. Furthermore, the dual variables are non-negative $\lambda_{i}(t)\succeq 0$, and thus the inner product $\langle g(x_{\ast}),\lambda_{i}(t) \rangle\leq 0$ is non-positive. Hence, we can remove the second inner product term from the left hand side of the preceding inequality 
\begin{align}
\nonumber
&\sum_{t=0}^{T-1}\sum_{i=1}^{n} 2\alpha(t)\Big(f_{i}(x_{i}(t))-f_{i}(x_{\ast})\Big)\\ \label{Eq:empirical_risk_form} &+\sum_{t=0}^{T-1}\sum_{i=1}^{n}2\alpha(t)\langle g(x_{i}(t)),\lambda_{i}(t) \rangle \leq \text{r.h.s. of eq. \eqref{Eq:after_convex}}.
 \end{align}\normalsize
We now proceed by computing a bound on the second term on the l.h.s. of the inequality \eqref{Eq:empirical_risk_form}. To this end, we state a lemma: 
\begin{lemma}
\label{Lemma:aux_post_1}
For all $T\in \mathbb{N}$, the following inequality holds
\begin{align}
\nonumber
&\sum_{t=0}^{T-1}\sum_{i=1}^{n}\alpha(t)\eta\|\lambda_{i}(t)\|^{2}- \sum_{t=0}^{T-1}\sum_{i=1}^{n}\alpha^{2}(t)\|\nabla_{\lambda}L_{i}(x_{i}(t),\lambda_{i}(t))\|^{2}\\ \label{Eq:Lower_Bound_on_Lemma}
&\leq \sum_{t=0}^{T-1}\sum_{i=1}^{n}2\alpha(t)\langle g(x_{i}(t)),\lambda_{i}(t) \rangle.
\end{align} 
\end{lemma}
\begin{proof}
See Appendix \ref{App:aux_Proof_of_Lemma}.
\end{proof}

Using the lower bound \eqref{Eq:Lower_Bound_on_Lemma} of Lemma \ref{Lemma:aux_post_1} in conjunction with eq. \eqref{Eq:empirical_risk_form} results in the following inequality
\small \begin{align}
 \label{Eq:empirical_risk_form_1}
&\sum_{t=0}^{T-1}\sum_{i=1}^{n} 2\alpha(t)\Big(f_{i}(x_{i}(t))-f_{i}(x_{\ast})\Big)\\ \nonumber
&\leq n\|x_{\ast}\|^{2}-\sum_{t=0}^{T-1}\sum_{i=1}^{n}\alpha(t)\eta\|\lambda_{i}(t)\|^{2}\\ \nonumber
&+\sum_{t=0}^{T-1}\sum_{i=1}^{n}\alpha^{2}(t)\left(\|\nabla_{x}L_{i}(x_{i}(t),\lambda_{i}(t))\|^{2}+\|\nabla_{\lambda}L_{i}(x_{i}(t),\lambda_{i}(t))\|^{2}\right).
 \end{align}\normalsize 

Now, recall the definition of the cumulative cost function $f(\cdot)$ in \eqref{Eq:emperical_1}. Our goal in the rest of the proof is to establish an inequality in terms of the cumulative cost function. To this end, we first use the convexity of $f_{j}(\cdot)$ (cf. Assumption \ref{Assumption:Lipschitz Functions}) to derive the following inequality
\begin{align}
\label{Eq:App+sub}
f_{i}(x_{j}(t))+\langle \nabla f_{i}(x_{j}(t)),x_{i}(t)-x_{j}(t) \rangle \leq f_{i}(x_{i}(t)).
 \end{align}\normalsize
We then substitute \eqref{Eq:App+sub} into eq. \eqref{Eq:empirical_risk_form_1} to obtain
\begin{align}
\label{Eq:proceed_2}
&\sum_{t=0}^{T-1}\sum_{i=1}^{n}2\alpha(t)\Big((f_{i}(x_{j}(t))-f_{i}(x_{\ast}))\Big)\\ \nonumber &+\sum_{t=0}^{T-1}\sum_{i=1}^{n}2\alpha(t)\langle  \nabla f_{i}(x_{j}(t)),x_{i}(t)-x_{j}(t) \rangle\Big)\\ \nonumber &\leq \text{r.h.s. of eq. \eqref{Eq:empirical_risk_form_1}}.
 \end{align}\normalsize 
From the definition of the cumulative cost function in eq. \eqref{Eq:emperical_1}, we have
\begin{align}
\label{Eq:replacing_for}
f(x_{j}(t))=\dfrac{1}{n}\sum_{i=1}^{n}f_{i}(x_{j}(t)).
 \end{align}\normalsize
A similar equality holds for $f(x_{\ast})$. We now proceed from eq. \eqref{Eq:proceed_2} by using the equality \eqref{Eq:replacing_for},
\small \begin{align}
\nonumber
&2n\sum_{t=0}^{T-1}\alpha(t)(f(x_{j}(t))-f(x_{\ast})) \leq n\|x_{\ast}\|^{2}-\sum_{t=0}^{T-1}\sum_{i=1}^{n}\alpha(t)\eta\|\lambda_{i}(t)\|^{2}
\\  \label{Eq:divide_and_use_1}
&+2 \sum_{t=0}^{T-1}\sum_{i=1}^{n}\alpha(t)\langle \nabla f_{i}(x_{j}(t)),x_{j}(t)-x_{i}(t) \rangle\\
\nonumber
&+\sum_{t=0}^{T-1}\sum_{i=1}^{n}\alpha^{2}(t)\left(\|\nabla_{x}L_{i}(x_{i}(t),\lambda_{i}(t))\|_{\ast}^{2}+\|\nabla_{\lambda}L_{i}(x_{i}(t),\lambda_{i}(t))\|^{2}\right).
 \end{align}\normalsize
From Step 5 of Algorithm \ref{CHalgorithm}, recall the definition of the running average
\begin{align}
\label{Eq:Construction_of_x_hat}
\widehat{x}_{j}(T)=\dfrac{\sum_{t=0}^{T-1}\alpha(t) x_{j}(t)}{\sum_{t=0}^{T-1}\alpha(t)}.
\end{align}
Since the functions $f_{i}(\cdot),i\in V$ is convex (cf. Assumption \ref{Assumption:Lipschitz Functions}), so is $f(\cdot)$ (see \eqref{Eq:emperical_1}). Therefore,
\begin{align}
\label{Eq:divide_and_use_2}
f(\widehat{x}_{j}(T))&\leq \dfrac{\sum_{t=0}^{T-1}\alpha(t)f(x_{j}(t))}{\sum_{t=0}^{T-1}\alpha(t)}.
 \end{align}\normalsize
We divide both sides of eq. \eqref{Eq:divide_and_use_1} by ${1\over (2n\sum_{t=0}^{T-1}\alpha(t))}$ and use the inequality \eqref{Eq:divide_and_use_2} to derive
\footnotesize \begin{align}
\nonumber
&f(\widehat{x}_{j}(T))-f(x_{\ast})\leq \dfrac{1}{\sum_{t=0}^{T-1}\alpha(t)}\Bigg[\dfrac{1}{2}{\|x_{\ast}\|^{2}}-\dfrac{\eta}{2n}\sum_{t=0}^{T-1}\sum_{i=1}^{n}\alpha(t)\|\lambda_{i}(t)\|^{2}\\ \label{Eq:Todays_real}
&+\dfrac{1}{n} \sum_{t=0}^{T-1}\sum_{i=1}^{n}\alpha(t)\langle \nabla f_{i}(x_{j}(t)),x_{j}(t)-x_{i}(t) \rangle\\ \nonumber
&+\dfrac{1}{n}\sum_{t=0}^{T-1}\sum_{i=1}^{n}\alpha^{2}(t)\left(\|\nabla_{x}L_{i}(x_{i}(t),\lambda_{i}(t))\|^{2}+\|\nabla_{\lambda}L_{i}(x_{i}(t),\lambda_{i}(t))\|^{2}\right)\Bigg].
\end{align}\normalsize
From the Cauchy-Schwarz inequality and the bound $\|\nabla f_{i}(x_{j}(t))\|\leq L$ (cf. Assumption \ref{Assumption:Lipschitz Functions}), we derive
\begin{align}
\nonumber
&\langle \nabla f_{i}(x_{j}(t)),x_{j}(t)-x_{i}(t) \rangle \\ \nonumber &\hspace{18mm}  \leq \| \nabla f_{i}(x_{j}(t))\|\cdot \|x_{j}(t)-x_{i}(t)\|\\ \label{Eq:Replacing_the_bound_tt}
&\hspace{18mm}\leq L\|x_{j}(t)-x_{i}(t)\|.
\end{align}
Substituting the upper bound  \eqref{Eq:Replacing_the_bound_tt} in \eqref{Eq:Todays_real}, and changing the roles of the variables $i$ and $j$ completes the proof. 

\subsection{Proof of Theorem \ref{Thm:2}}
\label{App:TheRestofTheProof}

To prove Theorem \ref{Thm:2}, we need the following two lemmas:
\begin{lemma}
\label{Lemma:GradientsBounds}
Consider the stepsize $\alpha(t)$ and the regularization parameter $\eta$ that satisfy $\alpha(t)\eta\leq 1$ for all $t\in [T]$. Then, the sub-gradients of the Lagrangian function are bounded by
\small \begin{align}
\label{Eq:Bound1}
\|\nabla_{\lambda}L_{i}(x_{i}(t),\lambda_{i}(t))\|^{2}&\leq 2mL^{2}R^{2}+2\eta^{2}\|\lambda_{i}(t)\|^{2},\\
\label{Eq:Bound2}
\|\nabla_{x}L_{i}(x_{i}(t),\lambda_{i}(t))\|^{2}&\leq L^{2}\left(1+\dfrac{nm^{2}LR}{\eta}\right)^{2},
\end{align}\normalsize
for all $t\in [T]$. 
\end{lemma}
\begin{proof}
See Appendix \ref{App:Proof}.
\end{proof}

\begin{lemma}
\label{Lemma:Consensus}
Consider the stepsize $\alpha(t)={\alpha\over {\sqrt{1+t}}}$ for some constant $\alpha\in \real_{+}$ and the regularization parameter $\eta$ that satisfies $\alpha(t)\eta\leq 1$ for all $t\in [T]$. Then, the consensus term is bounded by
\small \begin{align}
\label{Eq:consensus_bound}
&\sum_{t=0}^{T-1}\alpha(t)\|x_{i}(t)-x_{j}(t)\| \\ \nonumber
&\leq  5 L \left(1+\dfrac{nm^{2}LR}{\eta}\right)\left(\dfrac{\log(T\sqrt{nT})}{1-\sigma_{2}(W)}\right)^{3\over 2}\sum_{t=0}^{T-1}\alpha^{2}(t),
\end{align}\normalsize
for all $t\in [T]$, and for all $i,j\in V$.
\end{lemma}
\begin{proof}
See Appendix \ref{Section:Proof_of_Consensus}.
\end{proof}

Now, consider the upper bound \eqref{Eq:Prop_inequality} in Lemma \ref{Lem:1}. We use the upper bounds \eqref{Eq:Bound1}-\eqref{Eq:Bound2} in Lemmas \ref{Lemma:GradientsBounds} as well as the consensus bound \eqref{Eq:consensus_bound} in Lemma \ref{Lemma:Consensus} to obtain
\small \begin{align}
\label{Eq:Simplified}
f(\widehat{x}_{i}(T))-f(x_{\ast})&\leq \dfrac{1}{\sum_{t=0}^{T-1}\alpha(t)}\Bigg[\dfrac{\|x_{\ast}\|^{2}}{2}+A \sum_{t=0}^{T-1}\alpha^{2}(t)\\ \nonumber &+\dfrac{1}{2n}\sum_{t=0}^{T-1}\sum_{i=1}^{n}\left(2\eta^{2}\alpha^{2}(t)-\eta\alpha(t)\right)\|\lambda_{i}(t)\|^{2}\Bigg],
 \end{align}\normalsize
where after some algebraic calculations, the constant $A$ is computed as,
\small \begin{align}
\label{Eq:CONSTANT_A}
A\coloneqq  mL^{2}R^{2}+8L^{2} \left(1+\dfrac{nm^{2}LR}{\eta}\right)^{2}\left(\dfrac{\log(T\sqrt{nT})}{1-\sigma_{2}(W)}\right)^{3\over 2}.
\end{align}\normalsize
We further restrict the stepsize and the regularization parameter to satisfy $\eta\alpha(t)\leq {1\over 2}$. In this case, we can drop the last term in the upper bound \eqref{Eq:Simplified} since it is non-positive. Hence,
\small \begin{align}
\label{Eq:SUBSUITUIE}
f(\widehat{x}_{i}(T))-f(x_{\ast})\leq \dfrac{1}{\sum_{t=0}^{T-1}\alpha(t)}\Bigg[\dfrac{R^{2}}{2}+A \sum_{t=0}^{T-1}\alpha^{2}(t)\Bigg],
\end{align}\normalsize
where we also used the fact that $\|x_{\ast}\|^{2}\leq R^{2}$. For the choice of the stepsize $\alpha(t)={R\over \sqrt{t+1}}$, the following lower and upper bounds hold
\small \begin{align}
\label{Eq:Stepsize_1_upper_bound}
&\sum_{t=0}^{T-1}\alpha(t)=\sum_{t=0}^{T-1}{R\over \sqrt{t+1}}\geq 2R(\sqrt{T}-1)\\ \label{Eq:Stepsize_2_upper_bound}
&\sum_{t=0}^{T-1}\alpha^{2}(t)=\sum_{t=0}^{T-1}{R^{2}\over t+1}\leq R^{2}(1+\log(T))\leq 5R^{2}\log(T),
\end{align}\normalsize
for all $T\geq 2$, respectively. Substituting the preceding bounds in eq. \eqref{Eq:SUBSUITUIE} gives us
\small \begin{align*}
f(\widehat{x}_{i}(T))-f(x_{\ast})&\leq \dfrac{R}{\sqrt{T}-1}\left[\dfrac{1}{4}+\dfrac{5}{2}A\log(T) \right], \quad T\geq 2.
\end{align*}\normalsize
Since $\log(T)\geq 1/4$ for $T\geq 2$, we can obtain the following expression 
\begin{align*}
f(\widehat{x}_{i}(T))-f(x_{\ast})&\leq \dfrac{R\log(T)}{\sqrt{T}-1}\left[1+\dfrac{5}{2}A \right], \quad T\geq 2.
\end{align*}
Defining $C\coloneqq 1+(5A/2)$ completes the proof. 
\subsection{Proof of Lemma \ref{Lemma:GradientsBounds}}
\label{App:Proof}
From Algorithm \ref{CHalgorithm}, recall the update rule for the Lagrangian multipliers $\lambda_{i}(t+1)$. Taking the norm of the vector $\lambda_{i}(t+1)$ results in
\small \begin{align}
\nonumber
&\|\lambda_{i}(t+1)\|= \left\|\Pi_{\real^{m}_{+}}\left(\sum_{j=1}^{n}[W]_{ij}\gamma_{j}(t)\right)\right\|\\ \label{Eq:Non-negative_Orthant}
&=\left\|\Pi_{\real_{+}^{m}}\left(\sum_{j=1}^{n}[W]_{ij}(\lambda_{j}(t)+\alpha(t)\nabla_{\lambda}L_{j}(x_{j}(t),\lambda_{j}(t))) \right)\right\|,
\end{align}\normalsize
where in the last equality, we substituted for $\gamma_{j}(t)$ from eq. \eqref{Eq:aux_2} in Step 1 of Algorithm \ref{CHalgorithm}. We use the non-expansive property of the projection to compute the following inequality from eq. \eqref{Eq:Non-negative_Orthant}, 
\small \begin{align*}
\|\lambda_{i}({t+1})\|
&\leq  \left\|\sum_{j=1}^{n}[W]_{ij}(\lambda_{j}(t)+\alpha(t)\nabla_{\lambda}L_{j}(x_{j}(t),\lambda_{j}(t)))\right\|\\
&\leq  \sum_{j=1}^{n}[W]_{ij}\left\|\lambda_{j}(t)+\alpha(t)\nabla_{\lambda}L_{j}(x_{j}(t),\lambda_{j}(t))\right\|,
\end{align*}\normalsize
where the last step follows by using the triangular inequality. We now square both sides of the preceding inequality to obtain
\small \begin{align}
\nonumber
&\|\lambda_{i}(t+1)\|^{2}\leq \left(\sum_{j=1}^{n}[W]_{ij}\left\|\lambda_{i}(t)+\alpha(t)\nabla_{\lambda}L_{i}(x_{i}(t),\lambda_{i}(t))\right\|\right)^{2}\\ \label{Eq:Oa}
&\hspace{10mm} \stackrel{\rm(a)}{\leq} \sum_{j=1}^{n}[W]_{ij}\left\|\lambda_{j}(t)+\alpha(t)\nabla_{\lambda}L_{j}(x_{j}(t),\lambda_{j}(t))\right\|^{2},
\end{align}\normalsize
where $\rm{(a)}$ follows from Jensen's inequality. Recall the definition of the sub-gradient $\nabla_{\lambda}L_{j}(x_{j}(t),\lambda_{j}(t))$ from eq. \eqref{Eq:sub_gradient_L2}. Substituting $\nabla_{\lambda}L_{j}(x_{j}(t),\lambda_{j}(t))$ in eq. \eqref{Eq:Oa} yields
\small \begin{align}
\nonumber
\|\lambda_{i}(t+1)\|^{2}&\leq \sum_{j=1}^{n}[W]_{ij}\left \|(1-\eta\alpha(t))\lambda_{j}(t)+\alpha(t) g(x_{j}(t))\right \|^{2}\\ 
\label{Eq:balance}
&\stackrel{\rm(b)}{\leq} \sum_{j=1}^{n}[W]_{ij}\big((1+\delta(t))(1-\eta\alpha(t))^{2}\|\lambda_{j}(t)\|^{2}\\ \nonumber &\hspace{16mm} +(1+\delta^{-1}(t))\alpha^{2}(t) \| g(x_{j}(t))\|^{2} \big),
\end{align}\normalsize
where in $\rm{(b)}$, we used the Fenchel-Young inequality which holds for any $\delta(t)>0$. Here, $\delta(t)$ is a degree of freedom that allows us to tighten the upper bound in eq. \eqref{Eq:balance} by balancing the two terms inside the parenthesis.

Taking the summation with respect to $i=1,2,\cdots,n$ results in
\small \begin{align}
\label{Eq:Note01}
\sum_{i=1}^{n}\|\lambda_{i}(t+1)\|^{2}\leq &(1+\delta(t))(1-\eta\alpha(t))^{2}\sum_{i=1}^{n}\|\lambda_{i}(t)\|^{2}\\ \nonumber & +(1+\delta^{-1}(t))nm\alpha^{2}(t)L^{2}R^{2},
 \end{align}\normalsize
where we use the fact  $\|g(x_{j}(t))\|^{2}\leq mL^{2}R^{2}$ due to Lipschitz continuity of Assumption \ref{Assumption:Lipschitz Functions}, and $\sum_{i=1}^{n}[W]_{ij}=1$ since the weight matrix $W$ is doubly stochastic by Assumption \ref{Assumption:Weight_matrix}. 

Now, suppose that the regularization parameter satisfies $\eta\alpha(t)\leq 1$ for all $t\in [T]$. In this case, we choose $\delta(t)=\dfrac{\varepsilon(t)}{(1-\alpha(t)\eta)^{2}}-1$, where $\varepsilon(t)\in ((1-\alpha(t)\eta)^{2},1)$. Substituting for $\delta(t)$ in eq. \eqref{Eq:Note01}, we obtain
\small \begin{align*}
\sum_{i=1}^{n}\|\lambda_{i}(t+1)\|^{2}\leq &\varepsilon(t)\sum_{i=1}^{n}\|\lambda_{i}(t)\|^{2}\\ &+\dfrac{\varepsilon(t)\alpha^{2}(t)}{\varepsilon(t)-(1-\alpha(t)\eta)^{2}}nmL^{2}R^{2}.
\end{align*}\normalsize
From this recursion and the fact that $\lambda_{i}(0)=0$ for all $i\in [n]$ in Algorithm \ref{CHalgorithm}, we obtain
\small \begin{align}
\nonumber
\sum_{i=1}^{n}\|\lambda_{i}(t+1)\|^{2}&\leq nmL^{2}R^{2}\sum_{\ell=0}^{t} \dfrac{\alpha^{2}(\ell)\prod_{k=\ell}^{t}\varepsilon(k)}{\varepsilon(\ell)-(1-\alpha(\ell)\eta)^{2}}\\ \nonumber
&\stackrel{\rm{(c)}}{\leq} \dfrac{nmL^{2}R^{2}}{\eta^{2}}\sum_{\ell=0}^{t}\alpha(\ell)\eta \prod_{k=\ell+1}^{t}(1-\alpha(\ell)\eta)\\ \label{eq:upper_bound}
&\stackrel{\rm{(d)}}{\leq} \dfrac{nmL^{2}R^{2}}{\eta^{2}},
 \end{align}\normalsize
where ${\rm{(c)}}$ follows by setting $\varepsilon(k)=(1-\alpha(k)\eta)$ for all $k=0,1,\cdots,t$, and ${\rm(d)}$ is due to the following inequality 
\small \begin{align}
\label{Eq:an_inequality_2}
\sum_{\ell=0}^{t}\alpha(\ell)\eta \prod_{k=\ell+1}^{t}(1-\alpha(k)\eta)\leq 1,
 \end{align}\normalsize
which holds when $\alpha(t)\eta\leq 1,\forall t\in [T]$. We defer the proof of the inequality \eqref{Eq:an_inequality_2} to Appendix \ref{App:aux_ineq}. 

Using the upper bound in eq. \eqref{eq:upper_bound}, we derive
\small \begin{align}
\nonumber
&\|\nabla_{x} L_{i}(x_{i}(t),\lambda_{i}(t))\|=  \left\|\nabla f_{i}(x_{i}(t))+\sum_{k=1}^{m}\lambda_{i,k}(t)\nabla g_{k}(x_{i}(t))\right\| \\ \nonumber
&\stackrel{\rm{(f)}}{\leq} \|\nabla f_{i}(x_{i}(t))\|+\sum_{k=1}^{m}\lambda_{i,k}(t)\|\nabla g_{k}(x_{i}(t))\| \\
&\stackrel{\rm{(g)}}{=} L\left(1+\|\lambda_{i}(t)\|_{1}\right)\\ \label{Eq:Wealsonotethat}
&\stackrel{\rm{(h)}}{\leq } L\left(1+\dfrac{nm^{3/2}LR}{\eta}\right),
 \end{align}\normalsize
where $\rm{(f)}$ follows from the triangle inequality, $\rm{(g)}$ follows by using the upper bounds on the sub-gradients $\|\nabla f_{j}(x_{i}(t))\|\leq L$ and $\|\nabla g_{k}(x_{i}(t))\| \leq L$ from Assumption \ref{Assumption:Lipschitz Functions}, and $\rm{(h)}$ follows by using the upper bound \eqref{eq:upper_bound} in conjunction with the following inequality between the $\ell_{1}$- and $\ell_{2}$-norms,
\begin{align}
\label{Eq:inequality_between_l1_and_l2}
\|\lambda_{i}(t)\|_{1}\leq \sqrt{m} \|\lambda_{i}(t)\|_{2}.
\end{align}

We prove the upper bound on $\|\nabla_{\lambda}L_{i}(x_{i}(t),\lambda_{i}(t))\|^{2}$ in Lemma \ref{Lemma:Consensus} as follows
\small \begin{align*}
\|\nabla_{\lambda}L_{i}(x_{i}(t),\lambda_{i}(t))\|^{2}&=\|g(x_{i}(t))-\eta \lambda_{i}(t)\|^{2}\\
&\leq (\|g(x_{i}(t))\|+\eta\|\lambda_{i}(t)\|)^{2}\\
&\stackrel{\rm{(f)}}{\leq} 2\|g(x_{i}(t))\|^{2}+2\eta^{2}\|\lambda_{i}(t)\|^{2}\\
&\stackrel{\rm{(g)}}{\leq} 2mL^{2}R^{2}+2\eta^{2}\|\lambda_{i}(t)\|^{2},
\end{align*}\normalsize
where ${\rm(f)}$ follows from the inequality $(a+b)^{2}\leq 2a^{2}+2b^{2}$, and ${\rm (g)}$ follows by Lipschitz continuity of each $g_{k}(\cdot)$ for $k\in [m]$ as well as the compactness of the feasible region $\mathcal{X}$. This completes the proof of Lemma \ref{Lemma:Consensus}.

In the sequel, we prove an alternative upper bound on $\|\nabla_{\lambda}L_{i}(x_{i}(t),\lambda_{i}(t))\|^{2}$ for later use in Appendix \ref{Proof of Theorem 3}. In particular, using eq. \eqref{eq:upper_bound} we compute
\small \begin{align}
\nonumber 
\|\nabla_{\lambda}L_{i}(x_{i}(t),\lambda_{i}(t))\|&= \|g(x_{i}(t))-\eta \lambda_{i}(t)\|\\ \nonumber
&\stackrel{\rm{(e)}}{\leq} \|g(x_{i}(t))\|+\eta\|\lambda_{i}(t)\|\\ \nonumber
&{\leq} \sqrt{m}LR+\eta \cdot \dfrac{\sqrt{nm}LR}{\eta}  \\ \label{Eq:BoundRef1}
&\leq  2LR\sqrt{nm},
 \end{align}\normalsize 
where ${\rm{(e)}}$ is due to the triangle inequality. Therefore,
\small \begin{align}
\label{Eq:Wealsonotethat_0}
\|\nabla_{\lambda}L_{i}(x_{i}(t),\lambda_{i}(t))\|^{2}\leq 4L^{2}R^{2}nm.
 \end{align}\normalsize

\subsection{Proof of Lemma \ref{Lemma:Consensus}}
\label{Section:Proof_of_Consensus}

The general plan to prove Lemma \ref{Lemma:Consensus} is to compute a recursion for the consensus term $\|x_{i}(t)-x_{j}(t)\|$. To compute such a recursion, we begin from the update rule in Algorithm \ref{CHalgorithm}, we have 
\small \begin{align}
\nonumber
&\|x_{i}(t)-x_{j}(t)\|\\ \nonumber &\hspace{-2mm} =\left\|\Pi_{\ball_{d}(R)}\left(\sum_{\ell=1}^{n}[W]_{i\ell}y_{\ell}(t-1)\right)-\Pi_{\ball_{d}(R)}\left(\sum_{\ell=1}^{n}[W]_{j\ell}y_{\ell}(t-1) \right)\right\| \\ \nonumber
&\stackrel{\rm{(a)}}{\leq} \left\|\sum_{\ell=1}^{n}([W]_{i\ell}-[W]_{j\ell})y_{\ell}(t-1)\right\| \\ \label{Eq:Subs}
&\stackrel{\rm(b)}{\leq} \sum_{\ell=1}^{n}|[W]_{i\ell}-[W]_{j\ell}|\cdot\left\|y_{\ell}(t-1)\right\|,
\end{align} \normalsize
where ${\rm{(a)}}$ follows by using the non-expansive property of the projection (cf. \cite[Chapter III.3]{hiriart1996convex}), and $\rm{(b)}$ follows by the triangle inequality. Based on the update rule \eqref{Eq:aux_1} in Algorithm \ref{CHalgorithm} and the triangle inequality, we derive
\small \begin{align}
\nonumber
&\|y_{\ell}(t-1)\|=\|x_{\ell}(t-1)+\alpha(t-1) \nabla_{x}L_{\ell}(x_{\ell}(t-1),\lambda_{\ell}(t-1))\|\\  \label{Eq:I_Keep_Coming}
&\leq \|x_{\ell}(t-1)\|+\alpha(t-1)\|\nabla_{x}L_{\ell}(x_{\ell}(t-1),\lambda_{\ell}(t-1))\|.
\end{align}\normalsize
Substituting the upper bound \eqref{Eq:I_Keep_Coming} in \eqref{Eq:Subs} gives us 
\small \begin{align}
\label{Eq:Ineq_COMB01}
&\|x_{i}(t)-x_{j}(t)\| \leq \sum_{\ell=1}^{n}|[W]_{i\ell}-[W]_{j\ell}|\cdot\|x_{\ell}(t-1)\| \\ \nonumber &+\alpha(t-1) \sum_{i=1}^{n}|[W]_{i\ell}-[W]_{j\ell}|\cdot\|\nabla_{x}L_{\ell}(x_{\ell}(t-1),\lambda_{\ell}(t-1))\|,
 \end{align}\normalsize
Further, for the first term in the r.h.s. of eq. \eqref{Eq:Ineq_COMB01}, we derive
\begin{align}
\nonumber
\|x_{\ell}(t-1)\|&=\left\|\Pi_{\ball_{d}(R)}\left(\sum_{m=1}^{n}[W]_{\ell m}y_{m}(t-2)\right)\right\| \\ \label{Eq:Higher_Gross}
&\stackrel{\rm{(d)}}{\leq }\left\|\sum_{m=1}^{n}[W]_{\ell m}y_{m}(t-2)\right\| ,
\end{align}
where ${\rm{(d)}}$ follows by the non-expansive property of the projection.  Using the triangle inequality in conjunction with the inequality \eqref{Eq:Higher_Gross}, we obtain the following inequality
\begin{align}
\nonumber
&\|x_{\ell}(t-1)\| \leq \sum_{m=1}^{n}[W]_{\ell m}\cdot\|y_{m}(t-2)\|\\ \label{Eq:Ineq_COMB02}
&\stackrel{\rm{(e)}}{\leq}  \sum_{m=1}^{n}[W]_{\ell m}\cdot\|x_{m}(t-2)\|\\ \nonumber  &+\alpha(t-1)\sum_{m=1}^{n}[W]_{\ell m}\cdot\|\nabla_{x} L_{m}(x_{m}(t-2),\lambda_{m}(t-2))\|,
 \end{align}\normalsize
where ${\rm{(e)}}$ follows by the triangle inequality.

Plugging \eqref{Eq:Ineq_COMB02} in eq. \eqref{Eq:Ineq_COMB01} yields
\small \begin{align}
\label{Eq:Pursue} 
&\|x_{i}(t)-x_{j}(t)\|\\ \nonumber &\leq \sum_{\ell=1}^{n}|[W^{2}]_{i\ell}-[W^{2}]_{j\ell}|\cdot\|x_{\ell}(t-2)\|\\ \nonumber &+\alpha(t-2)\sum_{\ell=1}^{n}|[W^{2}]_{i\ell}-[W^{2}]_{j\ell}|\cdot\|\nabla_{x}L_{\ell}(x_{\ell}(t-2),\lambda_{\ell}(t-2))\|\\ \nonumber &+\alpha(t-1) \sum_{\ell=1}^{n}|[W]_{i\ell}-[W]_{j\ell}|\cdot\|\nabla_{x}L_{\ell}(x_{\ell}(t-1),\lambda_{\ell}(t-1))\|.
 \end{align}\normalsize
Define the state transition matrix $\Phi(t-1,r)\coloneqq W^{t-r}$. Pursuing the recursive analysis of \eqref{Eq:Pursue} and using the state transition matrix $\Phi(t-1,r)$, yields a more compact form of inequality,
\small \begin{align}
\label{Eq:ExpressionCOMB}
&\|x_{i}(t)-x_{j}(t)\| \leq \sum_{r=0}^{t-1}\alpha(r)\\ \nonumber &\times \sum_{\ell=1}^{n}\left|[\Phi(t-1,r)]_{i\ell}-[\Phi(t-1,r)]_{j\ell}\right|\cdot \|\nabla_{x}L_{\ell}(x_{\ell}(r),\lambda_{\ell}(r))\|,
 \end{align}\normalsize
where in deriving eq. \eqref{Eq:ExpressionCOMB} we used the initial condition $\|x_{i}(0)\|=0,\forall i\in [n]$. We use the upper bound \eqref{Eq:Wealsonotethat} to bound the norm $\|\nabla_{x}L_{\ell}(x_{\ell}(r),\lambda_{\ell}(r))\|$ in eq. \eqref{Eq:ExpressionCOMB},
\small\begin{align}
\label{Eq:ExpressionCOMB1}
&\|x_{i}(t)-x_{j}(t)\|\leq  L\left(1+\dfrac{nm^{3/2}LR}{\eta}\right)\\ \nonumber &\times\sum_{r=0}^{t-1}\alpha(r)\sum_{\ell=1}^{n}\left|[\Phi(t-1,r)]_{k\ell}-[\Phi(t-1,r)]_{i\ell}\right|.
 \end{align}\normalsize
Now, from the definition of $\ell_{1}$-norm we have 
\begin{align*}
\sum_{\ell=1}^{n}|[\Phi(t-1,r)]_{i\ell}&-[\Phi(t-1,r)]_{j\ell}|\\ &=\|[\Phi(t-1,r)]_{i}-[\Phi(t-1,r)]_{j}\|_{1}.
\end{align*}
Therefore, eq. \eqref{Eq:ExpressionCOMB1} can be rewritten more compactly as follows
\small\begin{align}
\|x_{i}(t)-x_{j}(t)\|\leq & L\left(1+\dfrac{nm^{3/2}LR}{\eta}\right)\\ \nonumber &\times\sum_{r=0}^{t-1}\alpha(r)\|[\Phi(t-1,r)]_{i}-[\Phi(t-1,r)]_{j}\|_{1}.
 \end{align}\normalsize
Multiply and divide by $\alpha(t)$ to derive
\begin{align}
\nonumber
&\|x_{i}(t)-x_{j}(t)\| \leq L\left(1+\dfrac{nm^{3/2}LR}{\eta}\right)\\ \label{Eq:earlier_inequaltiy_build} &\times \alpha(t)\left[\sum_{r=0}^{t-1}\dfrac{\alpha(r)}{\alpha(t)}\|[\Phi(t-1,r)]_{i}-[\Phi(t-1,r)]_{j}\|_{1}\right]. 
 \end{align}\normalsize
In the sequel, we bound the term inside the bracket on the right hand side of eq. \eqref{Eq:earlier_inequaltiy_build}. To do so, we need the following lemma:
\begin{lemma} \textsc{(Duchi, \textit{et al.} \cite{duchi2012dual})}
\label{Lemma:duchi}
For all doubly stochastic matrices $W$, and $[\Phi(t-1,r)]=[W^{t-r}]$, the following bound holds
\begin{align*}
\|[\Phi(t-1,r)]_{i}-[\Phi(t-1,r)]_{j}\|_{2}
&\leq 2\sigma_{2}(W)^{t-r}.
\end{align*}
\end{lemma}
Now, consider the stepsize $\alpha(t)={\alpha \over \sqrt{t+1}}$ for some constant $\alpha\in \real_{+}$. In this case, the stepsize ratio is ${\alpha(r)\over \alpha(t)}={\sqrt{{t+1\over r+1}}}\leq \sqrt{T}$, where the inequality follows from the fact that $t\in \{0,1,\cdots,T-1\}$ and $r\in \{0,1,\cdots,t-1\}$. Based on Lemma \ref{Lemma:duchi}, we compute the following upper bound
\begin{align}
\nonumber
\dfrac{\alpha(r)}{\alpha(t)}&\|[\Phi(t-1,r)]_{i}-[\Phi(t-1,r)]_{j}\|_{1}\\ &\leq 
 \sqrt{Tn}\|[\Phi(t-1,r)]_{i}-[\Phi(t-1,r)]_{j}\|_{2}\\ \label{Eq:white}
 &\leq \sqrt{Tn}\cdot (\sigma_{2}(W))^{t-r},
\end{align}
where in the first inequality, we used the fact that $\|x\|_{1}\leq \sqrt{n}\|x\|_{2}$ for any vector $x\in \real^{n}$.
Based on the upper bound we derived in eq. \eqref{Eq:white}, we observe that the following inequality holds
\begin{align}
\label{Eq:UpperBoundForFirst}
\dfrac{\alpha(r)}{\alpha(t)}\|[\Phi(t-1,r)]_{i}-[\Phi(t-1,r)]_{j}\|_{1}\leq \dfrac{1}{T},
\end{align}
if we have that
\begin{align}
\label{Eq:definitionOfThreshold}
t-r\geq \dfrac{\log(T\sqrt{nT})}{\log(\sigma_{2}(W))^{-1}}\coloneqq \tau.
\end{align}\normalsize
In the case that $t-r<\tau$, we simply use the following trivial bound
\begin{align}
\nonumber
\dfrac{\alpha(r)}{\alpha(t)}\|[\Phi(t-1,r)]_{i}-[\Phi(t-1,r)]_{j}\|_{1}&\leq 2\dfrac{\alpha(r)}{\alpha(t)}\\ \label{Eq:UpperBoundForSecond}
&=2\sqrt{\dfrac{t+1}{r+1}},
\end{align}
where the inequality follows by using the triangle inequality 
\begin{align*}
\|[\Phi(t-1,r)]_{i}&-[\Phi(t-1,r)]_{j}\|_{1}\\ 
&\leq \|[\Phi(t-1,r)]_{i}\|_{1}+[\Phi(t-1,r)]_{j}\|_{1},
\end{align*}
and the fact that $[\Phi(t-1,r)]$ is a doubly stochastic matrix which implies $\|[\Phi(t-1,r)]_{i}\|_{1}=1$ for all $i\in [n]$. 

Now, to bound the sum inside the bracket of eq. \eqref{Eq:earlier_inequaltiy_build}, we break the sum into two terms using $\tau$ as the threshold. We obtain that
\begin{align}
\nonumber
\sum_{r=0}^{t-1}\dfrac{\alpha(r)}{\alpha(t)}&\|[\Phi(t-1,r)]_{i}-[\Phi(t-1,r)]_{j}\|_{2}\\ \nonumber &=\sum_{r=0}^{t-\tau}\dfrac{\alpha(r)}{\alpha(t)}\|[\Phi(t-1,r)]_{i}-[\Phi(t-1,r)]_{j}\|_{2}\\ \nonumber
&+\sum_{r=t-\tau+1}^{t-1}\dfrac{\alpha(r)}{\alpha(t)}\|[\Phi(t-1,r)]_{i}-[\Phi(t-1,r)]_{j}\|_{2}\\ \label{Eq:Last_Term_Of_Ineq}
&\stackrel{{\rm{(g)}}}{\leq} \dfrac{t-\tau+1}{T}+2 \sum_{r=t-\tau+1}^{t-1} \sqrt{\dfrac{t+1}{r+1}},
\end{align}
where in ${\rm{(g)}}$, we used the inequality \eqref{Eq:UpperBoundForFirst} for the first sum and  \eqref{Eq:UpperBoundForSecond} for the second.

To compute an upper bound for the last term of the inequality \eqref{Eq:Last_Term_Of_Ineq}, we state the following lemma:
\begin{lemma}
\label{Lemma:Inequality}
For a given $\tau\in \mathbb{N}$, and for all $t\geq \tau-1$, the following inequality holds
\begin{align}
\label{Eq:Ineq_Lef_tau_12}
\sum_{r=t-\tau+1}^{t-1} \sqrt{{t+1\over r+1}}\leq \tau^{3/2}.
\end{align}
\end{lemma}
\begin{proof}
See Appendix \ref{App:Proof_of_Sum_tau}.
\end{proof}

We use the inequality \eqref{Eq:Ineq_Lef_tau_12} to upper bound \eqref{Eq:Last_Term_Of_Ineq} as below
\begin{align}
\nonumber
\sum_{r=0}^{t-1}\dfrac{\alpha(r)}{\alpha(t)}&\|[\Phi(t-1,r)]_{i}-[\Phi(t-1,r)]_{j}\|_{2}\\ \nonumber &\leq \dfrac{t-\tau+1}{T}+2\tau^{3/2}\\ \nonumber
&\stackrel{{\rm{(h)}}}{\leq}  1+2\left(\dfrac{\log(T\sqrt{nT})}{\log(\sigma_{2}(W))^{-1}}\right)^{3\over 2}\\ \nonumber
&\stackrel{{\rm{(i)}}}{\leq}   1+2\left(\dfrac{\log(T\sqrt{nT})}{1-\sigma_{2}(W)}\right)^{3\over 2}\\ \label{Eq:ReplaceIn}
&\stackrel{{\rm{(j)}}}{\leq}   5\left(\dfrac{\log(T\sqrt{nT})}{1-\sigma_{2}(W)}\right)^{3\over 2},
\end{align}\normalsize
where in ${\rm{(h)}}$ we used the fact that $t\in \{0,1,\cdots,T-1\}$ and thus $(t-\tau+1)/T\leq 1$, and subsuited the value of $\tau$ from eq. \eqref{Eq:definitionOfThreshold}. In addition, in ${\rm{(i)}}$ we used the fact that $\log(x)^{-1}\geq 1-x$, and ${\rm{(j)}}$ follows by the fact that $3\left({{\log(T\sqrt{nT})\over 1-\sigma_{2}(W)}}\right)^{3\over 2}\geq 1$ for all $T\geq 2$ and for all $n\in \mathbb{N}$.

We substitute eq. \eqref{Eq:ReplaceIn} in eq. \eqref{Eq:earlier_inequaltiy_build} which gives us
\small \begin{align}
\label{Eq:earlier_inequaltiy_build1}
\|x_{i}(t)-x_{j}(t)\|
\leq 5 L \left(1+\dfrac{nm^{3/2}LR}{\eta}\right)\left(\dfrac{\log(T\sqrt{nT})}{1-\sigma_{2}(W)}\right)^{3\over 2} \alpha(t).
\end{align}\normalsize
Lastly, we use the bound in eq. \eqref{Eq:earlier_inequaltiy_build1} to bound the consensus term 
\small \begin{align}
&\sum_{t=0}^{T-1}\alpha(t)\|x_{i}(t)-x_{j}(t)\|\\ \nonumber
&\leq  5 L \left(1+\dfrac{nm^{3/2}LR}{\eta}\right)\left(\dfrac{\log(T\sqrt{nT})}{1-\sigma_{2}(W)}\right)^{3\over 2} \sum_{t=0}^{T-1}\alpha^{2}(t).
\end{align}\normalsize

\subsection{Proof of Theorem \ref{Thm:3}}
\label{Proof of Theorem 3}

We start from our earlier result in eq. \eqref{Eq:after_convex} in Appendix \ref{Proof of Proposition 1}, which we repeat here,
\begin{align}
\nonumber
&\dfrac{1}{n} \sum_{t=0}^{T-1}\sum_{i=1}^{n}2\alpha(t)\Big(L_{i}(x_{i}(t),\lambda_{i}(t))-L_{i}(x_{\ast},\lambda_{i}(t))\Big)\\
\label{Eq:after_convex_earlier-B1}
&\leq \|x_{\ast}\|^{2}+\dfrac{1}{n}\sum_{t=0}^{T-1}\sum_{i=1}^{n}\alpha^{2}(t)\|\nabla_{x}L_{i}(x_{i}(t),\lambda_{i}(t))\|^{2}.
 \end{align}\normalsize
Moreover, from eq. \eqref{Eq:dual_ref_point_rewrite} in Appendix \ref{App:aux_Proof_of_Lemma}, we have the following inequality
\small \begin{align}
\nonumber
&\dfrac{1}{n}\sum_{t=0}^{T-1}\sum_{i=1}^{n}2\alpha(t)\left(L_{i}(x_{i}(t),\lambda)-L_{i}(x_{i}(t),\lambda_{i}(t))\right)\\ \label{Eq:dual_ref_point_rewrite_B2}  &\leq  \|\lambda\|^{2}+\dfrac{1}{n}\sum_{t=0}^{T-1}\sum_{i=1}^{n}\alpha^{2}(t)\|\nabla_{\lambda}L_{i}(x_{i}(t),\lambda_{i}(t))\|^{2},
 \end{align}\normalsize
for all $\lambda\in \real^{m}_{+}$.  Combining the inequalities in eqs. \eqref{Eq:after_convex_earlier-B1} and \eqref{Eq:dual_ref_point_rewrite_B2} yields
\small \begin{align}
\nonumber
&\dfrac{1}{n}\sum_{t=0}^{T-1}\sum_{i=1}^{n}2\alpha(t)\Big(L_{i}(x_{i}(t),\lambda)-L_{i}(x_{\ast},\lambda_{i}(t))\Big)\leq \|x_{\ast}\|^{2}+\|\lambda\|^{2}\\ \nonumber
&+\dfrac{1}{n}\sum_{t=0}^{T-1}\sum_{i=1}^{n}\alpha^{2}(t)\|\nabla_{x}L_{i}(x_{i}(t),\lambda_{i}(t))\|^{2}\\ \label{Eq:ConstraintViolationBound} &+ \dfrac{1}{n}\sum_{t=0}^{T-1}\sum_{i=1}^{n}\alpha^{2}(t)\|\nabla_{\lambda}L_{i}(x_{i}(t),\lambda_{i}(t))\|^{2}.
 \end{align}\normalsize
Recall the definition of the Lagrangian function from \eqref{Eq:Lagrangian_function},
\begin{align*}
L_{i}(x,\lambda)=f_{i}(x)+\langle \lambda,g(x) \rangle-\dfrac{\eta}{2}\|\lambda \|^{2}.
\end{align*}
We expand the Lagrangian functions on the left hand side of eq. \eqref{Eq:ConstraintViolationBound},
\small \begin{align}
\nonumber
&\dfrac{1}{n}\sum_{t=0}^{T-1}\sum_{i=1}^{n}2\alpha(t)\left(f_{i}(x_{i}(t))-f_{i}(x_{\ast})\right)\\ \nonumber &+\dfrac{1}{n}\sum_{t=0}^{T-1}\sum_{i=1}^{n}2\alpha(t)\left(\langle \lambda,g(x_{i}(t))\rangle-\langle \lambda_{i}(t),g(x_{\ast})\rangle\right)\\ \label{Eq:63} &+\dfrac{1}{n}\sum_{t=0}^{T-1}\sum_{i=1}^{n}\eta \alpha(t)(\|\lambda_{i}(t)\|^{2}-\|\lambda\|^{2})  \leq \text{r.h.s. of \eqref{Eq:ConstraintViolationBound}}. 
 \end{align}\normalsize
Now, we notice that $-\langle  \lambda_{i}(t),g(x_{\ast})\rangle\geq 0$ since $\lambda_{i}(t)\succeq 0$ and $g(x_{\ast})\preceq 0$ for an optimal point $x_{\ast}\in \mathcal{X}$ of the problem \eqref{Eq:emperical_1}-\eqref{Eq:emperical_2}. Furthermore, $\|\lambda_{i}(t)\|^{2}\geq 0$ is non-negative. By eliminating these two non-negative terms from the left hand side of \eqref{Eq:63}, we obtain that
\small \begin{align}
\nonumber
&\dfrac{1}{n}\sum_{t=0}^{T-1}\sum_{i=1}^{n}2\alpha(t)\left(f_{i}(x_{i}(t))-f_{i}(x_{\ast})\right)\\ \nonumber
&+\dfrac{1}{n}\sum_{t=0}^{T-1}\sum_{i=1}^{n}2\alpha(t)\langle \lambda,g(x_{i}(t))\rangle-\eta \|\lambda\|^{2}\sum_{t=0}^{T-1} \alpha(t)\\   &\leq \|x_{\ast}\|^{2}+\|\lambda\|^{2}+\dfrac{1}{n}\sum_{t=0}^{T-1}\sum_{i=1}^{n}\alpha^{2}(t)\|\nabla_{x}L_{i}(x_{i}(t),\lambda_{i}(t))\|^{2}\\ \nonumber
&\hspace{23mm} +\dfrac{1}{n}\sum_{t=0}^{T-1}\sum_{i=1}^{n}\alpha^{2}(t)\|\nabla_{\lambda}L_{i}(x_{i}(t),\lambda_{i}(t))\|^{2}.
 \end{align}\normalsize
We now move the quadratic term $\|\lambda\|^{2}$ from the right hand side to the left hand side of the inequality,
\small \begin{align}
\nonumber
&\dfrac{1}{n}\sum_{t=0}^{T-1}\sum_{i=1}^{n}2\alpha(t)\left(f_{i}(x_{i}(t))-f_{i}(x_{\ast})]\right)\\ \nonumber &+\left[\dfrac{1}{n}\sum_{t=0}^{T-1}\sum_{i=1}^{n}2\alpha(t)\langle \lambda,g(x_{i}(t))\rangle-\|\lambda\|^{2}\Big(1+\eta\sum_{t=0}^{T-1} \alpha(t)\Big)\right]\\ \nonumber &  \leq \|x_{\ast}\|^{2}
+\dfrac{1}{n}\sum_{t=0}^{T-1}\sum_{i=1}^{n}\alpha^{2}(t)\|\nabla_{x}L_{i}(x_{i}(t),\lambda_{i}(t))\|^{2}\\  \label{Eq:overbook}
&\hspace{12mm}+\dfrac{1}{n}\sum_{t=0}^{T-1}\sum_{i=1}^{n}\alpha^{2}(t)\|\nabla_{\lambda}L_{i}(x_{i}(t),\lambda_{i}(t))\|.
 \end{align}\normalsize
We now divide both sides of the preceding inequality by ${1\over 2\sum_{t=0}^{T-1}\alpha(t)}$. Due to the convexity condition (Assumption \ref{Assumption:Lipschitz Functions}) of $f_{i}(\cdot),\forall i\in [n]$ and $g_{k}(\cdot),\forall k\in [m]$ and the definition $\widehat{x}_{i}(T)$ in eq. \eqref{Eq:Construction_of_x_hat}, we compute
\small \begin{align}
\nonumber
&\dfrac{1}{n}\sum_{i=1}^{n}\left(f_{i}(\widehat{x}_{i}(T))-f_{i}(x_{\ast})\right)\\ \label{Eq:Inequalit} &+\left[\dfrac{1}{n}\sum_{i=1}^{n}\langle \lambda,g(\widehat{x}_{i}(T))\rangle-\|\lambda\|^{2}\Big(\dfrac{1}{2\sum_{t=0}^{T-1}\alpha(t)}+\dfrac{\eta}{2}\Big)\right]\\ \nonumber &  \leq \dfrac{1}{\sum_{t=0}^{T-1}\alpha(t)}\Big(\|x_{\ast}\|^{2}
+\dfrac{1}{n}\sum_{t=0}^{T-1}\sum_{i=1}^{n}\alpha^{2}(t)\|\nabla_{x}L_{i}(x_{i}(t),\lambda_{i}(t))\|^{2}\\ \nonumber
&\hspace{31mm} +\dfrac{1}{n}\sum_{t=0}^{T-1}\sum_{i=1}^{n}\alpha^{2}(t)\|\nabla_{\lambda}L_{i}(x_{i}(t),\lambda_{i}(t))\|^{2}\Big).
 \end{align}\normalsize
Since the vector $\lambda=(\lambda_{1},\cdots,\lambda_{m})\in \real^{m}_{+}$ is arbitrary, we can maximize the terms inside the bracket in the l.h.s. with respect to each element $\lambda_{k},k\in [m]$,
\small \begin{align}
\label{Eq:Tired}
&\max_{\lambda\in \real_{+}^{m}}\left[\dfrac{1}{n}\sum_{i=1}^{n}\langle \lambda,g(\widehat{x}_{i}(T))\rangle-\|\lambda\|^{2}\Big(\dfrac{1}{2\sum_{t=0}^{T-1}\alpha(t)}+\dfrac{\eta}{2}\Big)\right]\\ \nonumber &=\dfrac{1}{\left({2\eta}+{2\over \sum_{t=0}^{T-1}\alpha(t)}\right)} \sum_{k=1}^{m}\left[\dfrac{1}{n}\sum_{i=1}^{n}g_{k}(\widehat{x}_{i}(T)) \right]_{+}^{2}.
\end{align}\normalsize
Substituting eq. \eqref{Eq:Tired} into eq. \eqref{Eq:Inequalit} gives us
\small \begin{align}
\label{Eq:Inequality}
&\dfrac{1}{n}\sum_{i=1}^{n}\left(f_{i}(\widehat{x}_{i}(T))-f_{i}(x_{\ast})\right)\\ \nonumber &+\dfrac{1}{\left({2\eta}+{2\over \sum_{t=0}^{T-1}\alpha(t)}\right)} \sum_{k=1}^{m}\left[\dfrac{1}{n}\sum_{i=1}^{n}g_{k}(\widehat{x}_{i}(T)) \right]_{+}^{2} \leq \text{r.h.s. of \eqref{Eq:Inequalit}}. 
\end{align}\normalsize
We now use the upper bounds in eqs. \eqref{Eq:Wealsonotethat},\eqref{Eq:Wealsonotethat_0} to upper bound the sub-gradients on the right hand side, and use the fact that $\|x_{\ast}\|^{2}\leq R^{2}$ since $x_{\ast}\in \ball_{d}(R)$. After some calculations, we derive that
\small \begin{align}
\nonumber
&\dfrac{1}{n}\sum_{i=1}^{n}\left(f_{i}(\widehat{x}_{i}(T))-f_{i}(x_{\ast})\right)\\ \nonumber &+\dfrac{1}{\left({2\eta}+{2\over \sum_{t=0}^{T-1}\alpha(t)}\right)} \sum_{k=1}^{m}\left[\dfrac{1}{n}\sum_{i=1}^{n}g_{k}(\widehat{x}_{i}(T))\right]_{+}^{2}\\ \label{Eq:SUB0} &\leq  \dfrac{1}{\sum_{t=0}^{T-1}\alpha(t)} \left(R^{2}
+A\sum_{t=0}^{T-1}\alpha^{2}(t)\right),
 \end{align}\normalsize
where we recall the definition of the constant $A$ from eq. \eqref{Eq:CONSTANT_A}.

From the description of Algorithm \ref{CHalgorithm}, we note that $\widehat{x}_{i}(T)\in \ball_{d}(R)$. Since $\ball_{d}(R)$ contains the feasible set, \textit{i.e.}, $\mathcal{X}\subseteq \ball_{d}(R)$ two scenarios may occur:

\begin{itemize}
\item[(\textit{i})] ${1\over n}\sum_{i=1}^{n}\left(f_{i}(\widehat{x}_{i}(T))-f_{i}(x_{\ast})\right)\geq 0$: In this case, we simply eliminate this term from the left hand side of eq. \eqref{Eq:SUB0}.

\item[(\textit{ii})] ${1\over n}\sum_{i=1}^{n}\left(f_{i}(\widehat{x}_{i}(T))-f_{i}(x_{\ast})\right)\leq 0$: This case can occur since the output $\widehat{x}_{i}(T)\in \ball_{d}(R)$ of Algorithm \ref{CHalgorithm} belongs to a larger set compared to an optimal solution $x_{\ast}\in \mathcal{X}\subseteq \ball_{d}(R)$. Therefore, the value of the function at $\widehat{x}_{i}(T)$ can be smaller than that of an optimal solution.
\end{itemize}

To take both cases into account, we define the function $\mathcal{F}(T)$ as below
\begin{align}
\label{Eq:SUB1}
\mathcal{F}(T)\coloneqq -\dfrac{1}{n}\sum_{i=1}^{n}\left(f_{i}(\widehat{x}_{i}(T))-f_{i}(x_{\ast})\right).
\end{align}  
Clearly, the absolute value of the function in \eqref{Eq:SUB1} is finite, \textit{i.e.}, $|\mathcal{F}(T)|<\infty$ for all $T=0,1,2,\cdots$ since each local function $f_{j}(\cdot)$ is defined on a compact set $\ball_{d}(R)$ and it is Lipschitz continuous. Let $\mathcal{F}_{+}(T)\coloneqq \max \{0,\mathcal{F}(T)\}$. In Case (\textit{i}), we have $\mathcal{F}(T)\leq 0$ and thus $\mathcal{F}_{+}(T)$ vanishes. Further, in Case (\textit{ii}), $\mathcal{F}_{+}(T)=\mathcal{F}(T)$. Using the definition of $\mathcal{F}_{+}(T)$ in conjunction with eq. \eqref{Eq:SUB0} gives us
\small \begin{align}
\nonumber
&\dfrac{1}{\left({2\eta}+{2\over \sum_{t=0}^{T-1}\alpha(t)}\right)} \sum_{k=1}^{m}\left[\dfrac{1}{n}\sum_{i=1}^{n}g_{k}(\widehat{x}_{i}(T))\right]_{+}^{2}\\ &\leq  \mathcal{F}_{+}(T)+\dfrac{1}{\sum_{t=0}^{T-1}\alpha(t)} \left(R^{2}
+A\sum_{t=0}^{T-1}\alpha^{2}(t)\right),
\end{align}\normalsize
whence we derive
\small \begin{align}
\nonumber
&\left\|\left[\dfrac{1}{n}\sum_{i=1}^{n}g(\widehat{x}_{i}(T))\right]_{+}\right\|^{2}\leq {2\left(\eta+\dfrac{1}{\sum_{t=0}^{T-1}\alpha(t)}\right)} \mathcal{F}_{+}(T) \\  \label{Eq:sss} &+\dfrac{2\left(\eta\sum_{t=0}^{T-1}\alpha(t)+1\right)}{(\sum_{t=0}^{T-1}\alpha(t))^{2}} \left(2R^{2}
+A\sum_{t=0}^{T-1}\alpha^{2}(t)\right),
\end{align}\normalsize
where $g\coloneqq (g_{1},\cdots,g_{m})$.

To compute an asymptotic bound, note that $\sum_{t=0}^{T-1}\alpha(t)=\Omega(\sqrt{T})$ and $\sum_{t=0}^{T-1}\alpha^{2}(t)=\mathcal{O}(\log(T))$ for the stepsize $\alpha(t)=\alpha/\sqrt{t+1}$. Moreover, $\mathcal{F}_{+}(T)=\mathcal{O}(1)$ as $\mathcal{F}_{+}(T)<\infty$. Therefore, we obtain from \eqref{Eq:sss} that
\small \begin{align}
\left\|\left[\dfrac{1}{n}\sum_{i=1}^{n}g(\widehat{x}_{i}(T))\right]_{+}\right\|^{2}_{2}=\mathcal{O}(\eta).
\end{align}\normalsize

We now prove the second part of Theorem \ref{Thm:3}. To prove \eqref{Eq:SecondConstraintViolationBound}, in the sequel we show that when the inequality constraints are satisfied strictly, we indeed have $\mathcal{F}_{+}(T)=\mathcal{O}(\log(T)/\sqrt{T})$. To prove this result, we recall the following inequality from eq. \eqref{Eq:App+sub},
\small \begin{align}
\hspace{-4mm} f_{i}(x_{j}(t))\rangle \leq f_{i}(x_{i}(t))+\langle \nabla f_{i}(x_{j}(t)),x_{j}(t)-x_{i}(t)\rangle.
 \end{align}\normalsize
Using the Cauchy-Schwarz inequality as well as the Lipschitz bound  $\|\nabla f_{i}(x_{j}(t))\|\leq L$ in Assumption \ref{Assumption:Lipschitz Functions} gives us
 \begin{align*}
\hspace{-4mm} f_{i}(x_{j}(t))\rangle &\leq f_{i}(x_{i}(t))+\langle \nabla f_{i}(x_{j}(t)),x_{j}(t)-x_{i}(t)\rangle \\
&\leq f_{i}(x_{i}(t))+\|\nabla f_{i}(x_{j}(t))\|\|x_{j}(t)-x_{i}(t)\|\\
&\leq f_{i}(x_{i}(t))+L\|x_{j}(t)-x_{i}(t)\|.
 \end{align*}
We multiply both sides of the inequality by $\alpha(t)$ and subsequently take the sum over $t=0,1,\cdots,T-1$,
\begin{align*}
\sum_{t=0}^{T-1}\alpha(t)f_{i}(x_{j}(t))\leq \sum_{t=0}^{T-1}\alpha(t)&f_{i}(x_{i}(t))\\ &+L\sum_{t=0}^{T-1}\alpha(t)\|x_{j}(t)-x_{i}(t)\|.
\end{align*}
Divide both sides of the preceding inequality by ${1\over \sum_{t=0}^{T-1}\alpha(t)}$ and use the definition of $\widehat{x}_{i}(T)$ in eq. \eqref{Eq:Construction_of_x_hat} in conjunction with the inequality \eqref{Eq:divide_and_use_2} to derive
\begin{align*}
f_{i}(\widehat{x}_{j}(T))\leq f_{i}(\widehat{x}_{i}(T))+\dfrac{L}{\sum_{t=0}^{T-1}\alpha(t)} \sum_{t=0}^{T-1}\alpha(t)\|x_{j}(t)-x_{i}(t)\|.
\end{align*}
Equivalently, by subtracting $f_{i}(x_{\ast})$ from both sides of the inequality and then reversing the sign, we obtain
\begin{align}
\label{Eq:Now_Got_it}
-(f_{i}(\widehat{x}_{i}(T))-f_{i}(x_{\ast}&)) \leq -(f_{i}(\widehat{x}_{j}(T)-f_{i}(x_{\ast}))\\ \nonumber &+\dfrac{L}{\sum_{t=0}^{T-1}\alpha(t)} \sum_{t=0}^{T-1}\alpha(t)\|x_{j}(t)-x_{i}(t)\|.
\end{align}
From the inequality \eqref{Eq:consensus_bound} of Lemma \ref{Lemma:Consensus} with the stepsize $\alpha(t)={\alpha\over \sqrt{t+1}}$, the consensus term is bounded by
\begin{align}
\dfrac{L}{\sum_{t=0}^{T-1}\alpha(t)} \sum_{t=0}^{T-1}\alpha(t)\|x_{j}(t)-x_{i}(t)\|=\mathcal{O}\left({\log(T)\over \sqrt{T}}\right).
\end{align}
Therefore, the inequality \eqref{Eq:Now_Got_it} becomes
\small \begin{align}
\label{Eq:Now_Got_it_1}
-(f_{i}(\widehat{x}_{i}(T))-f_{i}(\widehat{x}_{\ast}&)) \leq -(f_{i}(\widehat{x}_{j}(T))-f_{i}(\widehat{x}_{\ast}))+\mathcal{O}\left({\log(T)\over \sqrt{T}}\right).
\end{align}\normalsize
We use \eqref{Eq:Now_Got_it_1} to upper bound $\mathcal{F}(T)$ in eq. \eqref{Eq:SUB1}. In particular
\small \begin{align}
 \label{Eq:Remove}
\mathcal{F}(T)&\leq -\dfrac{1}{n}\sum_{i=1}^{n}(f_{i}(\widehat{x}_{j}(T))-f_{i}(x_{\ast}))+\mathcal{O}\left({\log(T)\over \sqrt{T}}\right). 
\end{align}\normalsize
We now argue that the first term of the upper bound \eqref{Eq:Remove} is non-positive and hence can be dropped. To do so, we prove a lemma first:
\begin{lemma}
\label{Lemma:03}
Recall the definition of the cumulative cost function $f(x)\coloneqq {1\over n}\sum_{i=1}^{n} f_{i}(x)$ and suppose the optimal solution of the minimization problem in \eqref{Eq:empirical_risk_form}-\eqref{Eq:empirical_risk_form_1} satisfies the inequality constraints strictly. That is, $g(x_{\ast})\prec 0$. Then,
\begin{align}
\label{Eq:Main_Inequality_here}
f(x)-f(x_{\ast})\geq 0, 
\end{align}
for all $x\in \ball_{d}(R)$.
\end{lemma}
\begin{proof}
See Appendix \ref{App:Proof_of_Lemma_03}.
\end{proof}
We use the inequality \eqref{Eq:Main_Inequality_here} in Lemma \ref{Lemma:03} with $x=\widehat{x}_{j}(T)$ to derive
\begin{align*}
\dfrac{1}{n}\sum_{i=1}^{n}(f_{i}(\widehat{x}_{j}(T))-f_{i}(x_{\ast}))=f(\widehat{x}_{j}(T))-f(x_{\ast}))\geq 0.
\end{align*}
Consequently, the sum in eq. \eqref{Eq:Remove} is non-positive. We thus remove the sum from  eq. \eqref{Eq:Remove},
\begin{align*}
\mathcal{F}(T)=\mathcal{O}\left({\log(T)\over \sqrt{T}}\right),
\end{align*}
and hence $\mathcal{F}_{+}(T)=\max\{0,\mathcal{F}(T)\}=\mathcal{O}\left({\log(T)/\sqrt{T}}\right)$.

Recalling that $\sum_{t=0}^{T-1}\alpha(t)=\Omega(\sqrt{T})$ and $\sum_{t=0}^{T-1}\alpha^{2}(t)=\mathcal{O}(\log(T))$, we conclude from eq. \eqref{Eq:sss} that
\small \begin{align*}
\left\|\left[\dfrac{1}{n}\sum_{i=1}^{n}g(\widehat{x}_{i}(T))\right]_{+}\right\|_{2}^{2}=\mathcal{O}\left({\eta\log(T)\over \sqrt{T}}\right).
\end{align*}\normalsize

\subsection{Proof of Theorem \ref{Thm:High_Probability_Bound}}
\label{Proof_of_High_Probability}

Analogous to the derivation in eq. \eqref{Eq:Comb01} for the deterministic algorithm, for any realization of random variables $\{x_{i}(t),\lambda_{i}(t),K_{i}(t)\}_{t=1}^{T}$ it can be shown that for an optimal solution $x_{\ast}\in \mathcal{X}$ the following inequality holds,
\small \begin{align}
\label{Eq:Comb011}
&\sum_{t=0}^{T-1}\sum_{i=1}^{n}2\alpha(t)\langle\nabla_{x}\widehat{L}_{i}(x_{i}(t),\lambda_{i}(t);K_{i}(t)),x_{i}(t)-x_{\ast}\rangle\\ \nonumber &\leq  n\|x_{\ast}\|^{2}+\sum_{t=0}^{T-1}\sum_{i=1}^{n}\alpha^{2}(t)\|\nabla_{x}\widehat{L}_{i}(x_{i}(t),\lambda_{i}(t);K_{i}(t))\|^{2}.
\end{align}\normalsize
We define the estimation error $e_{i}(t)$ of the stochastic sub-gradient as follows
\begin{align}
\label{Eq:PuttingTogether}
e_{i}(t)\coloneqq \nabla_{x} L_{i}(x_{i}(t),\lambda_{i}(t))-\nabla_{x}\widehat{L}_{i}(x_{i}(t),\lambda_{i}(t);K_{i}(t)).
\end{align}
By putting together the inequality \eqref{Eq:Comb011} and the definition of the estimation error in eq. \eqref{Eq:PuttingTogether}, we obtain
\begin{align}
\nonumber
&\sum_{t=0}^{T-1}\sum_{i=1}^{n}2\alpha(t)\langle\nabla_{x}L_{i}(x_{i}(t),\lambda_{i}(t)), x_{i}(t)-x_{\ast}\rangle  \\ \label{Eq:137-1} &\leq n\|x_{\ast}\|^{2}+\sum_{t=1}^{T}\sum_{i=1}^{n} \|\nabla_{x}\widehat{L}_{i}(x_{i}(t),\lambda_{i}(t);K_{i}(t))\|^{2}\\ \nonumber &+\sum_{t=1}^{T}\sum_{i=1}^{n}2\alpha(t)\langle e_{i}(t), x_{i}(t)-x_{\ast}\rangle.
\end{align}
Further, using the same approach resulting in eq. \eqref{Eq:dual_ref_point}, we derive
\begin{align}
\label{Eq:138-1}
&\sum_{t=0}^{T-1}\sum_{i=1}^{n}2\alpha(t)\langle\nabla_{\lambda}\widehat{L}_{i}(x_{i}(t),\lambda_{i}(t)),\lambda-\lambda_{i}(t)\rangle \\ \nonumber &\leq  \sum_{i=1}^{n}\|\lambda\|^{2}+\sum_{t=0}^{T-1}\sum_{i=1}^{n}\alpha^{2}(t)\|\nabla_{\lambda}\widehat{L}_{i}(x_{i}(t),\lambda_{i}(t))\|^{2}.
\end{align}
Since the deterministic and stochastic sub-gradient with respect to the dual variables coincide $\nabla_{\lambda}L_{i}(x_{i}(t),\lambda_{i}(t))=\nabla_{\lambda}\widehat{L}_{i}(x_{i}(t),\lambda_{i}(t))$ (cf. \eqref{Eq:Unbiased_Estimator01}), we have
\begin{align}
\label{Eq:138-1}
&\sum_{t=0}^{T-1}\sum_{i=1}^{n}2\alpha(t)\langle\nabla_{\lambda}L_{i}(x_{i}(t),\lambda_{i}(t)),\lambda-\lambda_{i}(t)\rangle \\ \nonumber &\leq  \sum_{i=1}^{n}\|\lambda\|^{2}+\sum_{t=0}^{T-1}\sum_{i=1}^{n}\alpha^{2}(t)\|\nabla_{\lambda}L_{i}(x_{i}(t),\lambda_{i}(t))\|^{2}.
\end{align}

Thus, by following the steps \eqref{Eq:011}-\eqref{Eq:Replacing_the_bound_tt} of the proof of Lemma \ref{Lem:1} in Appendix \ref{Proof of Proposition 1} for the stochastic primal-dual algorithm algorithm, the following inequality can be shown,
\footnotesize \begin{align}
\nonumber
&f(\widehat{x}_{k}(T))-f(x_{\ast})\leq \dfrac{1}{\sum_{t=0}^{T-1}\alpha(t)}\Bigg[\dfrac{1}{2} \|x_{\ast}\|^{2}-\dfrac{\eta}{2n}\sum_{t=0}^{T-1}\sum_{i=1}^{n}\alpha(t)\|\lambda_{i}(t)\|^{2}\\ \nonumber
&+\dfrac{L}{n} \sum_{t=0}^{T-1}\sum_{i=1}^{n} \alpha(t) \|x_{k}(t)-x_{i}(t)\|+\dfrac{1}{n}\sum_{t=0}^{T-1}\sum_{i=1}^{n}\alpha(t)\langle e_{i}(t),x_{i}(t)-x_{\ast} \rangle \\ \nonumber
&+\dfrac{1}{2n}\sum_{t=0}^{T-1}\sum_{i=1}^{n}\alpha^{2}(t)\|\nabla_{x}\widehat{L}_{i}(x_{i}(t),\lambda_{i}(t);K_{i}(t))\|^{2}\\ \label{Eq:Long_Equation_Stochastic_Setting}
&+\dfrac{1}{2n}\sum_{t=0}^{T-1}\sum_{i=1}^{n}\alpha^{2}(t)\|\nabla_{\lambda}L_{i}(x_{i}(t),\lambda_{i}(t))\|^{2}\Bigg].
\end{align}\normalsize
The upper bound on the consensus term is similar to the deterministic primal-dual method, see eq. \eqref{Eq:consensus_bound} in Lemma \ref{Lemma:Consensus}. Moreover, similar to eq. \eqref{Eq:Bound1} in Lemma \ref{Lemma:GradientsBounds}, we have
\begin{align}
\|\nabla_{\lambda}L_{i}(x_{i}(t),\lambda_{i}(t))\|^{2}&\leq 2mL^{2}R^{2}+2\eta^{2}\|\lambda_{i}(t)\|^{2}.
\end{align}
We compute an upper bound for the sub-gradient with respect to the primal variable as follows
\small \begin{align}
\nonumber
\|\nabla_{x}\widehat{L}_{i}(x_{i}(t),&\lambda_{i}(t);K_{i}(t))\|\\  \nonumber
&=\left\|\nabla f_{i}(x_{i}(t))+\|\lambda_{i}(t)\|_{1}\cdot \nabla g_{K_{i}(t)}(x_{i}(t))\right\|\\ \nonumber
&\stackrel{\rm{(a)}}{\leq}  \left\|\nabla f_{i}(x_{i}(t))\right\|+\|\lambda_{i}(t)\|_{1} \| \nabla g_{K_{i}(t)}(x_{i}(t))\|\\ \nonumber
&\stackrel{\rm(b)} \leq L(1+\|\lambda_{i}(t)\|_{1})\\ \nonumber
&\stackrel{\rm(c)} \leq L(1+\sqrt{m}\|\lambda_{i}(t)\|),\\
&\stackrel{\rm(d)} \leq L\left(1+\dfrac{nm^{3/2}LR}{\eta}\right),
\end{align}\normalsize
where ${\rm (a)}$ follows by the triangle inequality, ${\rm (b)}$ follows from upper bound on the sub-gradients $\left\|\nabla f_{i}(x_{i}(t))\right\|$ and $\| \nabla g_{K_{i}(t)}(x_{i}(t))\|$ in Assumption \ref{Assumption:Lipschitz Functions}, ${\rm (c)}$ holds due to the inequality , and ${\rm(d)}$ follows from the upper bound \eqref{eq:upper_bound} on $\|\lambda_{i}(t)\|$ which also holds in the stochastic settings. Therefore,
\begin{align*}
\|\nabla_{x}\widehat{L}_{i}(x_{i}(t),\lambda_{i}(t);K_{i}(t))\|^{2}\leq L^{2}\left(1+\dfrac{nm^{3/2}LR}{\eta}\right)^{2}.
\end{align*}
We now obtain a high probability bound for the estimation error term $\sum_{t=0}^{T-1}\sum_{i=1}^{n}\langle e_{i}(t),x_{i}(t)-x_{\ast} \rangle$ in eq. \eqref{Eq:Long_Equation_Stochastic_Setting}. First, recall from eq. \eqref{Eq:Unbiased_Estimator} that the stochastic sub-gradient is an unbiased estimator of the deterministic sub-gradient, \textit{i.e.},
\small \begin{align}
\label{Eq:Rewrite}
\nabla_{x}L_{i}(x_{i}(t),\lambda_{i}(t))&\coloneqq \expect [\nabla_{x} \widehat{L}_{i}(x_{i}(t),\lambda_{i}(t);K_{i}(t))| \mathfrak{F}_{t}].
\end{align} \normalsize
Based on the definition of the estimation error in eq. \eqref{Eq:PuttingTogether}, we rewrite \eqref{Eq:Rewrite} as below
\begin{align}
\label{Eq:where_the_last}
\expect [e_{i}(t)|\mathfrak{F}_{t}]=0.
\end{align}
Consequently, since $x_{i}(t)$ is $\mathfrak{F}_{t}$-measurable and by the iterative law of the expectation, we can write
\small \begin{align}
\nonumber
\expect\big[\langle e_{i}(t), x_{i}(t)-x_{\ast}\rangle\big]&=\expect\big[\langle e_{i}(t), x_{i}(t)-x_{\ast} \rangle\big]\\  \nonumber
&=\expect\big[\expect\big[\langle e_{i}(t), x_{i}(t)-x_{\ast} \rangle\big] \big| \mathfrak{F}_{t}\big]\\ \nonumber
&=\expect\big[\langle \expect\big[ e_{i}(t) \big| \mathfrak{F}_{t}\big], x_{i}(t)-x_{\ast} \rangle \big]\\  \label{Eq:BoundedDifference}
&=0,
\end{align}\normalsize
where the last equality follows from \eqref{Eq:where_the_last}.

Moreover, using the Cauchy-Schwarz inequality, we can obtain that
\begin{align}
\label{Eq:BoundedDifference1}
\langle e_{i}(t),& x_{i}(t)-x_{\ast} \rangle \\ \nonumber &\leq \|e_{i}(t) \|\cdot \| x_{i}(t)-x_{\ast}\|\leq 2R\|e_{i}(t)\|,
\end{align}
where we used the fact that $\|x_{i}(t)-x_{\ast}\|\leq 2R$ as $x_{i}(t),x_{\ast}\in \ball_{d}(R)$. Moreover, by expanding the norm of the estimation error $\|e_{i}(t)\|$, we derive
\begin{align}
\nonumber
&\|e_{i}(t)\|\coloneqq \|\nabla_{x} L_{i}(x_{i}(t),\lambda_{i}(t))-\nabla_{x}\widehat{L}_{i}(x_{i}(t),\lambda_{i}(t);K_{i}(t))\|\\ \nonumber
&=\Big\|\sum_{k=1}^{m}\lambda_{i,k}(t)\nabla g_{k}(x_{i}(t))-\|\lambda_{i}(t)\|_{1}\nabla g_{K_{i}(t)}(x_{i}(t))\Big\|\\ \nonumber
&\leq \sum_{k=1}^{m}\lambda_{i,k}(t)\big\|\nabla g_{k}(x_{i}(t))\big\|+\|\lambda_{i}(t)\|_{1}\big\|\nabla g_{K_{i}(t)}(x_{i}(t))\big\|\\ \nonumber
&\leq 2L\|\lambda_{i}(t)\|_{1} \\ \label{Eq:From_error}
&\leq 2\sqrt{m}L\|\lambda_{i}(t)\|,
\end{align}
where the last inequality follows from \eqref{Eq:inequality_between_l1_and_l2}. Now, it is easy to see that the upper bound \eqref{eq:upper_bound} on $\|\lambda_{i}(t)\|$ which also holds in the stochastic settings. Therefore, from eq. \eqref{Eq:From_error}, we obtain 
\begin{align}
\label{Eq:Comb_02}
\|e_{i}(t)\|\leq \dfrac{2nm^{3/2}L^{2}R}{\eta}.
\end{align}
Combining eqs. \eqref{Eq:BoundedDifference1} and \eqref{Eq:Comb_02} gives us
\begin{align}
\langle e_{i}(t),x_{i}(t)-x_{\ast} \rangle\leq \dfrac{4nm^{3/2}L^{2}R^{2}}{\eta}.
\end{align}
From Jensen's inequality and the preceding inequality, we obtain that
\begin{align}
\nonumber
&\Big(\dfrac{1}{n}\sum_{i=1}^{n}\alpha(t) \langle e_{i}(t), x_{i}(t)-x_{\ast} \rangle\Big)^{2}\\ &\hspace{20mm} \leq \dfrac{1}{n}\sum_{i=1}^{n} \Big(\alpha(t) \langle e_{i}(t), x_{i}(t)-x_{\ast}\rangle\Big)^{2} \\
&\hspace{20mm} \leq \dfrac{16n^{2}m^{3}L^{4}R^{4}}{\eta^{2}}\alpha^{2}(t).
\end{align}
Applying the Azuma-Hoeffding inequality \cite{chung2006concentration} yields the tail bound,
\begin{align*}
&\Prob\left[\dfrac{1}{n}\sum_{t=0}^{T-1}\sum_{i=1}^{n}\alpha(t) \langle e_{i}(t),x_{i}(t)-x_{\ast}\rangle\geq \delta \right]\\
&\leq \exp\left(-\dfrac{\delta^{2}\eta^{2}}{32n^{2}m^{3}L^{4}R^{4}\sum_{t=0}^{T-1}\alpha^{2}(t)}\right). 
\end{align*}
Hence, with the probability of at least $1-\varepsilon$, the following inequality holds
\small \begin{align}
\nonumber
\dfrac{1}{n}\sum_{t=0}^{T-1}\sum_{i=1}^{n}&\alpha(t)\langle e_{i}(t),x_{i}(t)-x_{\ast}\rangle \\ \label{Eq:COMB011-L} &\leq  \dfrac{4nm^{3/2}L^{2}R^{2}}{\eta} \sqrt{2{\log{1\over \varepsilon} \sum_{t=0}^{T-1}\alpha^{2}(t) }}.
\end{align}\normalsize
Using the stepsize $\alpha(t)={R\over \sqrt{t+1}}$ in conjunction with the inequality \eqref{Eq:Stepsize_2_upper_bound}, in turn gives us the following upper bound
\begin{align}
\nonumber
&\dfrac{1}{n}\sum_{t=0}^{T-1}\sum_{i=1}^{n}\alpha(t)\langle e_{i}(t),x_{i}(t)-x_{\ast}\rangle \\ \label{Eq:High_Probability_Bound_1} &\hspace{20mm} \leq  \dfrac{4nm^{3/2}L^{2}R^{3}}{\eta} \sqrt{10{\log{1\over \varepsilon}}\cdot \log(T) }.
\end{align}
We substitute the preceding inequality in eq. \eqref{Eq:Long_Equation_Stochastic_Setting} to obtain the high probability bound, using the steps of the proof in Appendix \ref{App:TheRestofTheProof}. Choosing $\varepsilon=1/T$ in eq. \eqref{Eq:High_Probability_Bound_1} yields the high probability bound in Theorem \ref{Thm:High_Probability_Bound}. 

To prove the second part of Theorem \ref{Thm:High_Probability_Bound}, we take the expectation of both sides of inequality \eqref{Eq:Long_Equation_Stochastic_Setting} and use the fact that the expectation of the estimation error term is zero due to eq. \eqref{Eq:BoundedDifference}.

\section{Auxiliary Results}
\subsection{Proof of Inequality \eqref{Eq:an_inequality_2}}
\label{App:aux_ineq}
\begin{lemma}
Suppose $\alpha(t)\eta\leq 1$ for all $t\in [T]$. Then,
\begin{align*}
\sum_{\ell=0}^{t}\alpha(\ell)\eta\prod_{k=\ell+1}^{t}(1-\alpha(k)\eta)\leq 1.
\end{align*}
\end{lemma}
\begin{proof}
For ease of notation, let $\theta_{t}\coloneqq \alpha(t)\eta\in (0,1]$. First, suppose $\theta_{k}\not= 1$ for all $k=1,2,\cdots,t$. In this case, the sum of products is monotone increasing in $\theta_{0}$ and by expanding the sum, we derive
\small \begin{align}
\nonumber
\sum_{\ell=0}^{t}\theta_{\ell}\prod_{k=\ell+1}^{t}(1-\theta_{k})
&= \Big(\theta_{0}(1-\theta_{1})(1-\theta_{2})\cdots(1-\theta_{t})\Big)\\ \nonumber &\hspace{4mm} +\Big(\theta_{1}(1-\theta_{2})\cdots(1-\theta_{t})\Big)+\cdots+\theta_{t}\\ \nonumber
\nonumber
&\stackrel{(\rm{e})}{\leq} \Big((1-\theta_{1})(1-\theta_{2})\cdots(1-\theta_{t})\Big)\\ \nonumber &\hspace{4mm} +\Big(\theta_{1}(1-\theta_{2})\cdots(1-\theta_{t})\Big)+\cdots+\theta_{t}\\ 
\nonumber
&\stackrel{(\rm{f})}{=}\Big((1-\theta_{1}+\theta_{1})(1-\theta_{2})\cdots (1-\theta_{t})\Big)\\ \nonumber &+\Big(\theta_{2}(1-\theta_{3})\cdots(1-\theta_{t})\Big)+\cdots+\theta_{t}\\
\nonumber
&\stackrel{(\rm{g})}{=}\big((1-\theta_{2}+\theta_{2})\cdots (1-\theta_{t})\big)+\cdots+\theta_{t}\\ \nonumber
\hspace{5mm}\vdots \\ \label{Eq:expansion_reduction}
&=(1-\theta_{t})+\theta_{t}=1,
 \end{align}\normalsize
where $(\rm{f})$ follows by combining the first and second parentheses in $(\rm{e})$, $(\rm{g})$ follows by combining the first and second parentheses in $(\rm{f})$, and so on.

Now, consider the case where $\theta_{k}=1$ for some indices $k\in \{1,2,\cdots,t\}$, and let $j$ be the largest index among those indices, \textit{i.e.}, $\theta_{j}=1$ and $\theta_{k}<1$ for all $k\in \{j+1,\cdots,t\}$. In this case, 
\begin{align*}
\sum_{\ell=0}^{t}\theta_{\ell}\prod_{k=\ell+1}^{t}(1-\theta_{k})=\sum_{\ell=j}^{t}\theta_{\ell}\prod_{k=\ell+1}^{t}(1-\theta_{k})= 1,
\end{align*}
where the last equality can be proved using the same approach we used to derive \eqref{Eq:expansion_reduction}. 
\end{proof}

\subsection{Proof of Lemma \eqref{Lemma:aux_post_1}}
\label{App:aux_Proof_of_Lemma}

Similar to the proof of Lemma \ref{Lem:1} in Appendix \ref{Proof of Proposition 1}, we derive a recursive formula for the Lagrangian multipliers. For any $\lambda\in \real_{+}$, we have the following recursion for the Lagrangian multipliers in Algorithm \ref{CHalgorithm},
\small \begin{align}
\label{Eq:BeginFrom}
\|\lambda_{i}(t+1)-\lambda\|^{2}&\leq \left\|\Pi_{\real_{+}}\left(\sum_{j=1}^{n}[W]_{ij}\gamma_{j}(t)\right) -\lambda\right\|^{2}\\ \label{Eq:Continue}
&=\left\|\Pi_{\real_{+}}\left(\sum_{j=1}^{n}[W]_{ij}\gamma_{j}(t)\right) -\Pi_{\real_{+}}(\lambda)\right\|^{2},
\end{align}\normalsize
where the last equality is true since $\Pi_{\real_{+}}(\lambda)=\lambda$ for a vector $\lambda\in \real_{+}$. We continue from eq. \eqref{Eq:Continue} as follows
\begin{align}
\nonumber
&\|\lambda_{i}(t+1)-\lambda\|^{2}\\ \nonumber
&\stackrel{(\rm{a})}{\leq} \left\|\sum_{j=1}^{n}[W]_{ij}\gamma_{j}(t)-\lambda\right\|^{2}\\ \nonumber
&\stackrel{(\rm{b})}{=}\left\|\sum_{j=1}^{n}[W]_{ij}(\lambda_{j}(t)-\lambda+\alpha(t)\nabla_{\lambda}L_{j}(x_{j}(t),\lambda_{j}(t)))\right\|^{2}\\ \nonumber
&\stackrel{(\rm{c})}{\leq}\sum_{j=1}^{n}[W]_{ij}\|\lambda_{j}(t)-\lambda+\alpha(t)\nabla_{\lambda}L_{j}(x_{j}(t),\lambda_{j}(t))\|^{2}\\ \nonumber
&=\sum_{j=1}^{n}[W]_{ij}\Big(\|\lambda_{j}(t)-\lambda\|^{2}\\ \nonumber &\hspace{20mm} +2\alpha(t)\langle\nabla_{\lambda}L_{j}(x_{j}(t),\lambda_{j}(t)),\lambda_{j}(t)-\lambda\rangle\\ \label{Eq:dual} &\hspace{20mm}+\alpha^{2}(t)\|\nabla_{\lambda}L_{j}(x_{j}(t),\lambda_{j}(t))\|^{2}\Big),
 \end{align}\normalsize
where $\rm{(a)}$ follows by the non-expansive property of the projection (cf. \cite[Chapter III.3]{hiriart1996convex}), $\rm{(b)}$ follows by replacing $\gamma_{j}(t)=\lambda_{j}(t)+\alpha(t)\nabla_{\lambda}L_{j}(x_{j}(t),\lambda_{j}(t)))$ from eq. \eqref{Eq:aux_2} in Step 2 of Algorithm \ref{CHalgorithm}, $\rm{(c)}$ follows from the convexity of the squared norm. Now, taking the sum over $i=1,2,\cdots,n$ and using an analysis similar to the proof of Lemma \ref{Lem:1}, from eq. \eqref{Eq:dual} we compute
\small \begin{align}
\label{Eq:dual_ref_point}
&\sum_{t=0}^{T-1}\sum_{i=1}^{n}2\alpha(t)\langle\nabla_{\lambda}L_{i}(x_{i}(t),\lambda_{i}(t)),\lambda-\lambda_{i}(t)\rangle \\ \nonumber &\leq  \sum_{i=1}^{n}\|\lambda_{i}(0)-\lambda\|^{2}+\sum_{t=0}^{T-1}\sum_{i=1}^{n}\alpha^{2}(t)\|\nabla_{\lambda}L_{i}(x_{i}(t),\lambda_{i}(t))\|^{2}.
 \end{align}\normalsize
Since the Lagrangian function $L_{i}(x_{i}(t),\cdot)$ is concave, we can write the following inequality
\begin{align*}
L_{i}(x_{i}(t),\lambda)&-L_{i}(x_{i}(t),\lambda_{i}(t))\\ &\leq \langle\nabla_{\lambda}L_{i}(x_{i}(t),\lambda_{i}(t)),\lambda-\lambda_{i}(t)\rangle.
\end{align*}
Therefore, we rewrite eq. \eqref{Eq:dual_ref_point} using the preceding inequality 
\small \begin{align}
\nonumber
&\sum_{t=0}^{T-1}\sum_{i=1}^{n}2\alpha(t)(L_{i}(x_{i}(t),\lambda)-L_{i}(x_{i}(t),\lambda_{i}(t))) \\ \label{Eq:dual_ref_point_rewrite} &\leq  \sum_{i=1}^{n}\|\lambda\|^{2}+\sum_{t=0}^{T-1}\sum_{i=1}^{n}\alpha^{2}(t)\|\nabla_{\lambda}L_{i}(x_{i}(t),\lambda_{i}(t))\|^{2},
 \end{align}\normalsize
where we also use the fact that $\lambda_{i}(0)=0$ for all $i\in [n]$ (see Algorithm \ref{CHalgorithm}). Since $\lambda \in \real_{+}^{m}$ is an arbitrary vector, we set $\lambda=0$ in eq. \eqref{Eq:dual_ref_point_rewrite},
\small \begin{align}
\nonumber
\sum_{t=0}^{T-1}\sum_{i=1}^{n}&\alpha^{2}(t)\|\nabla_{\lambda}L_{i}(x_{i}(t),\lambda_{i}(t))\|^{2}\\ \nonumber &\geq 
\sum_{t=0}^{T-1}\sum_{i=1}^{n}2\alpha(t)L_{i}(x_{i}(t),0)-L_{i}(x_{i}(t),\lambda_{i}(t))\\  \nonumber
&=-\sum_{t=0}^{T-1}\sum_{i=1}^{n}2\alpha(t)\langle g(x_{i}(t))+\sum_{t=0}^{T-1}\sum_{i=1}^{n}\alpha(t)\eta\|\lambda_{i}(t)\|^{2}, 
 \end{align}\normalsize
where the last inequality follows by expanding the Lagrangian functions. 

\subsection{Proof of Lemma \ref{Lemma:Inequality}}
\label{App:Proof_of_Sum_tau}

\begin{lemma}
\label{Lemma:Inequality}
For a given $\tau\in \mathbb{N}$, and for all $t\geq \tau-1$, the following inequality holds
\begin{align}
\label{Eq:Ineq_Lef_tau}
\sum_{r=t-\tau+1}^{t-1} \sqrt{{t+1\over r+1}}\leq \tau^{3/2}.
\end{align}
\end{lemma}
\begin{proof}
Consider the change of variable $r=m+(t-\tau+1)$ for the following sum
\begin{align}
\nonumber
&\sum_{r=t-\tau+1}^{t-1} \sqrt{{1\over r+1}}=\sum_{m=0}^{\tau-2} {1\over \sqrt{m+(t-\tau+2)}}
\\ \nonumber &\stackrel{\rm{(a)}}{\leq} \dfrac{1}{\sqrt{t-\tau+2}}+\int_{0}^{\tau-2}\dfrac{\D x}{\sqrt{x+(t-\tau+2)}} \\ \
\label{Eq:Riemann_solution}
&=\dfrac{1}{\sqrt{t-\tau+2}}+2(\sqrt{t}-\sqrt{t-\tau+2}).
\end{align}
where ${\rm{(a)}}$ follows by the fact that the Riemann sum can be bounded from above by the integral. The second term in the preceding equality can be bounded as follows
\begin{align}
\nonumber
&2(\sqrt{t}-\sqrt{t-\tau+2})\\ \nonumber &= 2\Big(\sqrt{t-\tau+2}\cdot\sqrt{1+{\tau-2\over t-\tau+2} } -\sqrt{t-\tau+2}\Big)\\ \nonumber
&\stackrel{\rm{(b)}}{\leq} 
2\Big(\sqrt{t-\tau+2}\cdot\left(1+{1\over 2}{\tau-2\over t-\tau+2}\right) -\sqrt{t-\tau+2}\Big)\\ \label{Eq:PrecedingInequaltiy01}
&\leq \dfrac{\tau-2}{\sqrt{t-\tau+2}},
\end{align}
where $\rm{(b)}$ follows from the inequality $(1+x)^{r}\leq 1+rx$ for all $r\in [0,1]$ and $x\geq -1$. Using the upper bound \eqref{Eq:PrecedingInequaltiy01} in eq. \eqref{Eq:Riemann_solution} gives us
\begin{align*}
\sum_{r=t-\tau+1}^{t-1} \sqrt{{1\over r+1}}&\leq \dfrac{\tau-1}{\sqrt{t-\tau+2}}\\
&\leq \dfrac{\tau}{\sqrt{t-\tau+2}}.
\end{align*}
We multiply the preceding inequality by $\sqrt{t+1}$, 
\begin{align*}
\sum_{r=t-\tau+1}^{t-1} \sqrt{{t+1\over r+1}}&\leq \tau \sqrt{\dfrac{t+1}{t-\tau+2}}\\
&\stackrel{\rm{(c)}}{\leq} \tau \sqrt{\tau}.
\end{align*}
where ${\rm(c)}$ follows by the fact that $\sqrt{{t+1\over t-\tau+2}}$ is monotone decreasing in $t$, and thus it is maximized by $t=\tau-1$.
\end{proof}

\subsection{Proof of Lemma \ref{Lemma:03}}
\label{App:Proof_of_Lemma_03}

\begin{lemma}
Recall the definition of the cumulative cost function $f(x)\coloneqq {1\over n}\sum_{i=1}^{n} f_{i}(x)$ and suppose the optimal solution of the minimization problem in \eqref{Eq:empirical_risk_form}-\eqref{Eq:empirical_risk_form_1} satisfies the inequality constraints strictly. That is, $g(x_{\ast})\prec 0$. Then,
\begin{align}
\label{Eq:Main_Inequality}
f(x)-f(x_{\ast})\geq 0, 
\end{align}
for all $x\in \ball_{d}(R)$.
\end{lemma}

\begin{proof}
Since $f_{i}(\cdot)$ is convex on the Euclidean ball $\ball_{d}(R)$, so is $f(\cdot)$. Therefore, 
\small \begin{align}
\label{Eq:KKTargument}
f(x)-f(x_{\ast})\geq \langle \xi,x-x_{\ast}\rangle,  
 \end{align}\normalsize
for all $\xi\in \partial f(x_{\ast}),x\in \ball_{d}(R)$. We now write the Karush-Kuhn-Tucker (KKT) conditions \cite{boyd2004convex} for the optimal solution $x_{\ast}$ and the vector of optimal Lagrangian  multipliers $\lambda_{\ast}\coloneqq(\lambda_{\ast,1},\lambda_{\ast,2},\cdots,\lambda_{\ast,m})$,
\begin{itemize}[leftmargin=*]
\item[] [C1] $\langle \xi ,x-x_{\ast}\rangle \geq 0,\quad \forall \xi\in \partial f(x_{\ast})+\sum_{k=1}^{m}\lambda_{*,k}\cdot \partial g_{k}(x_{\ast})$,
\item[] [C2] $\lambda_{\ast,k}\cdot g_{k}(x_{\ast})=0, \quad k=1,2,\cdots,m$, 
\item[] [C3] $g(x_{\ast})\preceq 0$ and $\lambda_{\ast}\succeq 0$.
\end{itemize}
From [C2] we note that $\lambda_{\ast,k}=0$ since $g_{k}(x_{\ast})< 0$ for $k=1,2,\cdots,m$. Consequently, the condition in [C1] turns into
\begin{align*}
\langle \xi ,x-x_{\ast}\rangle \geq 0,\quad \quad \text{for all} \ \xi\in \partial f(x_{\ast}),x\in \ball_{d}(R).
\end{align*} 
\end{proof}
\normalsize

\bibliographystyle{ieeetran}
\bibliography{mybib1}

\end{document}